\documentclass[runningheads,a4paper]{llncs}

\usepackage{array}

\usepackage{algorithm}
\usepackage{algpseudocode}
\usepackage{amssymb}
\usepackage{bm, amsmath} 
\usepackage{graphicx}

\usepackage{url}
\urldef{\mailsa}\path|{alfred.hofmann, ursula.barth, ingrid.haas, frank.holzwarth,|
\urldef{\mailsb}\path|anna.kramer, leonie.kunz, christine.reiss, nicole.sator,|
\urldef{\mailsc}\path|erika.siebert-cole, peter.strasser, lncs}@springer.com|    
\newcommand{\keywords}[1]{\par\addvspace\baselineskip
\noindent\keywordname\enspace\ignorespaces#1}
\usepackage{mathrsfs}

\newcommand{\Ww}{\mathcal{W}}
\newcommand{\Ss}{\mathcal{S}}
\def\K{\ensuremath{\mathbb{Q}}}
\def\Q{\ensuremath{\mathbb{Q}}}
\def\R{\ensuremath{\mathbb{R}}}
\def\C{\ensuremath{\mathbb{C}}}
\def\Kbar {\ensuremath{{\mathbb{C}}}}
\def\S{\ensuremath{\mathfrak{S}}}
\def\scrR{\ensuremath{\mathscr{R}}}

\DeclareMathOperator{\sign}{sign}
\DeclareMathOperator{\thom}{Thom}
\DeclareMathOperator{\der}{Der}
\def\softO{\ensuremath{{O}{\,\tilde{ }\,}}}

\newcommand{\F}{\mathbb{F}}

\newcommand{\ZZ}{\mathrm{Zer}}
\newcommand{\RR}{\mathrm{Reali}}
\newcommand{\D}{\mathrm{D}}
\newcommand{\Z}{\mathbb{Z}}

\newcommand{\HH}{\mathrm{H}}
\newcommand{\card}{\mathrm{card}}

\newcommand {\hide}[1]{}

\newcommand{\X}{\mathbf{X}}
\newcommand{\x}{\mathbf{x}}

\newcommand{\y}{\mathbf{y}}

\newcommand{\length}{\mathrm{length}}
\newcommand{\gen}{\mathrm{gen}}

\newcommand{\W}{\mathcal{W}}

\DeclareBoldMathCommand{\f}{f}
\DeclareBoldMathCommand{\x}{x}
\DeclareBoldMathCommand{\y}{y}
\DeclareBoldMathCommand{\z}{z}
\DeclareBoldMathCommand{\b}{b}
\DeclareBoldMathCommand{\X}{X}
\DeclareBoldMathCommand{\T}{T}
\DeclareBoldMathCommand{\a}{a}
\DeclareBoldMathCommand{\u}{u}
\DeclareBoldMathCommand{\v}{v}
\def\calO{\ensuremath{\mathcal{O}}}

\DeclareMathOperator{\Comp}{Comp}

\DeclareMathOperator{\CompMax}{CompMax}
\def\N{\ensuremath{\mathbb{N}}}

\DeclareBoldMathCommand{\calG}{\mathfrak{G}}
\DeclareBoldMathCommand{\calS}{\mathcal{S}}
\DeclareBoldMathCommand{\calR}{\mathcal{R}}
\def\jac{\ensuremath{{\rm Jac}}}
\def\fraka{{\mathfrak{a}}}

\def\macrodelta{\ell}
\DeclareBoldMathCommand{\g}{g}
\def\calN{{\phi}}
\def\vl{{L}}

\def\calU{{\mathcal{U}}}
\def\calA{{\mathcal{A}}}
\DeclareMathOperator{\RM}{RM}
\newcommand{\comp}{\mathscr{C}}
\newcommand{\Coxeter}{\mathcal{C}}
\def\true{\ensuremath{{\rm \bf true}}}
\def\false{\ensuremath{{\rm \bf false}}}

\usepackage[most]{tcolorbox}
\tcbuselibrary{listingsutf8}

\newtcolorbox{codebox}{
  colback=gray!10,    
  colframe=black!80,  
  listing only,
  listing options={
    basicstyle=\ttfamily\footnotesize,
    breaklines=true,
    language=Maple
  },
  width=1\linewidth, 
  left=5pt,
  right=5pt,
  top=5pt,
  bottom=5pt,
  boxrule=0.5pt,
  arc=3pt
}

\usepackage[utf8]{inputenc}
\usepackage{amsmath, amssymb}
\usepackage{mdframed}

\usepackage{amsmath, amssymb}
\usepackage{xcolor}
\usepackage{mdframed}
\usepackage{lipsum} 

\newcounter{subproblem}[problem]

\renewcommand\theproblem{Problem \arabic{problem}}
\renewcommand\thesubproblem{Problem \arabic{problem}.\arabic{subproblem}}

\newenvironment{pb}[1][]{
  \refstepcounter{problem}%
  \begin{mdframed}[
    backgroundcolor=blue!10,
    linecolor=gray!80,
    linewidth=1pt,
    frametitlefont=\bfseries,
    frametitle={\theproblem. #1}
  ]
}{
  \end{mdframed}
}

\begin{document}

\mainmatter  

\title{Lecture Notes in Computer Science:\\Authors' Instructions
for the Preparation\\of Camera-Ready
Contributions\\to LNCS/LNAI/LNBI Proceedings}

\titlerunning{Lecture Notes in RTCA: Symmetry}

%
%
\author{Cordian Riener%
\thanks{Please note that the LNCS Editorial assumes that all authors have used
the western naming convention, with given names preceding surnames. This determines
the structure of the names in the running heads and the author index.}%
\and Thi Xuan Vu}
%

\institute{Springer-Verlag, Computer Science Editorial,\\
Tiergartenstr. 17, 69121 Heidelberg, Germany\\
\mailsa\\
\mailsb\\
\mailsc\\
\url{http://www.springer.com/lncs}}

%
%


\frontmatter  

\renewcommand*\contentsname{\hfill\makebox[0pt][r]{\normalfont\Large\bfseries Contents}}
\tableofcontents

\mainmatter 
\section*{Symbolic Computation with Symmetric Polynomials in Real Algebraic Geometry }
\addcontentsline{toc}{section}{{\bf Symbolic Computation with Symmetric Polynomials in Real Algebraic Geometry} \\ {Cordian Riener and Thi Xuan Vu}}
\stepcounter{section}

\vspace{1cm}
\begin{flushleft}  
{Cordian Riener}\footnote{Department of Mathematics and Statistics, UiT - the Arctic University of Norway, 9037 Troms\o, Norway, email: cordian.riener@uit.no} \\[1ex] 
{Thi Xuan Vu}\footnote{Univ. Lille, CNRS, Centrale Lille, UMR 9189 CRIStAL, Lille, France, email: thi-xuan.vu@univ-lille.fr} \\[1ex] 
\end{flushleft}

\begin{abstract}
Symmetry plays a central role in accelerating symbolic computation involving polynomials. This chapter surveys recent developments and foundational methods that leverage the inherent symmetries of polynomial systems to reduce complexity, improve algorithmic efficiency, and reveal deeper structural insights. The main focus is on symmetry by the permutation of variables. 
\keywords{invariant theory, symmetric polynomials, computational complexity, efficient algorithms.}
\end{abstract}

\subsection{Introduction}
\begin{quote}
{\it Symmetry is a vast subject, significant in art and nature. Mathematics lies at its root,
and it would be hard to find a better one on which to demonstrate the working of the
mathematical intellect. }\vspace{-0.5cm}
\begin{flushright}Hermann Weyl\end{flushright}
\end{quote}
Real algebraic geometry asks for real solutions of algebraic (in-)equalities. Since many computational problems arising in different contexts, including engineering, finance, and computer science, can be formulated in the language of polynomial equations and inequalities the question of realness of solutions is very natural. Generally, symbolic methods tend to be challenging and  solving real algebraic  problems is thus also  known to be algorithmically hard in general, whenever the dimension is growing. It is therefore  beneficial to explore the algebraic and geometric structures underlying a given problem to design more efficient algorithms, and a kind of structure which is omnipresent in algebra and geometry is symmetry. In the language of algebra, symmetry is the invariance of an object or a property by some action of a group. The goal of the present chapter is to present  techniques which allow us to reduce the complexity of a real algebraic computation problem with symmetry. Since it is impossible to give an exhaustive and detailed description of this vast domain, our goal here is to focus mainly on recent developments in a very special situation, namely the situation when the real algebraic problem in question is invariant by any permutation of the variables. In this particular case strong algebraic properties of so called {\em symmetric polynomials} allow for a quite strong reduction of complexity with respect to the dimension of the problem.

\subsubsection{Overview.}%
This chapter begins with a concise overview of its structure and contents. The remainder of the introduction provides essential preliminaries. In Section~\ref{sec:basic}, we recall fundamental notions from real algebraic geometry. Section~\ref{sec:comp} introduces the principal computational problems that arise in this setting. In Sections~\ref{sec:critical}, \ref{sec:road}, and~\ref{sec:sos}, we present key algorithmic tools: the \emph{critical point method}, the construction of \emph{roadmaps}, and \emph{sums of squares} decompositions, respectively.

The subsequent Section~\ref{seq:comm} discusses computational models and outlines the basic algebraic tasks and assumptions regarding the computational framework. Given that the primary focus of this chapter lies in the exploitation of symmetry, particularly with respect to the action of the symmetric group, we collect relevant background on symmetric polynomials in Section~\ref{sec:sym}.

Section~\ref{sec_prin} demonstrates how, in the case of fixed degree, certain structural properties of symmetric polynomials lead to substantial reductions in complexity. We first show that sums of squares decompositions admit \emph{constant complexity} (with respect to the number of variables) in this regime (Section~\ref{sec:symsos}). This is followed by a presentation of the \emph{degree principle} in Section~\ref{sec:degree}, which forms the basis for the development of more efficient algorithms described in Section~\ref{sec:degreealgo}.

In Section~\ref{sub:empty}, we discuss recent progress that enables efficient to deciding the emptiness of real algebraic sets even when the degree is not fixed. Finally, Section~\ref{sec:topology} addresses the topological aspects of symmetric semi-algebraic sets and illustrates how symmetry can be leveraged to design more efficient algorithms for computing topological invariants.

\subsubsection{Basic Notions of Real Algebraic Geometry.}\label{sec:basic} We first recall some basic notations of real algebraic geometry. 
\begin{definition}[Algebraic Sets]  
A subset \( V \subset \C^n \) is called a \(\K\)-algebraic set (or \(\K\)-algebraic variety) if there exists a set of polynomials \(\f = (f_1, \dots, f_s) \subset \K[x_1, \dots, x_n]\) such that \( V \) is the zero set in \(\C^n\) of the system \(\f = 0\); that is,
\[
V = V(\f) = \{ \a \in \C^n : \f(\a) = 0 \}.
\]
\end{definition}

If the algebraic set \( V \) is the zero locus of a single polynomial, it is called a \emph{hypersurface}. When this polynomial is linear, \( V \) is called a \emph{hyperplane}. Algebraic sets satisfy the following properties (see, e.g., \cite[Chapter I]{Hartshorne2013}): the empty set and the whole space \(\C^n\) are algebraic sets; the intersection of any collection of algebraic sets is an algebraic set; and the union of any finite collection of algebraic sets is an algebraic set. These properties imply that algebraic varieties behave as the closed sets of a topology on \(\C^n\).

\begin{definition}
An algebraic set is called a \emph{Zariski closed set}. The complement of a Zariski closed set is called a \emph{Zariski open set}. The \emph{Zariski topology} on \(\Kbar^n\) is the topology whose closed sets are the algebraic varieties. The \emph{Zariski closure} of a set \(V \subset \Kbar^n\) is the smallest algebraic set containing \(V\). A subset \(W \subset V\) of a variety \(V\) is \emph{Zariski dense} in \(V\) if its closure is \(V\).
\end{definition}

\begin{definition}[Real Algebraic Sets]  
A subset \( W \subset \R^n \) is called a \emph{real algebraic set} if there exists a set of polynomials \(\f = (f_1, \dots, f_s) \subset \R[x_1, \dots, x_n]\) such that
\[
W = \{ \a \in \R^n : \f(\a) = 0 \}.
\]
\end{definition}

\begin{definition}[Semi-algebraic Sets]  
A subset \( S \subset \R^n \) is called a \emph{semi-algebraic set} if it can be constructed from finitely many real algebraic sets using a finite number of unions, intersections, and complements. Equivalently, \( S \) can be described by a finite Boolean combination of polynomial equalities and inequalities:
\[
S = \bigcup_{i=1}^m \bigcap_{j=1}^{r_i} \{ \a \in \R^n : P_{ij}(\a) \ast_{ij} 0 \},
\]
where each \( P_{ij} \in \R[x_1, \dots, x_n] \) and each \(\ast_{ij} \in \{=, >, <, \geq, \leq\}\).
\end{definition}
Semi-algebraic sets satisfy the following properties (see e.g., {\cite[Chapters 2 and 5]{BPR06}}):  the image of a semi-algebraic set under a projection map is also semi-algebraic (Tarski–Seidenberg Theorem) and every semi-algebraic set has finitely many connected components, each of which is semi-algebraic.

Note that although algebraic sets and semi-algebraic sets both arise from polynomial equations, they differ significantly in some structural properties. First, while semi-algebraic sets are closed under finite unions, intersections, and complements (i.e., they form a Boolean algebra), algebraic sets are closed under finite intersections but not necessarily under unions or complements. Second, algebraic sets are not closed under projections; the image of an algebraic set under a polynomial map need not be algebraic. Moreover, real algebraic sets can have infinitely many connected components (with respect to the Euclidean topology), whereas semi-algebraic sets always have finitely many connected components, each of which is itself semi-algebraic. Finally, semi-algebraic geometry admits full quantifier elimination over real closed fields, which is a key feature in the study of these sets.

\begin{example}
Let 
\(
A = \{ x \in \mathbb{R} \mid x > 0 \}\)  and \( B = \{ x \in \mathbb{R} \mid x < 1 \}
\) be two semi-algebraic sets. 
Then, 
\[
A \cap B = (0,1), \quad A \cup B = \mathbb{R} \setminus [0,1], \quad \text{and} \quad \mathbb{R} \setminus A = (-\infty,0]
\]
are all semi-algebraic sets. On the other hand, consider the algebraic set 
\[
W = \{ x \in \mathbb{C} \mid x = 0 \} = V(x).
\]
Its complement is 
\[
\mathbb{C} \setminus W = \{ x \in \mathbb{C} \mid x \neq 0 \},
\]
which is not algebraic since no polynomial vanishes exactly on this set.

\medskip
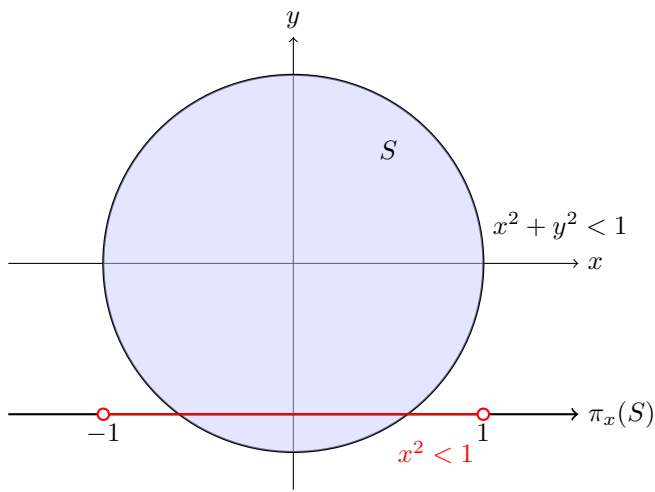
\begin{center}
    
\begin{tikzpicture}[scale=2.5]

    \draw[->] (-1.5,0) -- (1.5,0) node[right] {$x$};
    \draw[->] (0,-1.2) -- (0,1.2) node[above] {$y$};
    
    \draw[thick] (0,0) circle (1);
    
    \fill[blue!20,opacity=0.5] (0,0) circle (1);
    
    \node at (0.5,0.6) {$S$};
    
    \draw[thick, ->] (-1.5,-0.8) -- (1.5,-0.8) node[right] {$\pi_x(S)$};
    
    \draw[thick, red] (-1,-0.8) -- (1,-0.8);
    
    \draw[thick, red, fill=white] (-1,-0.8) circle (0.03);
    \draw[thick, red, fill=white] (1,-0.8) circle (0.03);
    
    \node[below] at (-1,-0.8) {$-1$};
    \node[below] at (1,-0.8) {$1$};

    \node[right] at (1,0.2) {$x^2 + y^2 < 1$};
    
    \node[right] at (0.5,-1) {\textcolor{red}{$x^2 < 1$}};

\end{tikzpicture}

\end{center}

Now consider the semi-algebraic set 
\(
S = \{ (x,y) \in \mathbb{R}^2 \mid x^2 + y^2 < 1 \}.
\)
Projecting onto the \(x\)-axis yields 
\[
\pi_x(S) = \{ x \in \mathbb{R} \mid x^2 < 1 \} = (-1,1),
\]
which is semi-algebraic. Meanwhile, consider the algebraic set in \(\mathbb{C}^2\):
\[
V = \{ (x,y) \in \mathbb{C}^2 \mid xy = 1 \}.
\]
Projecting onto the \(x\)-coordinate yields
\[
\pi_x(V) = \{ x \in \mathbb{C} \mid \exists y \in \mathbb{C}, \, xy = 1 \} = \{ x \in \mathbb{C} \mid x \neq 0 \}.
\]
The set \(\pi_x(V)\) is not algebraic since it is the complement of the algebraic set \(V(x) = \{0\}\).

\smallskip

Finally, the semi-algebraic set
\[
\{ x \in \mathbb{R} \mid x < -1 \text{ or } 0 < x < 1 \text{ or } x > 2 \}
\]
has three connected components, each of which is semi-algebraic.

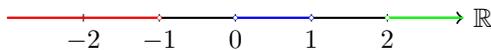
\begin{center}
\begin{tikzpicture}[scale=1]
    \draw[thick, ->] (-3,0) -- (3,0) node[right] {\(\mathbb{R}\)};
    
    \foreach \x in {-2, -1, 0, 1, 2}
        \draw (\x,0.05) -- (\x,-0.05) node[below] {\(\x\)};
        
    \draw[thick, red] (-3,0) -- (-1,0);
    \filldraw[red] (-1,0) circle (0.7pt); 
    \draw[white, fill=white] (-1,0) circle (0.5pt); 
    
    \draw[thick, blue] (0,0) -- (1,0);
    \filldraw[blue] (0,0) circle (0.7pt); 
    \draw[white, fill=white] (0,0) circle (0.5pt);
    \filldraw[blue] (1,0) circle (0.7pt); 
    \draw[white, fill=white] (1,0) circle (0.5pt);
    
    \draw[thick, green] (2,0) -- (3,0);
    \filldraw[green] (2,0) circle (0.7pt); 
    \draw[white, fill=white] (2,0) circle (0.5pt);
\end{tikzpicture}
\end{center}
\end{example}

\subsubsection{Computational Tasks in Real Algebraic Geometry.}\label{sec:comp}

Real algebraic geometry studies \emph{semi-algebraic sets} -- subsets of \(\mathbb{R}^n\) defined by finite Boolean combinations of polynomial equalities and inequalities with real coefficients. Several fundamental algorithmic problems naturally arise.

\begin{pb}[Computational problems]
Given a semi-algebraic set \( S \subset \mathbb{R}^n \), we aim to algorithmically answer the following tasks:
\begin{enumerate}
  \item \textbf{Emptiness Testing:} Check if  \( S \) is not empty and provide any real point in $S$.
  \item \textbf{Connected Components:} Count the  semi-algebraically connected components of  \( S \) and compute (at least)  one point in each component?
    \item \textbf{Topological Invariants:} Compute the  Betti numbers \( b_i(S) \) of $S$  or the Euler–Poincaré characteristic.
  \item \textbf{Connectivity Queries:} Given two points \( \u, \v \in S \), decide if they in the same connected component of $S$.

\end{enumerate}
\end{pb}

These problems are decidable by Tarski’s theorem; however,  the resulting methods such as the Tarski–Seidenberg quantifier elimination procedure suffer from non-elementary complexity. Over the past decades, significant progress has been made, particularly through the development of singly exponential algorithms when the dimension is fixed. Comprehensive treatments of classical results and algorithms for these tasks can be found in the textbook by Basu, Pollack, and Roy \cite{BPR06}, which we recommend to interested readers. Additionally, the survey article by Basu \cite{survey} provides an accessible overview of these topics.

Two fundamental algorithmic tools in real algebraic geometry are the \emph{critical point method} and the construction of \emph{roadmaps}. These techniques enable efficient solutions to the problems above, especially when the ambient dimension is fixed.

\subsubsection{The Critical Point Method.}\label{sec:critical}

The \emph{critical point method} is a central technique in algorithmic real algebraic geometry. It reduces global geometric questions to the study of local extrema of polynomial maps restricted to algebraic or semi-algebraic sets. Typically, one considers projection maps, such as the projection onto the first coordinate:
\[
\pi: \mathbb{R}^n \to \mathbb{R}, \quad (x_1, \dots, x_n) \mapsto x_1,
\]
and studies the behavior of \( \pi \) restricted to the given set.

\paragraph{Real root finding.}
For a real algebraic variety \( V \subset \mathbb{R}^n \), the critical point method allows the computation of at least one real point in each semi-algebraically connected component. This is achieved by finding the critical points of the projection \( \pi \) restricted to \( V \). Under genericity assumptions, each component contains a local extremum of \( \pi \), and these extrema correspond to critical points. The method is thus the basis for efficient algorithms for real root isolation in fixed dimension.

\paragraph{Connectivity and sampling.}
Using the same principles, one can test whether a given semi-algebraic set is connected and compute a sample point in each connected component. This is particularly useful in applications such as motion planning and control theory, where one needs to explore the connected regions of feasible configurations.

\paragraph{Computing topological invariants.}
The method also plays a role in computing topological invariants of semi-algebraic sets, such as the number of connected components and the Euler–Poincaré characteristic. In some cases, higher Betti numbers can be computed using stratifications based on critical values.

\paragraph{Quantifier elimination and decision procedures.}
The critical point method contributes to algorithms for deciding the truth of first-order formulas over the real numbers. In these algorithms, quantifiers are eliminated recursively by analyzing the real algebraic structure of the solution set and computing suitable projections.

\medskip

The critical points of \( \pi|_S \), where \( S \subset \mathbb{R}^n \) is a semi-algebraic set, can be computed using Lagrange multipliers or more generally using tools from differential and computational algebra. In the bounded and smooth case, the set of critical points is finite, and computing them allows us to recover global geometric information from local conditions.

However, when \( S \) is defined by symmetric polynomials, the use of coordinate projections such as \( x_1 \) breaks the symmetry. In such cases, modifications of the critical point method—such as using invariant projections—are required. We will return to this in later sections.

\subsubsection{Roadmaps and Connectivity.}\label{sec:road}

A complementary approach to understanding the connectivity of semi-algebraic sets is the construction of \emph{roadmaps}. Roadmaps were introduced by Canny~\cite{Canny} as a tool for describing the topological structure of a semi-algebraic set via a one-dimensional subset.

\begin{definition}
Let \( S \subset \mathbb{R}^n \) be a semi-algebraic set, and let \( \pi: \mathbb{R}^n \to \mathbb{R} \) denote the projection onto the first coordinate. For \( a \in \mathbb{R} \), define the fiber
\[
S_a = \{ \b \in \mathbb{R}^{n-1} : (a, \b) \in S \}.
\]
A \emph{roadmap} of \( S \) is a semi-algebraic subset \( \RM(S) \subset S \) of dimension at most one that satisfies:
\begin{itemize}
  \item For every semi-algebraically connected component \( C \subset S \), the intersection \( C \cap \RM(S) \) is non-empty and semi-algebraically connected.
  \item For every \( a \in \mathbb{R} \), and for every semi-algebraically connected component \( C' \subset S_a \), the intersection \( C' \cap \RM(S) \) is non-empty.
\end{itemize}
\end{definition}
Thus, \( \RM(S) \) forms a one-dimensional ``skeleton'' that intersects every component of \( S \) and allows one to navigate within and between components.

If \( \mathcal{M} \subset S \) is a finite set of points, a \emph{roadmap for the pair} \( (S, \mathcal{M}) \) is a roadmap \( \RM(S, \mathcal{M}) \subset S \) that contains all points in \( \mathcal{M} \). This is particularly useful for testing whether two designated points \( \u, \v \in S \) lie in the same connected component: it suffices to check whether they are connected within \( \RM(S, \mathcal{M}) \).

\paragraph{Construction.}
The construction of roadmaps is recursive and builds upon the critical point method. One proceeds by projecting onto the first coordinate, computing critical values, and sampling points in fibers. Curves are then traced through these sample points to ensure connectivity across slices, and the construction is recursively applied to the lower-dimensional fibers. This approach leads to singly exponential algorithms in the number of variables, significantly improving upon the doubly exponential complexity of earlier approaches such as cylindrical algebraic decomposition.

\paragraph{Applications.}
Roadmaps have applications in robotics (motion planning), verification, control, and computational geometry, where it is crucial to understand how to navigate within a feasible region defined by polynomial constraints.

\subsubsection{Sums of Squares Techniques.}
\label{sec:sos}
In recent years, particularly following the foundational works of Lasserre and Parrilo \cite{parrilo2000thesis,lasserre2001global}, a new paradigm has emerged that has significantly advanced computational real algebraic geometry: the study of polynomial decompositions as sums of squares.

\begin{definition}
Let \( f \in \mathbb{R}[x_1, \ldots, x_n] \). We say that \( f \) is a \emph{sum-of-squares} (SOS) polynomial if there exist polynomials \( q_1, \ldots, q_k \in \mathbb{R}[x_1, \ldots, x_n] \) such that
\[
f = q_1^2 + \cdots + q_k^2.
\]
\end{definition}

\begin{example}
For instance, 
\[
f(x,y) = x^2 - 2xy + y^2 = (x - y)^2
\]
is a sum of squares.
\end{example}
Note that while the right hand side, i.e., the decomposition of $f$ as a square clearly shows that $f$ is non-negative on \( \mathbb{R}^2 \) this property of $f$ is less obviously clear from its original definition on the right hand side.  
In general, if a polynomial \( f(x) \in \mathbb{R}[x_1, \dots, x_n] \) admits a decomposition 
\[
p(x) = \sum_i q_i(x)^2,
\]
then clearly \( f(x) \geq 0 \) for all \( x \in \mathbb{R}^n \) and the  SOS decomposition provides a \textbf{certificate of non-negativity} for \( p \).

However, not all non-negative polynomials are sums of squares. A classical example due to Motzkin is
\[
M(x, y, z) = x^4 y^2 + x^2 y^4 + z^6 - 3 x^2 y^2 z^2,
\]
which satisfies \( M(x, y, z) \geq 0 \) for all real \( x, y, z \), but is not SOS. This illustrates that the cone of SOS polynomials is strictly contained in the cone of non-negative polynomials, except in special cases (e.g., univariate polynomials or quadratics).

\paragraph{Sums of squares certificates.} Sums of square techniques can play a key role in certifying the emptiness of semi-algebraic sets. Suppose we are given polynomials
\[
g_1, \dots, g_m, \quad h_1, \dots, h_k \in \mathbb{R}[x_1, \dots, x_n],
\]
and consider the semi-algebraic set
\[
S = \left\{ x \in \mathbb{R}^n : g_i(x) \geq 0, \; h_j(x) = 0 \right\}.
\]
Then the so-called \emph{Positivstellens\"atz}  (for a comprehensive overview see  for example \cite{prestel2001positive}) provides a certificate for the infeasibility (i.e., emptiness) of \( S \). For example, with Stengle's Positivstellensatz \cite{stengle1974nullstellensatz}, we obtain that  \( S = \emptyset \), if there exists an identity.
\[
-1 = \sum_{\alpha} \sigma_\alpha(x) m_\alpha(x) + \sum_{j} t_j(x) h_j(x),
\]
where each \( \sigma_\alpha(x) \) is an SOS polynomial and \( m_\alpha(x) \) is a product of the \( g_i \). Such an identity clearly serves as a certificate for the emptiness of $S$: if \( x \in S \), then all \( g_i(x) \geq 0 \) and each \( \sigma_\alpha(x) m_\alpha(x) \geq 0 \), while the right-hand side equals \(-1\), a contradiction.

Thus, such  certificates provide an algebraic method to prove infeasibility of systems of polynomial constraints and are central in real algebraic geometry and polynomial optimization. This leads to the following computational problem.

\begin{pb}[Sum of Squares Decomposition]
Let $f\in\R[x_1,\ldots,x_n]$ be a polynomial of degree $2d$. Decide if there exists a sum of squares decomposition for $f$
\end{pb}
The renewed interest in sums of squares stems from the connection of this computational problem to  to  semidefinite programming (SDP), which can be solved numerically efficient (\cite{vandenberghe1996semidefinite}). The key in this approach is the so called Gram matrix method \cite{choi1995sums} as the following. Let \( p(x) \) be a polynomial of degree \( 2d \). Define \( v(x) \) to be the vector of all monomials of degree at most \( d \). Then \( p(x) \) is SOS if and only if there exists a symmetric matrix \( Q \succeq 0 \) such that
\[
p(x) = v(x)^T Q v(x).
\]
This is known as the \emph{Gram matrix representation} of \( p \). The condition \( Q \succeq 0 \) means that \( Q \) is positive semidefinite, and guarantees that \( p(x) \) is a sum of squares.

To find such a matrix \( Q \), one expands the identity \( p(x) = v(x)^T Q v(x) \) and equates the coefficients of the resulting polynomial with those of \( p \). This leads to a system of linear equations in the entries of \( Q \), together with the semidefinite constraint \( Q \succeq 0 \). Hence, deciding whether \( p \) is SOS reduces to solving an SDP feasibility problem.

Semidefinite programming is a well-studied area of convex optimization, and such problems can be solved efficiently (up to numerical accuracy) using interior point methods. Thus, SOS certificates provide a efficient solution to finding a sums of squares decomposition if existent and thus a relaxation of the generally hard problem of polynomial non-negativity.

In the next sections, we will explore how the mentioned computational problems can be adapted to exploit symmetry when the defining polynomials of the semi-algebraic set are symmetric.

\subsection{Computational Models}\label{seq:comm}

 In this section, we describe some basic algebraic tasks and the computational model used. Our model is a Random Access Machine (RAM) over a field \(\K\).  

\subsubsection{Straight-line Programs.} Let \( f \) be a polynomial of degree \( d \) in \( n \) variables. The standard (dense) encoding of \( f \) is an array of \(\binom{n+d}{n}\) coefficients. In contrast, a sparse encoding only records the nonzero coefficient-exponent tuples. Here, we use an alternative representation: the \emph{straight-line program} (SLP), also known as an \emph{algebraic circuit}.  

The idea of using straight-line programs first emerged in the context of probabilistic polynomial identity testing. In computer algebra applications, SLPs were introduced for tasks such as elimination in univariate problems (see~\cite{heintz1981absolute,kaltofen1988greatest,kaltofen1989factorization}). Our motivation for adopting this representation is rooted in its use in polynomial system-solving algorithms based on Newton-Hensel lifting techniques, initiated in~\cite{giusti1995polynomial,giusti1997lower,giusti1998straight}.

\begin{definition}[Straight-line Programs]
A \emph{straight-line program} $($SLP$)$ computing a set of polynomials \( f_1, \dots, f_m \in \K[x_1, \dots, x_n] \) is a sequence
\[
\gamma = (\gamma_{-n+1}, \dots, \gamma_0, \gamma_1, \dots, \gamma_L),
\]
where the initial inputs are the variables: \(\gamma_{-n+1} := x_1, \dots, \gamma_0 := x_n\), and for \( k > 0 \), each \(\gamma_k \) is one of the following forms:
\[
\gamma_k = a * \gamma_i \quad \text{or} \quad \gamma_k = \gamma_i * \gamma_j,
\]
with \( * \in \{+, -, \times\} \), \( a \in \K \), and \( i, j < k \). The number \( L \) is called the \emph{length} of the program.
\end{definition}

\begin{example}
Consider the polynomial \( x_1^{2^k} \in \K[x_1, \dots, x_n] \).  Its dense encoding consists of the coefficient 1 followed by \(2^k\) zeros and its sparse encoding is the pair \((2^k; 1)\).    A straight-line program computing it is:
    \begin{equation} \label{eq:slp}
        \gamma_0 = x_1, \quad \gamma_1 = \gamma_0 \times \gamma_0 = x_1^2, \quad \dots, \quad \gamma_k = \gamma_{k-1} \times \gamma_{k-1} = x_1^{2^k}.
    \end{equation}
Hence, the dense encoding has length \(2^k + 1\), the sparse encoding has length 2, and the SLP has length at most \(k\).
\end{example}

Straight-line programs are stable under linear changes of variables, unlike sparse encodings. For instance, to compute \((x_1 + x_2)^{2^k}\), we can augment the SLP in~\eqref{eq:slp} with \(\gamma_{-1}' = x_2\) and \(\gamma_1' = \gamma_{-1}' + \gamma_0\). This leads to an SLP of length at most \(k + 1\). In contrast, the dense and sparse encodings of this polynomial have length \(\binom{2^k}{2} = O(2^{2k})\).  Furthermore, some polynomials with many monomials can still be computed by very short SLPs. For example, \( (x_1 + x_2)^{2^k} \) contains \(2^k + 1\) monomials but can be computed with an SLP of length \(k\).

Importantly, using SLPs is not a restrictive assumption, as they generalize both dense and sparse encodings. Specifically, any polynomial of degree \(d\) in \(n\) variables can be computed by an SLP of length at most \(3\binom{n+d}{n}\). This follows from the fact that there are \(\binom{n+d}{n}\) monomials of degree at most \(d\); computing each monomial and summing them up takes at most \(2\binom{n+d}{n}\) operations. For a sparse polynomial with \(N\) nonzero monomials, an SLP of length \(O(Nd)\) suffices.

%

\subsubsection{Zero-dimensional Parametrizations.}
We use univariate polynomials to represent finite sets of points in \(\Kbar^n\). This representation, introduced in the early works of Kronecker and Macaulay \cite{Kronecker1882,Macaulay1916}, has since become a fundamental data structure in computer algebra; see, for instance, \cite{giusti1998straight,Gianni1989,Alonso1996,Giustifast1995,GiustiGrob2001,Rouillier1999}. 

\begin{definition}[Zero-dimensional Parametrization] \label{def:zeor-dim-para}
Let \(V \subset \Kbar^n\) be a finite set defined by polynomials in \(\K[x_1, \dots, x_n]\). A \emph{zero-dimensional parametrization} \(\mathscr{R} = ((q, v_1, \dots, v_n), \gamma)\) of \(V\) consists of:
\begin{itemize}
    \item a square-free polynomial \(q \in \K[T]\), where \(T\) is a new indeterminate and \(\deg(q) = |V|\);
    \item polynomials \(v_1, \dots, v_n \in \K[T]\), each of degree at most \(\deg(q)\), such that
    \[
    V = \left\{ \left( \frac{v_1(\tau)}{q'(\tau)}, \dots, \frac{v_n(\tau)}{q'(\tau)} \right) \in \Kbar^n \, : \, q(\tau) = 0 \right\},
    \]
    where \(q' = \frac{\mathrm{d}q}{\mathrm{d}T}\);
    \item a linear form \(\gamma = \gamma_1 x_1 + \cdots + \gamma_n x_n\), with coefficients in \(\K\), such that
    \[
    \gamma(v_1, \dots, v_n) \equiv T q' \mod q,
    \]
    i.e., the roots of \(q\) are the values taken by \(\gamma\) on \(V\).
\end{itemize}
When these conditions hold, we write \(V = Z(\mathscr{R})\).
\end{definition}
One reason for using this representation is that the rational parametrization, with \(q'\) in the denominator, allows effective control of the bit-size of coefficients when \(\K\) is the field of rational numbers, or of the degree in \(y\) when \(\K\) is the field of fractions \(k(y)\), where \(k\) is a field \cite{Alonso1996,GiustiGrob2001,Rouillier1999}.
\begin{example}
Consider the finite set \( V = \{(1,2), (3,4)\} \subset \overline{\mathbb{K}}^2 \). A zero-dimensional parametrization \(\mathscr{R} = ((q, v_1, v_2), \gamma)\) of \(V\) can be given as follows:
\begin{align*}
& q(t) = (T - 1)(T - 3) = T^2 - 4T + 3,\\
& v_1(t) = T \cdot q'(T) = T \cdot (2T - 4), \\
& v_2(t) = (T + 1) \cdot q'(T) = (T + 1) \cdot (2T - 4),
\end{align*}
and the linear form
\(
\gamma = x_1.
\)
Then, for each root \(\tau\) of \(q(T)\),
\[
\left( \frac{v_1(\tau)}{q'(\tau)}, \frac{v_2(\tau)}{q'(\tau)} \right) \in V.
\]
Indeed, at \(\tau=1\), we have
\(
q'(1) = 2(1) - 4 = -2,
\) and
\[
\frac{v_1(1)}{q'(1)} = \frac{1 \times (-2)}{-2} = 1, \quad
\frac{v_2(1)}{q'(1)} = \frac{(1+1) \times (-2)}{-2} = 2,
\]
giving the point \((1,2)\).
Similarly, at \(\tau=3\),
\(
q'(3) = 2(3) - 4 = 2,
\) and
\[
\frac{v_1(3)}{q'(3)} = \frac{3 \times 2}{2} = 3, \quad
\frac{v_2(3)}{q'(3)} = \frac{(3+1) \times 2}{2} = 4,
\]
giving the point \((3,4)\).
Thus, \(\mathscr{R}\) parametrizes the finite set \(V\).
\end{example}


\subsubsection{Encoding of Real Algebraic Points.}
\label{subs_realpoint}
We encode such a point using univariate polynomials and a Thom encoding.  
\begin{definition}
Let $\mathcal{F} \subseteq \R[x_1, \dots, x_n]$ and let $x \in \R^n$.  
A mapping $\zeta: \mathcal{F} \rightarrow \{-1, 0, 1\}$ is called a {\em sign condition} on $\mathcal{F}$.  
The {\em sign condition realized by} $\mathcal{F}$ at the point $x$ is the map
\[
\sign(\mathcal{F}, x): \mathcal{F} \rightarrow \{-1, 0, 1\}, \quad f \mapsto \sign(f(x)).
\]
We say that $\mathcal{F}$ {\em realizes} the sign condition $\zeta \in \{-1, 0, 1\}^{\mathcal{F}}$ if and only if $$\sign(\mathcal{F}, x) = \zeta.$$

Let $q$ be a polynomial in $\Q[T]$. Define the list of formal derivatives of $q$ as
\[
\der(q) := \{q, q^{(1)}, q^{(2)}, \dots, q^{(\deg(q))}\},
\]
where $q^{(i)}$ denotes the $i$-th formal derivative of $q$ for $i \ge 1$.
A mapping $\zeta: \der(q) \rightarrow \{-1, 0, 1\}$ is called a {\em Thom encoding} of a real root $\vartheta \in \R$ of $q$ if $\zeta(q) = 0$ and the sign condition $\zeta$ is realized at $\vartheta$, i.e., $\sign(\der(q), \vartheta) = \zeta$.

\end{definition}
Note that distinct real roots of $q$ correspond to distinct Thom encodings~\cite[Proposition~2.28]{BPR06}.  
For further details on Thom encodings and real univariate representations, see~\cite[Chapters 2 and 12]{BPR06}.

\begin{definition}[{Representations of Real Algebraic Points}]
    A {\em real univariate representation} representing $x \in \R^n$
    consists of: 
    \begin{itemize}
        \item a zero-dimensional parametrization 
        \[
        \mathscr{Q} = (q(T), q_0(T), v_1(T), \dots, v_n(T)),
        \]
        where $q, q_0, v_1, \dots, v_n \in \Q[T]$, with $\gcd(q, q_0) = 1$  
        \((\)e.g., one can take a square-free polynomial \(q\) and \(q_0 = \frac{dq}{dT}\) as in Definition~\ref{def:zeor-dim-para} \()\), and
        \item a Thom encoding $\zeta$ representing an element $\vartheta \in \R$ such that
        \[
        q(\vartheta) = 0 \quad \text{and} \quad x =
        \left(\frac{v_1(\vartheta)}{q_0(\vartheta)}, \dots,
        \frac{v_n(\vartheta)}{q_0(\vartheta)}\right) \in \R^n.
        \]
    \end{itemize}
\end{definition}

\begin{example}
A real univariate representation for the point $x = \left( \frac{1}{2}, \, \frac{3}{2} - \frac{\sqrt{2}}{4} \right)$ is $(\mathscr{Q}, \zeta)$, where
\[
\mathscr{Q} = (q, q_0, q_1, q_2) = (T^2 - 2, 2T, T, 3T - 1),
\]
and $\zeta = (1, 1)$ (equivalently, $\zeta = (+, +)$, i.e., $q' > 0 \wedge q'' > 0$).
Indeed, let $\vartheta = \sqrt{2}$, a root of $q$. Then
\[
x = \left( \frac{q_1(\vartheta)}{q_0(\vartheta)}, \, \frac{q_2(\vartheta)}{q_0(\vartheta)} \right).
\]
Moreover, $q'(T) = 2T$ and $q''(T) = 2$, so $q'(\vartheta) > 0$ and $q''(\vartheta) > 0$.  
Thus, $\zeta = (+, +)$ serves as a Thom encoding of $\vartheta$.
\end{example}
Let \( q \in \mathbb{R}[T] \) be a non-zero polynomial. We define the 
\(
\textsc{Thom\_Encoding}(q),
\)
which returns the ordered list of Thom encodings of the real roots of \( q \). This routine can be implemented using \cite[Algorithm 10.14]{BPR06}, and it requires 
\[
O(\delta_q^4 \log(\delta_q))
\]
arithmetic operations over \(\mathbb{Q}\), where \( \delta_q = \deg(q) \).

\begin{example}
Let \( q(T) = T^3 - 3T + 1 \in \mathbb{Q}[T] \). This polynomial has three real roots. Using \textsc{Maple}'s \texttt{RealRootIsolate}, we find numerical approximations:
\[
\vartheta_1 \approx -1.879,\quad \vartheta_2 \approx 0.347,\quad \vartheta_3 \approx 1.532.
\]
We compute the derivatives:
\[
q'(T) = 3T^2 - 3, \quad q''(T) = 6T.
\]
Evaluating signs at each root:
\[
\begin{aligned}
& q'(\vartheta_1) < 0,\quad q''(\vartheta_1) < 0 && \Rightarrow \text{encoding } (-,-),\\
& q'(\vartheta_2) > 0,\quad q''(\vartheta_2) < 0 && \Rightarrow \text{encoding } (+,-),\\
& q'(\vartheta_3) > 0,\quad q''(\vartheta_3) > 0 && \Rightarrow \text{encoding } (+,+).
\end{aligned}
\]
Thus, \textsc{Thom\_Encoding}$(q)$ returns:
\(
[(-,-),\ (+,-),\ (+,+)].
\)
\end{example}

Now, let \( p \in \mathbb{R}[T] \) be another polynomial of degree \( \delta_p \), and let \( \thom(q) \) denote the list of Thom encodings of the real roots of \( q \). We define the routine
\(
\textsc{Sign\_ThomEncoding}(q, p, \thom(q)),
\)
which, for each encoding \( \zeta \in \thom(q) \) corresponding to a real root \( \vartheta \) of \( q \), returns the sign of \( p(\vartheta) \). This procedure can be carried out using \cite[Algorithm 10.15]{BPR06}, and has complexity
\[
O\left( \delta_q^2 \left( \delta_q \log(\delta_q) + \delta_p \right) \right)
\]
arithmetic operations in \(\mathbb{Q}\).
\begin{example}
Let \( p(T) = T^2 - 2 \). We evaluate the sign of \( p \) at each real root of \( q \):
\[
\begin{aligned}
& p(\vartheta_1) = (-1.879)^2 - 2 \approx 1.53 > 0,\\
& p(\vartheta_2) \approx 0.347^2 - 2 \approx -1.88 < 0,\\
& p(\vartheta_3) \approx 1.532^2 - 2 \approx 0.35 > 0.
\end{aligned}
\]
Hence, \textsc{Sign\_ThomEncoding}$(q, p, \thom(q))$ returns:
\(
[+,\ -,\ +].
\)
\end{example}
We can use the following commands in \textsc{Maple}, a computer algebra software. 
\begin{codebox}
\begin{verbatim}
with(RealRootIsolate):
with(RootFinding):
q := T^3 - 3*T + 1:
p := T^2 - 2:
RealRootIsolate(q); # return intervals containing each real root 
\end{verbatim}
\end{codebox}

\hspace{-0.5 cm}\textsc{Maple} doesn't have \textsc{Thom\_Encoding} as a named function, but we can simulate the process using \texttt{RealRootClassification} or by evaluating derivatives' signs:
\begin{codebox}
\begin{verbatim}
f := q:
df := diff(f, T):
d2f := diff(df, T):
roots := fsolve(f, T, useassumptions=true); # approx roots

eval([sign(df), sign(d2f)], T=~[roots[]]); # sign  at each root
\end{verbatim}
\end{codebox}
\hspace{-0.5 cm}Finally, 
\begin{codebox}
\begin{verbatim}
map(r -> sign(eval(p, T = r)), [roots]);
\end{verbatim}
\end{codebox}
\hspace{-0.5 cm}gives us the sign of \(p(\vartheta)\) at each root \(\vartheta\) of \(q\).

\subsubsection{The Probabilistic Aspects.}
Randomized (or probabilistic) algorithms are algorithms that make random choices during their execution. These algorithms are widely used in both theoretical computer science and practical applications, as they often lead to simpler, faster, or more elegant solutions than their deterministic counterparts. There are several types of randomized algorithms, each with different guarantees on correctness and performance. The most common categories are Las Vegas algorithms and Monte Carlo algorithms.

A {Las Vegas algorithm} is a type of randomized algorithm that always produces the correct output. However, its running time is a random variable that may vary depending on the random choices made during execution.
The key feature of Las Vegas algorithms is that they never return incorrect results. Instead, they may run longer in some cases while still ensuring correctness. Las Vegas algorithms are ideal when correctness is critical and some variability in performance is acceptable.

In contrast, {Monte Carlo algorithms} have a fixed running time but may return incorrect results with a small probability (typically less than \(1/2\)). These algorithms trade certainty of correctness for predictable performance. Although there is a small chance of error, repeating the algorithm multiple times can reduce this probability to an arbitrarily low level.
Monte Carlo algorithms are useful in situations where speed is crucial and a small probability of error is acceptable. This error probability can be rigorously controlled using the Schwartz–Zippel lemma (see, e.g., \cite{Schwartz1980} and \cite[Lemma~6.44]{GathenGerhard2003}).

\begin{lemma}[Schwartz–Zippel Lemma]
Let \(R\) be an integral domain and \(S\) a finite subset of \(R\). Let \(f\) be a polynomial in \(R[x_1, \dots, x_n]\) of total degree at most \(d\). If \(f\) is not the zero polynomial, then \(f\) has at most \(d|S|^{n-1}\) zeros in \(S^n\).
\end{lemma}

\begin{example}
Let \(S\) be a finite set of complex numbers of cardinality \(100\), and consider the generic quadratic polynomial \(h = ax^2 + bx + c\). The polynomial \(h\) has two distinct roots when the discriminant \(\Delta(a,b,c) = a(b^2 - 4ac)\) is nonzero.

Choosing \((\alpha, \beta, \gamma)\) uniformly at random from \(S^3\), by the Schwartz–Zippel lemma, the probability that \(\Delta(\alpha, \beta, \gamma) = 0\) is at most \[\frac{1}{\deg(\Delta)\cdot |S|^2} = \frac{1}{3 \cdot 100^2} = \frac{1}{30000}.\]Hence, the probability that \(h = ax^2 + bx + c\) has two distinct roots is at least \(1 - {1}/{30000}\).
\end{example}

The randomized algorithms presented in this chapter are of the Monte Carlo type in the sense that they make random choices of points that lead to the correctness of computations; these points are chosen outside certain Zariski closed subsets (e.g, the \(\Delta(a,b,c)\) above) of an appropriate affine space. More precisely, we define sufficiently large finite subsets of the numbers, with cardinalities depending on the degrees of the polynomials defining these Zariski closed sets, and then make the necessary random choices uniformly from these sets.

\subsubsection{Generic Properties.} In algebraic geometry, certain properties hold for “most” objects of a given type. For example, most square matrices are invertible, and most univariate polynomials of degree \(d\) have \(d\) distinct solutions. We use the term “generic” to describe such situations.

\begin{definition}
Let \(X\) be a variety. A subset \(Y \subset X\) is called \emph{generic} if it contains a non-empty Zariski open subset of \(X\). A property is said to be \emph{generic} if the set of points on which that property holds is a generic subset.
\end{definition}

In other words, a property is said to hold generically for a polynomial system \(\f\) if there exists a nonzero polynomial in the coefficients of \(\f\) such that the property holds for all \(\f\) for which that polynomial does not vanish. Note that the notion of genericity depends on the context, and care must be exercised in its use.

\begin{example}
Consider a quadratic polynomial \(h = ax^2 + bx + c\). Generically, this equation has two solutions in \(\C\), counted with multiplicity. This property holds when \(a \ne 0\). Let \(\calO\) be the Zariski open subset of \(\C\) defined as the complement of the Zariski closed set \(V(a)\); this set is non-empty, since for example, \(1 \in \calO\). Then, for any \(a \in \calO\), the equation \(ax^2 + bx + c = 0\) has two solutions in \(\C\), counted with multiplicity.

On the other hand, generically, the equation \(ax^2 + bx + c = 0\) has two \emph{distinct} solutions in \(\C\). This property holds when \(a(b^2 - 4ac) \ne 0\). Define the non-empty Zariski open subset \(\calO_1 \subset \C^3\) as the complement of the Zariski closed set \(V(a(b^2 - 4ac))\). Then, for any point \((a, b, c) \in \calO_1\), the equation \(ax^2 + bx + c = 0\) has two distinct solutions in \(\C\).
\end{example}

\subsection{Symmetric Polynomials}\label{sec:sym}
Let \(\K[x_1, \dots, x_n]\) be the ring of polynomials in \(n\) variables over a field \(\K\). The symmetric group \(\S_n\) acts on this ring by permuting the variables. That is, for any \(\sigma \in \S_n\) and \(f \in \K[x_1, \dots, x_n]\),
\(
\sigma(f) = f(x_{\sigma(1)}, \dots, x_{\sigma(n)}).
\)

\begin{definition}[Symmetric Polynomials]
A polynomial \(f \in \K[x_1, \dots, x_n]\) is called \emph{symmetric} \emph{(}or \(\S_n\)-\emph{invariant)} if it is invariant under the action of the symmetric group. That is,
\[
\sigma(f) = f \ \text{for all } \ \sigma \in \S_n.
\]
\end{definition}
For a non-negative integer \(k\), let \(\S_n^k\) denote the set of homogeneous symmetric polynomials of degree \(k\). The symmetric polynomials form a subring \(\K[x_1, \dots, x_n]^{\S_n}\) of \(\K[x_1, \dots, x_n]\), which is a graded ring with
\[
\K[x_1, \dots, x_n]^{\S_n} = \bigoplus_{k \ge 0} \S_n^k.
\]

\begin{definition}[Elementary Symmetric Polynomials]
For each integer \(k\) with \(1 \leq k \leq n\), the \emph{elementary symmetric polynomial} \(e_k(x_1, \dots, x_n)\) is defined as
\[
e_k(x_1, \dots, x_n) = \sum_{1 \leq i_1 < i_2 < \cdots < i_k \leq n} x_{i_1} x_{i_2} \cdots x_{i_k}.
\]
By convention, \(e_0 = 1\).
\end{definition}
Let \(T\) be a new variable and consider the polynomial 
\[P(T) = (T- x_1)(T-x_2) \cdots (T-x_n)\] with \(\X = (x_1, \dots, x_n)\) being its roots. If we expand the right-hand side, we have
\begin{equation} \label{eq:vieta}
    P(T) = T^n - e_1(\X)T^{n-1} + \cdots + (-1)^{n-1} e_{n-1}(\X) T^{n-1} + (-1)^ne_n(\X).
\end{equation} In general, for any univariate polynomial of degree \(n\)
\[Q(T) = c_nT^n + \cdots + c_1T + c_0 \text{ with } c_n \ne 0\] has \(n\) complex roots \(\alpha_1, \dots, \alpha_n\). Then we have the following, which are known as Vieta's formulas, named after  François Viète
\[e_k(\alpha_1, \dots, \alpha_n) = (-)^k \frac{c_{n-k}}{c_n}\text{\ for } k = 1, \dots, n. \]
\begin{definition}[Complete Homogeneous Symmetric Polynomials]
For each non-negative integer \(k\), the \emph{complete homogeneous symmetric polynomial} \(h_k(x_1, \dots, x_n)\) is defined by
\[
h_k(x_1, \dots, x_n) = \sum_{\substack{i_1 + \cdots + i_n = k \\ i_j \geq 0}} x_1^{i_1} x_2^{i_2} \cdots x_n^{i_n}.
\]
In particular, \(h_0 = 1\).
\end{definition}

\begin{definition}[Power Sum Symmetric Polynomials]
For each positive integer \(k\), the \emph{power sum symmetric polynomial} \(p_k(x_1, \dots, x_n)\) is defined as
\[
p_k(x_1, \dots, x_n) = \sum_{i=1}^n x_i^k.
\]
\end{definition}

\begin{definition}[Monomial Symmetric Polynomials]
Given a partition \(\lambda = (\lambda_1 \geq \lambda_2 \geq \cdots \geq \lambda_\ell > 0)\) with \(\ell \leq n\), the \emph{monomial symmetric polynomial} \(m_\lambda(x_1, \dots, x_n)\) is the sum of all distinct monomials obtained by permuting the variables in the monomial
\[
x_1^{\lambda_1} x_2^{\lambda_2} \cdots x_\ell^{\lambda_\ell}.
\]
Formally,
\[
m_\lambda = \sum_{\alpha} x_1^{\alpha_1} x_2^{\alpha_2} \cdots x_n^{\alpha_n},
\]
where the sum runs over all distinct rearrangements \(\alpha = (\alpha_1, \dots, \alpha_n)\) of \(\lambda\) padded with zeros to length \(n\).
\end{definition}

The relations among these different generating sets are classical and given by \emph{Newton's identities}, which express power sums in terms of elementary symmetric polynomials and vice versa.

\begin{theorem}[Newton's Identities]
For \(k = 1, 2, \dots, n\), the power sums \(p_k\) and elementary symmetric polynomials \(e_k\) satisfy the recursive relations
\[
p_k = (-1)^{k-1} k e_k + \sum_{j=1}^{k-1} (-1)^{j-1} e_j p_{k-j},
\]
where by convention \(e_0 = 1\) and \(p_0 = n\).
Equivalently, these identities can be rearranged to express each \(e_k\) as a polynomial in the power sums \(\{p_1, \dots, p_k\}\).
\end{theorem}

\subsubsection{Fundamental Theorem of Symmetric Polynomials.}
From the elementary symmetric polynomials, we can construct other symmetric polynomials by taking polynomials in \(e_1, \dots, e_n\). For example,
\[
e_1e_2 - 3e_3 = x^2 y + x y^2 + y^2 z + y z^2 + z^2 x + z x^2.
\]
Conversely, every symmetric polynomial can be uniquely represented in terms of the elementary symmetric polynomials, which is known as the Fundamental Theorem of Symmetric Polynomials (FTSP) (see, e.g., {\cite[Theorem\ 3.1]{derksen2015}}).

\begin{theorem}[Fundamental Theorem of Symmetric Polynomials]\label{thm:funda}
Every symmetric polynomial \(f \in \K[x_1, \dots, x_n]^{\S_n}\) can be expressed uniquely as a polynomial in the elementary symmetric polynomials \(e_1, e_2, \dots, e_n\). In other words, there exists a unique polynomial \(F \in \K[y_1, \dots, y_n]\) such that
\[
f(x_1, \dots, x_n) = F(e_1(x_1, \dots, x_n), e_2(x_1, \dots, x_n), \dots, e_n(x_1, \dots, x_n)).
\]
\end{theorem}

There are many proofs of the FTSP. One such proof, given by Gauss, includes what is considered the earliest explicit statement of the lexicographic order among monomials: given two terms \(M x^\alpha y^\beta z^\gamma \cdots\) and \(M' x^{\alpha'} y^{\beta'} z^{\gamma'} \cdots\), the first is considered greater if:
\[
\text{either } \alpha > \alpha', \text{ or } \alpha = \alpha' \text{ and } \beta > \beta', \text{ or } \alpha = \alpha', \beta = \beta', \text{ and } \gamma > \gamma', \text{ etc}.
\]
For example, \(x^2 y > x y^{10}\). We will not present a full proof, which, in fact, yields an algorithm for writing a symmetric polynomial in terms of the elementary symmetric polynomials, but instead provide a concrete example to illustrate how the algorithm works. A detailed proof can be found in \cite[Chapter 7 - Section 1 - Theorem 3]{Cox1997}.

\begin{example} \label{ex:FTSF}
Consider the symmetric polynomial
\[
f = x^2 y + x y^2 + y^2 z + y z^2 + z^2 x + z x^2 \in \K[x, y, z].
\]
The largest monomial under the lexicographic order is \(x^2 y\), corresponding to the exponent vector \(\alpha = 2, \beta = 1\). We then compute
\[
f_1 = f - e_1^{\alpha - \beta} e_2^\beta = f - (x + y + z)(xy + yz + zx) = -3xyz.
\]
Clearly, \(f_1 = -3e_3\), and thus \(f = e_1 e_2 - 3e_3\).
\end{example}

In general, if the "largest" monomial (under lex order) is \(x_1^{\alpha_1} \cdots x_n^{\alpha_n}\), then \(\alpha_i \ge \alpha_{i+1}\) for all \(i = 1, \dots, n-1\). Moreover, the polynomial
\[
e_1^{\alpha_1 - \alpha_2} e_2^{\alpha_2 - \alpha_3} \cdots e_{n-1}^{\alpha_{n-1} - \alpha_n} e_n^{\alpha_n}
\]
shares the same largest monomial. We define
\[
f_1 = f - M \cdot e_1^{\alpha_1 - \alpha_2} e_2^{\alpha_2 - \alpha_3} \cdots e_{n-1}^{\alpha_{n-1} - \alpha_n} e_n^{\alpha_n},
\]
where \(M\) is the coefficient of the largest monomial in \(f\). The new polynomial \(f_1\) then has a strictly smaller leading monomial. We repeat this process until all terms are accounted for.

\medskip

Besides the elementary symmetric polynomials, the ring \(\K[x_1, \dots, x_n]^{\S_n}\) is also generated by the complete homogeneous symmetric polynomials \(h_1, \dots, h_n\) and the power sum symmetric polynomials \(p_1, \dots, p_n\). That is,
\[
\K[x_1, \dots, x_n]^{\S_n} = \K[h_1, \dots, h_n] = \K[p_1, \dots, p_n].
\]
The sets \((p_k)\) and \((h_k)\) form algebraically independent generating sets, making \(\K[x_1, \dots, x_n]^{\S_n}\) isomorphic to a polynomial ring in these variables. However, the power sums \((p_k)\) form an algebraically independent set only over fields of characteristic zero (e.g., \(\mathbb{Q}, \mathbb{R}, \mathbb{C}\)).

\begin{example} \label{ex_10}
Consider the symmetric polynomial
\[
f = x_1^2 x_2 + x_1 x_2^2 + x_2^2 x_3 + x_2 x_3^2 + x_3^2 x_1 + x_3 x_1^2 \in \mathbb{Q}[x_1, x_2, x_3],
\]
as in Example~\ref{ex:FTSF}. Then
\(
f = e_1 e_2 - 3e_3.
\)
The complete homogeneous symmetric polynomials in three variables are:
\[
h_1 = x_1 + x_2 + x_3 = e_1, \quad
h_2 = \sum_{1 \le i \le j \le 3} x_i x_j, \quad
h_3 = \sum_{1 \le i \le j \le k \le 3} x_i x_j x_k.
\]
Using these, we also have:
\(
f = \frac{4}{3} h_1^3 - h_1 h_2 + \frac{1}{3} h_3.
\)
Finally, the power sum symmetric polynomials are:
\[
p_1 = x_1 + x_2 + x_3 = e_1, \quad
p_2 = x_1^2 + x_2^2 + x_3^2, \quad
p_3 = x_1^3 + x_2^3 + x_3^3,
\]
and we find
\(
f = p_1 p_2 - p_3.
\)
\end{example}

In general, let \( \calG \) be a finite matrix group acting on \( \K[x_1, \dots, x_n] \). The  invariant ring of \( \calG \) is the ring \( \K[x_1, \dots, x_n]^\calG \) consisting of all polynomials invariant under the action of \( \calG \). Then, there exist homogeneous, algebraically independent polynomials \(G = (g_1, \dots, g_n)\) and homogeneous invariants \(\sigma = (\sigma_1, \dots, \sigma_r)\) such that:
\[
\K[x_1, \dots, x_n]^\calG = \bigoplus_{j=1}^r \K[g_1, \dots, g_n] \cdot \sigma_j.
\]
Here, \(g_1, \dots, g_n\) and \(\sigma_1, \dots, \sigma_r\) are called the primary and secondary invariants, respectively. In other words, any \(h \in \K[x_1, \dots, x_n]^\calG\) can be uniquely written as
\begin{equation} \label{eq:second-first}
h = \sum_{j=1}^r F_j(g_1, \dots, g_n) \cdot \sigma_j,
\end{equation}
for some \(F_j \in \K[y_1, \dots, y_n]\). This decomposition is know as \emph{Hironaka decomposition of the invariant ring}.

When \(\calG\) is a finite pseudo-reflection group, the invariant ring \(\K[x_1, \dots, x_n]^\calG\) is generated by a finite set of algebraically independent polynomials \(g_1, \dots, g_n\), known as the basic invariants. That is,
\[
\K[x_1, \dots, x_n]^\calG = \K[g_1, \dots, g_n],
\]
or equivalently, any \(f \in \K[x_1, \dots, x_n]^\calG\) can be uniquely expressed as \(F(g_1, \dots, g_n)\) for some \(F \in \K[y_1, \dots, y_n]\). A specific case is when \(\calG = \S_n\), and \(g_1, \dots, g_n\) are the elementary, power sum (in characteristic zero), or complete homogeneous symmetric polynomials.

\smallskip

More generally, consider \(n\) algebraically independent elements \(g_1, \dots, g_n\) in \(\K[x_1, \dots, x_n]\). Let \(\K[g_1, \dots, g_n]\) be the subring of \(\K[x_1, \dots, x_n]\) generated by these elements. Then, for any \(f \in \K[g_1, \dots, g_n]\), there exists a unique polynomial \(F \in \K[y_1, \dots, y_n]\) such that \(f = F(g_1, \dots, g_n)\). Determining \(F\) given \(f\) and the generators \(g_1, \dots, g_n\) is known as the subfield membership problem (see \cite{gathen2003multivariate,sweedler1993using}), which is generally quite challenging. Table \ref{tab:F-computation-strategies} gives a summary for the algorithms and their complexities for this problem. 

\begin{pb}[Representation in Subring]
Let \(g_1, \dots, g_n \in \K[x_1, \dots, x_n]\) be \(n\) algebraically independent polynomials. Given any \(f \in \K[g_1, \dots, g_n]\), compute the unique polynomial \(F \in \K[y_1, \dots, y_n]\) such that
\(
f = F(g_1, \dots, g_n).
\)
\end{pb}

One can always use Gröbner bases to compute \(F\). Specifically, consider the polynomial ring \(\K[x_1, \dots, x_n, y_1, \dots, y_n]\), and fix a monomial order \(\succ\) such that any monomial involving one of the \(x_i\) variables is greater than all monomials in \(\K[y_1, \dots, y_n]\).  
Let \(\mathcal{B}\) be a Gröbner basis with respect to \(\succ\) for the ideal \(\langle y_1 - g_1, \dots, y_n - g_n \rangle \subset \K[x_1, \dots, x_n, y_1, \dots, y_n]\). Then \(f\) can be obtained as the remainder of \(h\) upon division by \(\mathcal{B}\). However, the complexity of this procedure has not been fully analyzed (the same holds for the process presented in Example~\ref{ex:FTSF}).  
For example, with \(f\) as in Example~\ref{ex_10}, one can use the following \textsc{Maple} code to compute the polynomial \(F\) in terms of the power sums.
\begin{codebox}
    \begin{verbatim}
with(Groebner):
F := z:
p1 := x1 + x2 + x3: 
p2 := x1^2 + x2^2 + x3^2: 
p3 := x1^3 + x2^3 + x3^3:
G := [y1 - p1, y2 - p2, y3 - p3, z - f]:
vars := [x1, x2, x3, y1, y2, y3, z]:
GB := Basis(G, plex(x1, x2, x3, y1, y2, y3, z)):
select(has, GB, z);

>> z - y1*y2 - y3
    \end{verbatim}
\end{codebox}

Another straightforward strategy is to use linear algebra by exploiting the symbolic equality \(f = F(g_1, \dots, g_n)\). By comparing the coefficients on both sides of this equality, one obtains a linear system of \(\binom{n+\Delta}{n}\) equations in \(\binom{n+\Delta}{n}\) unknowns, namely, the coefficients of \(F\). The complexity of solving for \(F\) using this method is \(\softO\big(\binom{n+\Delta}{n}^\omega\big)\) operations in \(\K\), where \(\Delta\) is a degree bound for \(F\) and \(\omega\) is the exponent for linear algebra operations.

When \(f\) is invariant under the action of the symmetric group and \((g_1, \dots, g_n)\) are the elementary symmetric polynomials, Gaudry, Schost, and Thiéry showed in \cite{gaudry2006evaluation} that the complexity of evaluating \(F\) is bounded by \(\delta(n) L_2 + 2\) operations in \(\K\), where \(L_2\) is the cost of evaluating \(h\), and \(\delta(n) \leq 4^n (n!)^2\). The authors noted that it is unknown whether the factor \(\delta(n)\), which grows super-polynomially with \(n\), is optimal. In the same context of symmetric polynomials where the \(g_i\) are themselves symmetric, Bläser and Jindal \cite{BlaserJindal18} proved that the complexity of evaluating \(f\) is bounded by \(\softO(d^2 L_2 + d^2 n^2)\) operations in \(\K\), where \(d\) is the degree of \(f\). Recall that \((x_1, \dots, x_n)\) are the roots of the univariate polynomial \(P(T)\) defined in Eq.~\eqref{eq:vieta}. Their main idea is to apply Newton iteration on \(P(T)\) (after a suitable shift) to construct the required power series root \(v_i\) in \(\K[[y_1, \dots, y_n]]\). Building on this, Chaugule et al.~\cite[Theorem 4.16]{chaugule2023schur} extended the Bläser--Jindal result to other bases, such as the homogeneous and power-sum bases, while maintaining the same complexity bound of \(\softO(d^2 L_2 + d^2 n^2)\) operations in \(\K\).

Most recently, Vu \cite{vu2025} proposed a randomized algorithm, along with a complete complexity analysis, for computing \(F\) from \(f\) and \(g_1, \dots, g_n\). Let \(\mathcal{M}(\Delta, n)\) denote the cost of multiplying \(n\)-variate power series up to total degree \(\Delta\). Then the algorithm in \cite{vu2025} achieves complexity  
\[
  \softO\big((nL_1 + n^4 + L_2) \cdot \mathcal{M}(\Delta, n)\big)
\]
operations in \(\K\), where \(L_1\) and \(L_2\) are the lengths of the straight-line programs computing \((g_1, \dots, g_n)\) and \(f\), respectively, and \(\Delta\) is a degree bound on \(F\).

\begin{table}[h]
\centering
\caption{Complexity of Different Strategies for Computing \(F\)}
\renewcommand{\arraystretch}{1.3}
\begin{tabular}{|>{\centering\arraybackslash}p{3.5cm}|
                >{\centering\arraybackslash}p{4cm}|
                >{\centering\arraybackslash}p{4cm}|}
\hline
\textbf{Strategy} & \textbf{Applicable To} & \textbf{Complexity} \\
\hline
Gröbner basis & General & Not fully analyzed \\
\hline
Linear algebra & General & \(\tilde{O}\left(\binom{n + \Delta}{n}^\omega\right)\) \\
\hline
Gaudry et al.\ (2006) & \(f\) symmetric, \(g_i = e_i\) & \(\delta(n) L_2 + 2\) \\
\hline
Bläser–Jindal (2018) & \(f\) symmetric, \(g_i = e_i\) & \(\tilde{O}(d^2 L_2 + d^2 n^2)\) \\
\hline
Chaugule et al.\ (2023) & \(f\) symmetric, various bases & \(\tilde{O}(d^2 L_2 + d^2 n^2)\) \\
\hline
Vu (2025) & General & \(\tilde{O}\big((nL_1 + n^4 + L_2)\, \mathcal{M}(\Delta, n)\big)\) \\
\hline
\end{tabular}
\label{tab:F-computation-strategies}
\end{table}

We conclude with the classical result that provides a degree bound for expressing a symmetric polynomial in terms of the power sum symmetric functions.
\begin{proposition}[{\cite[Chapter I, §2, Theorem 2.3]{Macdonald1995}}]  \label{lm:deg_restrict}
Let \(\K\) be a field of characteristic zero. Then any symmetric polynomial \( f \in \K[x_1, \dots, x_n]^{\mathfrak{S}_n} \) of degree \( d \leq n \) can be uniquely expressed as
\[
f = F(p_1, \dots, p_d),
\]
where \( F \in \K[y_1, \dots, y_d] \) is a polynomial, and \( p_k = \sum_{i=1}^n x_i^k \) denotes the \(k\)-th power sum symmetric polynomial.
\end{proposition}

\begin{remark} \label{rk-el}
More generally, the same result holds over any field \( \K \), without any assumption on the characteristic, provided that the symmetric polynomial \( f \in \K[x_1, \dots, x_n]^{\mathfrak{S}_n} \) is expressed in terms of any algebraically independent family of symmetric generators, such as the elementary symmetric polynomials \( e_1, \dots, e_n \), or the complete homogeneous symmetric polynomials \( h_1, \dots, h_n \), with degree sequence \( \{1, \dots, n\} \). We refer the reader to \cite[Section 4]{vu2025} for degree bounds on \( F \) in the general setting where the generators are arbitrary algebraically independent   polynomials \( g_i \)'s.
\end{remark}

\subsubsection{Block Symmetric Polynomials.} \label{sec:block}
A block symmetric polynomial is a polynomial that is symmetric within disjoint subsets (blocks) of variables, but not necessarily symmetric across blocks.

Formally, let \(r\) be a positive integer. For \(1 \le k \le r\), let \(\X_k = (x_{k, 1}, \dots, x_{k, \ell_k})\) be a set of \(\ell_k\) indeterminates. The group \(\S_{\ell_1} \times \cdots \times \S_{\ell_r}\) acts naturally on the polynomial ring \(\K[\X_1, \dots, \X_r]\). We can extend the classical Fundamental Theorem of Symmetric Polynomials (FTSP) in the following way: any \(\S_{\ell_1} \times \cdots \times \S_{\ell_r}\)-invariant polynomial \(f \in \K[\X_1, \dots, \X_r]\) can be expressed uniquely as a polynomial in \(\left(e_{k,1}(\X_k), \dots, e_{k,\ell_k}(\X_k)\right)_{1 \le k \le r}\), where \(e_{k,j}(\X_k)\) denotes the elementary symmetric polynomial of degree \(j\) in the variables of \(\X_k\).

In other words, there exists a unique polynomial \(F \in \K[\bm Y_1, \dots, \bm Y_r]\), where \(\bm Y_k = (y_{k,1}, \dots, y_{k,\ell_k})\), such that
\[
f = F(e_{1, 1}(\X_1), \dots, e_{1, \ell_1}(\X_1), \dots, e_{r, 1}(\X_r), \dots, e_{r, \ell_r}(\X_r)).
\]

\begin{example}
The polynomial
\[
f(x, y, z, a, b) = (x + y + z)^3 + (a^2 + b^2)
\]
is symmetric in each block separately (i.e., symmetric in \(x, y, z\) and in \(a, b\)), but not necessarily symmetric across the two blocks. For this polynomial \(f\), we have
\[
F = y_{1,1}^3 + y_{2,1}^2 - 2y_{2,2}.
\]
\end{example}

Of course, the ring of \(\S_{\ell_1} \times \cdots \times \S_{\ell_r}\)-invariant polynomials is also generated by other standard (block) symmetric polynomial bases such as the complete homogeneous symmetric, power sum symmetric polynomials, and others.

\subsubsection{Orbit Spaces.}In many settings, especially those involving symmetry, it is natural to study a space through the lens of group actions. Given a group \( G \) acting on a space \( X \), the \emph{orbit} of a point \( x \in X \) is the set \[ G \cdot x = \{ g \cdot x \mid g \in G \}, \] consisting of all points obtained by acting on \( x \) with elements of \( G \). The collection of all such orbits forms the \emph{orbit space}, denoted \( X/G \), which captures the geometry of \( X \) modulo the symmetries induced by \( G \). In this quotient space, two points in \( X \) are identified if they lie in the same orbit. Orbit spaces arise naturally in algebraic geometry, differential geometry, and invariant theory, and they play a central role in understanding moduli spaces and classification problems. 

Compositions and partitions of intergers play a central role in understanding and encoding the orbit spaces, especially in more algebraic, combinatorial, and representation theoretic contexts. Partitions classify types of orbits (invariant under permutation), whereas compositions classify positions or arrangements before taking the group action into account.

\begin{definition}[Compositions and Partitions of Integers]
\label{def:composition-order}
A sequence of positive integers $\lambda = (\lambda_1, \ldots, \lambda_\ell)$ with $|\lambda| := \sum_{i=1}^{\ell} \lambda_i = n$ is called a \emph{composition of $n$ into $\ell$ parts}, and we call $\ell$ the \emph{length} of $\lambda$.

We denote by $\Comp(n)$ the set of compositions of $n$, and by $\Comp(n,\ell)$ the set of compositions of $n$ of length $\ell$. For every $d$, we define the set of \emph{alternate odd} compositions $\CompMax(n,d)$ as
\[
\left\{ \lambda = (\lambda_1, \ldots, \lambda_d) \in \Comp(n) \; \middle| \;
\lambda_{2i+1} = 1, \; 0 \leq i < d/2 \right\}.
\]

A composition $\lambda = (\lambda_1, \ldots, \lambda_\ell)$ is called a \emph{partition} of $n$ if $\lambda_1 \le \cdots \le \lambda_\ell$. In this case, we may also denote the partition $\lambda$ by $\lambda = (n_1^{\ell_1}\, \dots\, n_r^{\ell_r})$, where
\[
n = n_1 \ell_1 + \cdots + n_r\, \ell_r \quad \text{and} \quad \ell = \ell_1 + \cdots + \ell_r.
\]
\end{definition}

\begin{example}
    Consider $n = 4$. There exist eight compositions of $4$, namely:
    \[
    (4),\; (3,1),\; (1,3),\; (2,2),\; (2,1,1),\; (1,2,1),\; (1,1,2),\; (1,1,1,1).
    \]
    Among these, the five partitions are:
    \begin{align*}
        (4) = (4^1), 
        (3,1) = (3^1\, 1^1),
        (2,2) = (2^2), 
        (2,1,1) = (2^1\, 1^2),
        (1,1,1,1) = (1^4).
    \end{align*}
    For the compositions of length $3$, 
    \(
    \Comp(4,3) = \{(2,1,1), (1,2,1), (1,1,2)\}.
    \)
    For $d = 3$, the set of alternate odd compositions is
    \(
    \CompMax(4,3) = \{(1,2,1)\}.
    \)
\end{example}
\begin{definition}[Weyl Chambers] \label{def:Weyl}
We denote by $\W_c$ the cone defined by the inequalities
\[
x_1 \leq x_2 \leq \cdots \leq x_n,
\]
called the \emph{canonical Weyl chamber} for the action of the symmetric group $\S_n$. This region is a fundamental domain for the $\S_n$-action on $\R^n$ by coordinate permutation. The \emph{walls} of $\W_c$ consist of those points $\a = (a_1, \ldots, a_n) \in \W_c$ where the inequalities are not strict at some positions, i.e., $a_i = a_{i+1}$ for some $i \in \{1, \ldots, n-1\}$.

For $n \in \N$ and a composition $\lambda = (\lambda_1, \ldots, \lambda_\ell) \in \Comp(n)$, we define the subset $\W_c^{\lambda} \subseteq \W_c$ by
\[
x_1 = \cdots = x_{\lambda_1} \leq x_{\lambda_1+1} = \cdots = x_{\lambda_1+\lambda_2}
\leq \cdots \leq x_{\lambda_1+\cdots+\lambda_{\ell-1}+1} = \cdots = x_n.
\]
For every $\lambda \in \Comp(n,d)$, the set $\W_c^\lambda$ defines a $d$-dimensional face of the cone $\W_c$, and every face arises in this way from a composition.
We denote by $L_\lambda$ the linear span of $\W_c^\lambda$. Then
\[
\dim L_\lambda = \dim \W_c^\lambda = \length(\lambda).
\]
\end{definition}
\begin{example}
    For \(n = 3\), the canonical Weyl chamber \(\W_c \subset \R^3\) is stratified as follows:
    \begin{itemize}
        \item \(W_c^{(1,1,1)}\) consists of all points \((x_1, x_2, x_3) \in \R^3\) such that \(x_1 < x_2 < x_3\).
        \item \(W_c^{(2,1)}\) consists of all points \((x_1, x_2, x_3) \in \R^3\) such that \(x_1 = x_2 < x_3\).
        \item \(W_c^{(3)}\) consists of all points \((x_1, x_2, x_3) \in \R^3\) such that \(x_1 = x_2 = x_3\).
    \end{itemize}
    Each of these strata corresponds to a face of the canonical Weyl chamber, with lower dimensions as more equalities among the coordinates are imposed.
\end{example}

Given a general point \( \a \in \mathbb{R}^n \), there is a unique element \( \a' \in \mathcal{W}_c \) in the \(\S_n \)-orbit of \( \a \). This \( \a' \) is obtained by ordering the coordinates of \( \a \) in increasing order. Algorithmically, this can be done using the classical bubble sort algorithm. By ordering the coordinates, we have thus a unique representative of each orbit, and the Weyl chamber can thus be seen as a model of the orbit space of $\S_n$. 

Moreover, bubble sort can also be used to find the unique permutation \( \sigma \in \S_n \) that transforms \( \a \) into \( \a' \). Since \( \sigma \in \S_n \), it can be expressed as a product of pairwise adjacent transpositions \( s_i = (i, i+1) \) for \( i \in \{1, \ldots, n-1\} \). We will use the following variant of bubble sort to obtain the minimal set of adjacent transpositions necessary to describe \( \sigma \).

\begin{algorithm}
\caption{\textsc{Minimal\_Adjacent\_Transpositions}}
\label{algo:bubble}
\begin{algorithmic}[1]
\Require \( \a = (a_1, \dots, a_n) \in \mathbb{R}^n \)
\Ensure A minimal sequence \( T \) of adjacent transpositions sorting \( \a \), and the sorted vector \( \a' \)

\State Initialize an empty list \( T \leftarrow [] \)
\For{\( i = 1 \) to \( n-1 \)}
    \For{\( j = 1 \) to \( n - i \)}
        \If{\( a_j > a_{j+1} \)}
            \State Swap \( a_j \leftrightarrow a_{j+1} \)
            \State Append transposition \( (j, j+1) \) to \( T \)
        \EndIf
    \EndFor
\EndFor
\State \Return \( (T, \a) \)
\end{algorithmic}
\end{algorithm}

\begin{theorem}
The above algorithm correctly computes a minimal decomposition of the permutation \(\sigma\) into adjacent transpositions. Its time complexity is \(O(n^2)\).
\end{theorem}

\begin{proof}
Correctness follows from the fact that each adjacent transposition reduces the inversion number of the permutation by exactly one. Since the total number of inversions corresponds to the minimal number of adjacent transpositions needed to sort the permutation, this procedure yields a minimal decomposition. 

The complexity result follows immediately because the algorithm is essentially bubble sort, which is known to have a worst-case time complexity of \(O(n^2)\).
\end{proof}

\begin{remark}
Although there exist sorting algorithms with better time complexity, Bubble sort is chosen here because the goal of Algorithm \ref{algo:bubble} is not just to sort the vector, but to obtain a minimal sequence of \emph{adjacent} transpositions that transforms the original vector into its sorted form. More efficient sorting algorithms often perform swaps of non-adjacent elements, which do not correspond to adjacent transpositions. Thus, the inherent complexity of finding a minimal adjacent transposition decomposition imposes a lower bound of \(O(n^2)\) on the time complexity.
\end{remark}

\begin{definition}[Orders on compositions]
    Let \(n \in \N\) and \(\lambda, \lambda'\) be two compositions of \(n\). Then we denote \(\lambda \prec \lambda'\) if \(\W_c^\lambda \supset \W_c^{\lambda'}\). Equivalently,  $\lambda  \prec  \lambda'$ if $\lambda'$ can be obtained from $\lambda$ by replacing some of the commas in $\lambda$ by $+$ signs. 

    If \(\lambda, \lambda'\) are two compositions of \(n\), then the smallest composition that is bigger than \(\lambda\) and \(\lambda'\) is called the \emph{join} of \(\lambda\) and \(\lambda'\).  The face corresponding to the join of $\lambda$ and $\lambda'$ is $\W_c^\lambda\cap\W_c^{\lambda'}$.
\end{definition}
    \begin{example}
        For  any $\lambda' \in \{(3,1), (1,3), (4)\}$, we have  $(1,2,1) \prec \lambda'$. Furthermore, the join of
$(2,1,1)$ and $(1,1,1,1)$ is $(2,1,1)$. 
    \end{example}
\begin{definition}[Largest Composition] \label{def_largest}
For $\u \in \W_c \cap \R^n$, $\comp(\u)$ denote the largest composition $\lambda = (\lambda_1, \dots, \lambda_\ell)$ of $n$, such that $\u$ can be written as 
    \[ \u = (\underbrace{c_1, \dots, c_1}_{\lambda_1\text{-times}},
    \underbrace{c_2, \dots, c_2}_{\lambda_2\text{-times}}, \dots,
    \underbrace{c_\ell, \dots, c_\ell}_{\lambda_\ell\text{-times}}).\] 
\end{definition}

\begin{definition}[Refinement Orders on Partitions]
    Let \(n, n' \in \N\) and \(\lambda, \lambda'\) be  partitions of \(n\) and \(n'\) respectively. Then  \(\lambda \cup \lambda'\) is the partition of \(n + n'\) whose ordered list is obtained by merging those of \(\lambda\) and \(\lambda'\).  

    Let \(\lambda = (n_1^{\ell_1} \dots n_r^{\ell_r})\) and \(\lambda' = (m_1^{k_1} \dots m_r^{k_r})\) be partitions of the same integer \(n\). We say \(\lambda\) \emph{refines} \(\lambda'\), denoted by \(\lambda \le \lambda'\), if \(\lambda\) is the union of some partitions \((\lambda_{i,j})_{1 \le i \le s, 1 \le j \le k_i}\), where \(\lambda_{i,j}\) is a partition of \(m_i\) for all \(i,j\). 
\end{definition}

A simple way to parameterize $\S_n$-orbits in $\K^n$ is through the use of partitions of $n$. For a partition $\lambda = (n_1^{\ell_1}\, \dots\, n_r^{\ell_r})$ of $n$, we define the set $U_\lambda \subseteq \K^n$ to consist of all vectors $\bm{u}$ that can be written in the form:
\begin{multline} \label{eq:type}
\bm{u} = (\, \underbrace{u_{1,1}, \dots, u_{1,1}}_{n_1},~ 
    \dots,~ \underbrace{u_{\ell_1,1}, \dots, u_{\ell_1,1}}_{n_1},~
    \dots,\\
    \underbrace{u_{1,r}, \dots, u_{1,r}}_{n_r},~
    \dots,~ \underbrace{u_{r, \ell_r}, \dots, u_{r, \ell_r}}_{n_r} \,).
\end{multline}
We denote by $U_\lambda^{\rm strict} \subset U_\lambda$ the subset consisting of those $\bm{u}$ for which all the values $u_{i,j}$ appearing in eq.\eqref{eq:type} are pairwise distinct.

\begin{definition}[Type of a point and of an \(\S_n\)-orbit]
For any point $\bm{u} \in \K^n$, its \emph{type} is the unique partition $\lambda$ of $n$ such that there exists a permutation $\sigma \in \S_n$ with $\sigma(\bm{u}) \in U_\lambda^{\rm strict}$. Since all points in an orbit share the same type, we also define the \emph{type of an orbit} to be the type of any of its points.
\end{definition}

Clearly, all points in \(U_\lambda^{\rm strict}\) have type \(\lambda\). However, this does not hold for the entirety of \(U_\lambda\), which also contains points of type \(\lambda'\) for all \(\lambda' \ge \lambda\) in the dominance order. In fact,
\[
U_\lambda = \bigsqcup_{\lambda' \ge \lambda} U_{\lambda'}^{\rm strict},
\]
the disjoint union of all \(U_{\lambda'}^{\rm strict}\) for partitions \(\lambda' \ge \lambda\).

A point of type $\lambda = (n_1^{\ell_1}\, \dots \, n_r^{\ell_r})$ is stabilized by the subgroup
\[
\S_{n_1}^{\ell_1} \times \cdots \times \S_{n_r}^{\ell_r} \subseteq \S_n.
\]
Hence, the size of any orbit of type \(\lambda\) is
\[
\mu_\lambda := \binom{n}{\underbrace{n_1, \dots, n_1}_{\ell_1}, \dots, \underbrace{n_r, \dots, n_r}_{\ell_r}} = \frac{n!}{n_1!^{\ell_1} \cdots n_r!^{\ell_r}}.
\]
\begin{example}
    For the partitions of $n = 3$, we have 
    \((1^3) < (1^1\, 2^1) < (3^1)\). Moreover,
    \begin{itemize}
        \item \(U_{(1^3)}\) is \(\K^3\) and \(U_{(1^3)}^{\rm strict}\) is the set of all points of pairwise distinct coordinates. 
        \item \(U_{(1^1\, 2^1)}\) is the set of all points that can be written as \(\bm u = (u_{1,1}, u_{2,1}, u_{2,1})\) while  \(U_{(1^3)}^{\rm strict}\) is the subset of \(U_{(1^1\, 2^1)}\) where \(u_{1,1} \ne u_{2,1}\). 
        \item \(U_{(3^1)}= U_{(3^1)}^{\rm strict}\) is the set of points whose coordinates are all equal.
    \end{itemize}
    In addition, \(U_{(1^3)} = U_{(3^1)}^{\rm strict} \bigsqcup  U_{(1^1 \, 2^1)}^{\rm strict}  \bigsqcup U_{(1^3)}^{\rm strict}.  \)
\end{example}
For a partition \(\lambda = (n_1^{\ell_1}\, \dots \, n_r^{\ell_r})\) of length \(\ell\), we define a mapping
\(
E_\lambda : U_\lambda \rightarrow \C^\ell
\)
by
\begin{equation} \label{eq:comp}
E_\lambda \colon \u \text{ as in } \eqref{eq:type} \mapsto \left( e_{i,1}(u_{i,1}, \dots, u_{i,\ell_i}), \dots, e_{i,\ell_i}(u_{i,1}, \dots, u_{i,\ell_i}) \right)_{1 \le i \le r},
\end{equation}
where, for each \(i = 1, \dots, r\) and \(j = 1, \dots, \ell_i\), \(e_{i,j}(u_{i,1}, \dots, u_{i,\ell_i})\) denotes the degree \(j\) elementary symmetric polynomial in \(u_{i,1}, \dots, u_{i,\ell_i}\).

We can interpret this mapping as an orbit compression. Through the application of \(E_\lambda\), one can represent an entire orbit \(\mathcal{O}\) of type \(\lambda\), which has size \(\mu_\lambda\), by a single point:
\[
E_\lambda(\mathcal{O} \cap U_\lambda) = E_\lambda(\mathcal{O} \cap U_\lambda^{\text{strict}}).
\]
Conversely, we can recover an orbit from its image. Given \(\bm \varepsilon = (\varepsilon_{1,1}, \dots, \varepsilon_{r, \ell_r}) \in \C^\ell\), define the Vieta polynomials \(P_1(T), \dots, P_r(T)\) as
\[
P_i(T) = T^{\ell_i} - \varepsilon_{i,1} T^{\ell_i - 1} + \cdots + (-1)^{\ell_i} \varepsilon_{i,\ell_i}, \quad \text{for } 1 \le i \le r.
\]
We can then reconstruct a point \(\bm u\) in the preimage \(E_\lambda^{-1}(\bm \varepsilon)\) by computing the roots \(u_{i,1}, \dots, u_{i,\ell_i}\) of \(P_i(T)\) for each \(i\).
\subsubsection{A Decision Procedure.}
Consider a partition $\lambda = (n_1^{\ell_1}\, \dots \, n_r^{\ell_r})$ of
$n$, and let 
\[\mathscr{R} = ((q,  v_{1,1}, \dots, v_{1, \ell_1}, \dots,
v_{r, 1},\dots, v_{r, \ell_r},\gamma)\]
be a parametrization which encodes a finite set $W_\lambda\subset \C^\ell $. This set lies in the target space of the algebraic map $E_\lambda : U_\lambda \rightarrow \Kbar^\ell$ defined in \eqref{eq:comp}. Let \(V_\lambda\) be the preimage of $W_\lambda$  by $E_\lambda$. In
this subsection we present a procedure called \textsc{Decision}($\mathscr{R}_\lambda$) which takes as input \textsc{$\mathscr{R}_\lambda$}, and
decides whether the set $V_\lambda$ contains real points. 

In order to do this, a straightforward strategy consists in solving the polynomial system to invert the map $E_\lambda$. Because of the group action of $\S_{\ell_1}\times \cdots \times \S_{\ell_r}$, we would then obtain $\ell_1 !\cdots \ell_r! $ points in the preimage of a single point in $W_\lambda$; therefore we would lose the benefit symmetry.

This difficulty can be bypassed by encoding one single point per orbit in the preimage of the points in $W_\lambda$. This can be done via the following steps.

\begin{algorithm}
\caption{\textsc{Decide}}
\begin{algorithmic}[1]
\Require A parametrizations $\scrR_\lambda = (q, q', {v}_{i,j})_{1 \le i \le r, 1 \le j \le \ell_i}$ defining a finite set $W_\lambda$
\Ensure \true \ if all fibers of \(W_\lambda\) by the map \(E_\lambda\) contain  real points, \false \ otherwise

\State Group  variables $\bm y_i = (y_{i,1}, \ldots, y_{i,\ell_i})$ corresponding to symmetric functions $e_{i,j}$
\State Denote by $v_{i,j}$ the parametrizations corresponding to $e_{i,j}$

\For{$i = 1$ to $r$}
  \State Construct the Vieta polynomial:
  \[
    \rho_i(u, T) = q' \cdot u^{\ell_i} - v_{i,1} \cdot u^{\ell_i - 1} + \cdots + (-1)^{\ell_i} v_{i,\ell_i} \in \K[T][u]
  \]
  \State Compute the Sturm-Habicht sequence of $\left(\rho_i, \frac{\partial \rho_i}{\partial u} \right)$ in $\K[T]$ \label{step:th}
  \State Compute Thom encodings of real roots of $q$ using \textsc{Thom\_Encoding}  \label{step:th_2}
 
  \For{each real root $\vartheta$ of $q$}
    \State Evaluate the signed subresultant sequence of $\rho_i$ at $T = \vartheta$
    \State Compute the number of real roots of $\rho_i(\vartheta, u)$ using the Cauchy index via \textsc{Sign\_ThomEncoding} \label{step:th_3}
    \If{the number of real roots is less than $\ell_i$}
      \State \Return \false
    \EndIf
  \EndFor
\EndFor

\State \Return \true
\end{algorithmic}
\end{algorithm}

\begin{theorem} \label{thm:dec}
    Let \(\lambda = (n_1^{\ell_1}\, \dots\, n_r^{\ell_r})\) be a partition of \(n\) of length \(\ell\) and \(\mathscr{R}_\lambda= ((q,  v_{i,j}),\gamma)_{1 \le i \le r, 1 \le j \le \ell_i}\) a zero-dimensitional parametrization of \(W_\lambda \subset \C^\ell\) of size \(a \).   
    The complexity of the \textsc{Decide}(\(\mathscr{R}_\lambda\)) is  \(O((n^4 a + n a^3))\) operations in \(\Q\). The output of the algorithm is yes if \(E^{-1}_\lambda(W_\lambda)\) has a real point and no otherwise, where  $E_\lambda$  is the algebraic map defined in \eqref{eq:comp}. 
\end{theorem}
\begin{proof}
    Note that \(\deg(q) = a\). Let \(b\) be the maximum of the partial degrees of \(\rho_i\)'s with respect to \(u\).  By the complexity analysis of
 \cite[Algorithm 8.21 - Section 8.3.6]{BPR06}, Step {\bf \ref{step:th}.}  is performed within
 $O\left( b^{4}a \right)$ arithmetic operations in $\K$ using a classical
 evaluation interpolation scheme (there are $b$ polynomials to interpolate, all
 of them being of degree is at mot $2 a b$). 

  Step {\bf \ref{step:th_2}.}  requires $O\left(
   a^{4}\log(a) \right)$ arithmetic operations in $\K$ (see  Section \ref{subs_realpoint} or  the complexity
 analysis of \cite[Algorithm 10.14 - Section 10.4]{BPR06}). Finally, in Step {\bf \ref{step:th_3}.}, we
 evaluate the signs of $b$ polynomials of degree $\leq 2ab$ at the real roots of
 $v$ (of degree $a$) whose Thom encodings were just computed. This is performed
 using $O\left(a^{3}\left( (\log(a) + b) \right) \right)$ arithmetic
 operations in $\K$ following the complexity analysis of \cite[Algorithm 10.15 - Section
 10.4]{BPR06}. The sum of these estimates lies in $O\left(b^{4}a + 
   a^{3}\left( (\log(a) + b) \right) \right)$. 

    Finally, since the degree of \(\rho_i\) with respect to \(u\) equals \(\ell_i\) and \(\ell_i \le n\). This means \(b \le n\). All in all, we deduce that the total cost of the \textsc{Decide} algorithm is \(O(n^4 a + n a^3)\) operation in \(\K\). 
    \end{proof}

\subsubsection{Some Homomorphisms on Polynomial Rings.}
\label{sec_hom}
Let \(\lambda = (\lambda_1, \dots, \lambda_\ell)\) be a composition of \(n\), and let \(f \in \K[x_1, \dots, x_n]\) be a polynomial. We define the polynomial
\[
f^{[\lambda]} := f(\underbrace{x_1, \dots, x_1}_{\lambda_1\text{ times}}, \underbrace{x_2, \dots, x_2}_{\lambda_2\text{ times}}, \dots, \underbrace{x_\ell, \dots, x_\ell}_{\lambda_\ell\text{ times}}).
\]
Note that \(f^{[\lambda]}\) is a polynomial in \(\ell\) variables, and its image coincides with the image of \(f\) when restricted to the subset \(\W_c^\lambda\).

When \(\lambda = (\lambda_1, \dots, \lambda_\ell) = (n_1^{\ell_1}\, \dots \, n_r^{\ell_r})\) is a partition of \(n\), for \(1 \le k \le r\), let \(\X_k = (x_{k,1}, \dots, x_{k,\ell_k})\) be a set of \(\ell_k\) indeterminates as defined in Section~\ref{sec:block}. Then,
\[
f^{[\lambda]} = f\big(
\underbrace{x_{1,1}, \dots, x_{1,1}}_{n_1\text{ times}}, \dots, \underbrace{x_{1,\ell_1}, \dots, x_{1,\ell_1}}_{n_1\text{ times}}, \dots, 
\underbrace{x_{r,1}, \dots, x_{r,1}}_{n_r\text{ times}}, \dots, \underbrace{x_{r,\ell_r}, \dots, x_{r,\ell_r}}_{n_r\text{ times}}
\big).
\]
 The following is straightforward.

\begin{lemma}
Let \(f\) be a symmetric polynomial in \(\K[x_1, \dots, x_n]\), and let \(\lambda\) be a partition of \(n\). Then \(f^{[\lambda]}\) is \(\S_{\ell_1} \times \cdots \times \S_{\ell_r}\)-invariant.
\end{lemma}

Fix a partition \(\lambda = (n_1^{\ell_1}\, \dots \, n_r^{\ell_r})\) of \(n\), and let \(\ell\) be its length. Define
\[
I_{i,j} := \{\sigma_{i,j}+1, \dots, \sigma_{i,j}+n_i\}, \quad 1 \le i \le r, \ 1 \le j \le \ell_i,
\]
with \(\sigma_{i,j} := \sum_{k=1}^{i-1} \ell_k n_k + (j-1)n_i\). The variables \(x_m\) for \(m \in I_{i,j}\) are precisely those that map to \(x_{i,j}\).

Define the matrix \(\bm D \in \Q^{\ell \times n}\), where \(\ell = \ell_1 + \cdots + \ell_r\), with rows indexed by pairs \((i,j)\) and columns by \(m \in \{1, \dots, n\}\). For each \((i,j)\), the entry in row \((i,j)\) and column \(m \in I_{i,j}\) is set to \(1/n_i\); all other entries are zero. That is,
\[
\bm D = \mathrm{diag}(\bm D_1, \dots, \bm D_r), \quad \text{where each } \bm D_i \in \Q^{\ell_i \times n_i\ell_i} \text{ is of the form}
\]
\[
\bm D_i = \left( 
\begin{matrix}
\begin{matrix}\frac{1}{n_i} & \cdots & \frac{1}{n_i} \end{matrix} & {\bf
    0} & \cdots & {\bf 0}\\
{\bf 0} &\begin{matrix}\frac{1}{n_i} & \cdots &
  \frac{1}{n_i} \end{matrix} & \cdots & {\bf 0} \\
\vdots &  &\ddots & \vdots \\
{\bf 0} & {\bf 0} & \cdots & \begin{matrix}\frac{1}{n_i} & \cdots &
  \frac{1}{n_i} \end{matrix} 
\end{matrix} 
\right).  
\]

\begin{example}\label{ex:matZ}
Consider the partition \(\lambda = (2^2 \, 3^1)\) of \(n=7\). Then \(n_1 = 2\), \(\ell_1 = 2\); \(n_2 = 3\), \(\ell_2 = 1\), so \(\ell = 3\). The matrix \(\bm D\) becomes:
\[
\bm D = \begin{pmatrix}
\frac{1}{2} & \frac{1}{2} & 0 & 0 & 0 & 0 & 0 \\
0 & 0 & \frac{1}{2} & \frac{1}{2} & 0 & 0 & 0 \\
0 & 0 & 0 & 0 & \frac{1}{3} & \frac{1}{3} & \frac{1}{3}
\end{pmatrix}.
\]
\end{example}

Let \(\f = (f_1, \dots, f_s)\) be a sequence of polynomials in \(\K[x_1, \dots, x_n]\). Denote by \(\jac(\f)\) the Jacobian matrix of \(\f\) with respect to the variables \(x_1, \dots, x_n\).

\begin{lemma}\label{lemma:decom}
Let \(\f = (f_1, \dots, f_s) \subset \K[x_1, \dots, x_n]\) be a sequence of symmetric polynomials, and let \(\lambda\) be a partition of \(n\). Then:
\[
\jac_{x_1, \dots, x_n}^{[\lambda]}(\f) = \jac_{\bm X_1, \dots, \bm X_r}\big(\f^{[\lambda]}\big) \cdot \bm D.
\]
\end{lemma}

\begin{proof}
For any \(f \in \K[x_1, \dots, x_n]\), \(f^{[\lambda]}\) is obtained by evaluating \(f\) at \(x_m = x_{i,j}\) for \(m \in I_{i,j}\). By the chain rule,
\[
\frac{\partial f^{[\lambda]}}{\partial x_{i,j}} = \sum_{m \in I_{i,j}} \left( \frac{\partial f}{\partial x_m} \right)^{[\lambda]}.
\]
If \(f\) is symmetric, then \(\left( \frac{\partial f}{\partial x_m} \right)^{[\lambda]} = \left( \frac{\partial f}{\partial x_{m'}} \right)^{[\lambda]}\) for all \(m, m' \in I_{i,j}\). Hence, for all \(m \in I_{i,j}\),
\[
\left( \frac{\partial f}{\partial x_m} \right)^{[\lambda]} = \frac{1}{n_i} \frac{\partial f^{[\lambda]}}{\partial x_{i,j}}.
\]
This generalizes to the Jacobian matrix for the sequence \(\f\).
\end{proof}

\begin{lemma} \label{lemma:extend_slice}
Let \(\f = (f_1, \dots, f_s)\) be a sequence of symmetric polynomials such that \(\jac(\f)\) has rank \(s\) at every point of \(V(\f)\). Then \(\jac(\f^{[\lambda]})\) also has rank \(s\) at every point of \(V(\f^{[\lambda]})\).
\end{lemma}

\begin{proof}
Let \(\bm \alpha = (\alpha_{1,1}, \dots, \alpha_{1, \ell_1}, \dots, \alpha_{r, 1}, \dots, \alpha_{r, \ell_r}) \in \Kbar^\ell\) be a zero of \(\f^{[\lambda]}\). Construct the point \(\bm \varepsilon \in \Kbar^n\) as
\[
\bm \varepsilon = (\underbrace{\alpha_{1,1}, \dots, \alpha_{1,1}}_{n_1}, \dots, \underbrace{\alpha_{1, \ell_1}, \dots, \alpha_{1, \ell_1}}_{n_1}, \dots, \underbrace{\alpha_{r, 1}, \dots, \alpha_{r, 1}}_{n_r}, \dots, \underbrace{\alpha_{r, \ell_r}, \dots, \alpha_{r, \ell_r}}_{n_r}).
\]
Then \(\bm \varepsilon \in V(\f)\), and for any \(g \in \K[x_1, \dots, x_n]\), we have \(g^{[\lambda]}(\bm \alpha) = g(\bm \varepsilon)\). In particular,
\[
\jac^{[\lambda]}(\f)(\bm \alpha) = \jac(\f)(\bm \varepsilon).
\]
By Lemma~\ref{lemma:decom}, we also have:
\[
\jac(\f)(\bm \varepsilon) = \jac_{\bm X_1, \dots, \bm X_r}(\f^{[\lambda]})(\bm \alpha) \cdot \bm D.
\]
Since \(\jac(\f)(\bm \varepsilon)\) has rank \(s\), the left kernel of \(\jac_{\bm X_1, \dots, \bm X_r}(\f^{[\lambda]})(\bm \alpha)\) must also be trivial.
\end{proof}

\begin{lemma}\label{lemma:transferreg_ele}
Let \(\bm g = (g_1, \dots, g_s) \subset \K[\X_1, \dots, \X_r]\) be \(\S_{\ell_1} \times \cdots \times \S_{\ell_r}\)-invariant, and assume \(\jac(\bm g)\) has rank \(s\) at every zero of \(\bm g\). Suppose each \(g_i\) can be written as
\[
g_i = G_i(e_{1,1}(\X_1), \dots, e_{1,\ell_1}(\X_1), \dots, e_{r,1}(\X_r), \dots, e_{r,\ell_r}(\X_r)),
\]
for some \(G_i \in \K[\bm Y_1, \dots, \bm Y_r]\). Then the Jacobian matrix \(\jac(\bm G)\) has  rank s at every zero of \(\bm G\). 
\end{lemma}

\begin{proof}
The Jacobian \(\jac(\bm g)\) factors as:
\begin{equation}
    \label{eq:vander}
    \jac(\bm g) = \jac(\bm G)(\bm E) \cdot \bm V, \quad \text{where } \bm V = \mathrm{diag}(V_1, \dots, V_r),
\end{equation}
with \(\bm E = (\bm e_1, \dots, \bm e_r)\) and each \(V_i\) the Jacobian matrix of the elementary symmetric polynomials in \(\X_i\). 
If \(\bm \eta \in V(\bm G)\), and \(\bm \varepsilon \in \bm E^{-1}(\bm \eta)\), then rank deficiency of \(\jac(\bm G)\) at \(\bm \eta\) implies that \(\jac(\bm g)(\bm \varepsilon)\) is also rank-deficient, contradicting the assumption.
\end{proof}

Similarly, instead of using the Jacobian matrix $V_i$ of the elementary symmetric polynomials  as
in \eqref{eq:vander}, we can use $V_i$ as the Jacobian matrix of power sums in $\X_i$. This gives a similar result
but for the polynomials in the power sums.

\begin{lemma} \label{lemma:transferreg}
   Assume \((g_1, \dots, g_s) \subset \K[\X_1, \dots, \X_k]\) is \(\S_{\ell_1} \times \cdots \times \S_{\ell_r}\)-invariant and \(\jac(\bm g)\) has  rank s at each of its zeros. Let \(G_i'\) be the unique expression of \(g_i\) in terms of power sums. Then the Jacobian matrix of \(G_1', \dots, G_s'\) also has rank \(s\) at every point in its vanishing set.
\end{lemma}
\subsubsection{The Vandermonde Maps.} The Vandermonde map is a classical object in algebra and geometry that arises naturally in the study of symmetric functions, hyperbolic polynomials, real algebraic geometry, and invariant theory. 

\begin{definition}
For \( m = (m_1, \ldots, m_n) \in \mathbb{N}^n \), the \emph{weighted power sum} is defined as
\[
p_j^{(m)} := \sum_{i=1}^n m_i x_i^j.
\]
\end{definition}
\noindent Clearly, the standard power sum polynomials correspond to the trivial weight \( m = (1, \ldots, 1) \). Furthermore, given a composition \( \lambda \) of \( n \), the restriction of the \( j \)-th power sum polynomial \( p_j \) to the wall \( W_c^\lambda \), denoted \( p_j^{[\lambda]} \), coincides with the weighted power sum \( p_j^{(\lambda)} \).

\begin{definition}
Let \( d \in \{1, \ldots, n\} \) and \( m \in \mathbb{N}^n \). The \emph{\( d \)-th weighted Vandermonde map} is the function
\[
\nu_{n,d,m} \colon \mathbb{R}^n \to \mathbb{R}^d, \quad \u \mapsto \left(p_1^{(m)}(\u), \ldots, p_d^{(m)}(\u)\right).
\]
When \( m = (1, \ldots, 1) \), we denote the map simply as \( \nu_{n,d} \). When \( d = n \) and \( m = (1, \ldots, 1) \), we refer to it as the (standard) \emph{Vandermonde map}.

For \( \a = (a_1, \ldots, a_d) \in \mathbb{R}^d \), the fiber of the \( d \)-Vandermonde map is the set
\[
V(\a) := \left\{ \u \in \mathbb{R}^n \mid p_1^{(m)}(\u) = a_1, \ldots, p_d^{(m)}(\u) = a_d \right\},
\]
called the \emph{\( m \)-weighted Vandermonde variety} associated to \( \a \). Note that when \( n = d \) and \( m = (1, \ldots, 1) \), every non-empty Vandermonde variety consists of the orbit of a single point \( \u \in \mathbb{R}^n \).
\end{definition}

Note that, in the case when \( m = (1, \dots, 1) \), the fiber is not necessarily the orbit of a point, unless \( d = n \). However, the orbit of a point \( \u \) is always contained in the fiber of \(\nu_{n,d}(\u)\).

\begin{example}
    For \( n = 3 \), \( d = 2 \), and weights \( m = (1,1,1) \), we have
    \[
    \nu_{3,2}(x_1, x_2, x_3) = (p_1(x), p_2(x)) = (x_1 + x_2 + x_3, \, x_1^2 + x_2^2 + x_3^2).
    \]
    The fiber over \((a_1, a_2)\) is
    \[
    V(a_1, a_2) = \{(x_1, x_2, x_3) \in \mathbb{R}^3 \mid x_1 + x_2 + x_3 = a_1, \quad x_1^2 + x_2^2 + x_3^2 = a_2 \}.
    \]
    This is a subset of \(\mathbb{R}^3\) and can be a curve, a point, or empty, depending on whether the system has real solutions.
\end{example}

The significance of the Vandermonde map has been emphasized in the work of Arnold, Givental, and Kostov on hyperbolic polynomials, and has also been further explored in \cite{Blekherman}.

\begin{theorem}[\cite{Blekherman,arnold1986hyperbolic,givental1987moments,kostov1989geometric}]\label{thm:arnold}
The (standard) Vandermonde map \( \nu_{n,n} \) with \( m = (1, \ldots, 1) \) defines a homeomorphism between a Weyl chamber \( \mathcal{W} \) and its image. Moreover, for any composition \( \lambda \), the weighted Vandermonde map \( \nu_{n, \ell(\lambda), \lambda} \) provides a homeomorphism between the associated chamber \( \mathcal{W}_c^\lambda \) and its image. Furthermore, for every \( \a \in \mathbb{R}^d \) and \( m \in \mathbb{Z}^d \), the weighted Vandermonde variety \( V_m(\a) \cap \mathcal{W}_c \) is either contractible or empty.
\end{theorem}

\subsection{Fixed Degree Symmetric Polynomials}
\label{sec_prin}

The fundamental theorem of symmetric polynomials (Theorem \ref{thm:funda}) allows us to transfer algebraic computations from the ring \(\R[x_1,\ldots,x_n]\) to the new ring \(\R[e_1,\ldots,e_n]\) whenever the computation is invariant under the action of \(\S_n\). The fact that the elementary symmetric polynomials are algebraically independent directly reduces the number of variables when the degree is fixed (see Remark \ref{rk-el}). Indeed, if \(f \in \R[x_1,\ldots,x_n]\) is symmetric of degree \(d\), then there exists a unique \(F \in \R[e_1,\ldots,e_d]\) such that
\[
  f(x) = F(e_1(x),\ldots,e_d(x)).
\]
This transfer from \(x_1,\ldots,x_n\) to \(e_1,\ldots,e_d\) can thus drastically simplify the computation when the degree is fixed. We will first see how this plays out in the context of sums-of-squares decompositions for symmetric polynomials.

\subsubsection{Symmetric Sums of Squares.}\label{sec:symsos}

As discussed in the previous section, sums of squares (SOS) provide algebraic certificates of non-negativity and form the basis of powerful semidefinite programming (SDP) relaxations for problems in real algebraic geometry. When the polynomials in question exhibit additional structure, such as symmetry, one can exploit this structure to simplify SOS representations and reduce computational complexity. The key techniques here are based on representation theory and, in particular, {\em Schur's Lemma}. We refer to \cite{riener2013exploiting,debus2020reflection,gatermann2004symmetry,blekherman2021symmetric} for a more general description of the techniques and results in this area. For this chapter, the particular case of a polynomial \( f \) that is \emph{symmetric} is of special interest.

As was remarked previously, algebraic calculations up to symmetry can be more efficiently done in the invariant ring. However, the following example highlights a potential problem.

\begin{example}
The polynomial $f=x_1^2-2x_1x_2+x_2^2$ is symmetric. Since we have $f=(x_1-x_2)^2$ we see that $f$ is also a symmetric sum of squares. However, the polynomial $(x_1-x_2)$ is not symmetric. Therefore, even though $f$ possesses a sum of squares decomposition in the polynomial ring $\R[x_1,x_2]$, it does not have a sum of squares decomposition in the ring of symmetric polynomials. 
\end{example}
The above example highlights that in general invariant sums of squares are not squares of invariant polynomials. To overcome this issue, one needs to characterize invariant  sums of squares algebraically as cone in the invariant ring. We illustrate this firstly  in the case \( n = 2 \) for symmetric polynomials. The key here is that  every polynomial \( q \in \R[x_1, x_2] \) can be uniquely decomposed into symmetric and anti-symmetric parts:
\[
  q = s_1 + (x_1 - x_2)s_2,
\]
where \( s_1 \) and \( s_2 \) are symmetric polynomials.

Now, let \( f \in \R[x_1, x_2] \) be a symmetric SOS polynomial:
\[
  f = \sum_i q_i^2 = \sum_i \big( s_{i,1} + (x_1 - x_2)s_{i,2} \big)^2.
\]
Expanding this gives:
\[
  f = \sum_i \Big( s_{i,1}^2 + 2s_{i,1} s_{i,2} (x_1 - x_2) + s_{i,2}^2 (x_1 - x_2)^2 \Big).
\]
The mixed terms \( s_{i,1} s_{i,2} (x_1 - x_2) \) are anti-symmetric and must cancel in the sum because \( f \) is symmetric. This implies that any symmetric SOS polynomial admits a decomposition of the form
\[
  f = f_1 + (x_1 - x_2)^2 f_2,
\]
where \( f_1 \) and \( f_2 \) are sums of squares of symmetric polynomials.

This idea nicely generalizes to all (compact/reductive) groups. Firstly, in any such case the polynomial ring is a finitely generated module over the invariant ring, i.e., there exist finitely many polynomials $b_1,\ldots, b_m$ such that every polynomial $f\in\R[x_1,\ldots, x_n]$  can be represented as \[f=\sum_{i=1}^m s_ib_i,\] where $s_i$ are invariant polynomials.
Relatively to this generating set  we now construct  a matrix polynomial  with entries 
$$H^S_{u,v}:=\mathcal{R}_G(b_u\cdot b_v),\, \text{ where }1\leq u,v\leq m.$$  Using this construction, an invariant polynomial $f$ is a sum of squares if and only if  $f$ can be written as \[f=\text{Trace} (S(x),H^S(x)),\]
where $S(x)$ is a \emph{sum of squares matrix polynomial} of invariant polynomials, i.e., there exists a matrix $L(x)$ whose entries are invariant polynomials, such that $S(x)=L(x)^tL(x)$. By construction this certificate only involved invariant polynomials and thus can be represented entirely in  the invariant ring - resulting in a reduction of computational efforts.

This construction can be further simplified by carefully selecting the generating set in such a way that it represents the so-called \emph{isotypic decomposition} of the polynomial ring, i.e., the partition into irreducible representations. In this case, results from invariant theory, such as Schur's Lemma, can be further applied to reduce the computational complexity. For example, this allows for dramatic simplifications in the case of the symmetric group: It turns out that for homogeneous symmetric polynomials of fixed degree, the size of the semidefinite program needed to decide whether \( f \) is SOS becomes independent of the number of variables \( n \), once \( n \) exceeds a threshold depending on the degree. This result was first proven in~\cite{riener2013exploiting,rie1}.

\begin{theorem}[Riener et al. ~{\cite{riener2013exploiting,rie1}}]
Let \( H_{n,2d}^{\S_n} \) denote the space of homogeneous symmetric polynomials of degree \( 2d \) in \( n \) variables. Then, for all \( n \geq 2d \), the size of the Gram matrix required to determine whether \( f \in H_{n,2d}^{\S_n} \) is a sum of squares depends only on \( d \), not on \( n \).
\end{theorem}

The above result  relies on the representation theory of the symmetric group \( \S_n \). The key insight is that, for a fixed degree the number of irreducible $S_n$ representations in the space of homogeneous polynomials of a fixed degree stabilizes. On the one hand, their number is fixed and on the other hand their generators are naturally connected. This allows for a \emph{block-diagonalizaton} of the matrix polynomial $H^S$ above which will be of a fixed size for all $n>2d$. 

For example, the following result provides a  representation theorem of symmetric quartics.

\begin{theorem}[Blekherman-Riener \cite{blekherman2012nonnegative}]
Let $f$ be a symmetric quartic polynomial in $n>4$ variables. Then  $f$
is a sum of squares if and only if it can be written in the form
\begin{eqnarray*}
f^{(n)}&=&\alpha_{11}\pi_1^4+2\alpha_{12}\pi_1^2\pi_2+\alpha_{22}\pi_2^2\\
&+&\beta_{11}\left(\pi_1^2\pi_2-\pi_1^4\right)+2\beta_{12}\left(p_{(3,1)}-\pi_1^2\pi_2\right)+\beta_{22}\left(\pi_4-\pi_2^2\right)\\
&+&\gamma\left(\frac{1}{2}
\pi_1^4-\pi_1^2\pi_2+\frac{n^2-3n+3}{2n^2}\pi_2^2+\frac{2n-2}{n^2}
\pi_1\pi_3+\frac{1-n}{2n^2} \pi_4\right),
\end{eqnarray*}
where  $\pi_j=\frac{1}{n}(x_1^j+\ldots x_n^j)$ for $1\leq j\leq 4$  and the parameters $\alpha_{11},\alpha_{12}\,\alpha_{22},\beta_{11},\beta_{12},$ $\beta_{22}$ are chosen 
such that $\gamma\geq 0$ and the matrices $\begin{pmatrix}
\alpha_{11}&\alpha_{12}\\
\alpha_{12}&\alpha_{22}
\end{pmatrix}
$ and $\begin{pmatrix}
\beta_{11}&\beta_{12}\\
\beta_{12}&\beta_{22}
\end{pmatrix}$ are positive semidefinite.
\end{theorem}

This stabilization phenomenon extends beyond \( \S_n \). In~\cite{debus2020reflection}, it was shown that similar dimension-independent results hold for polynomials invariant under general finite reflection groups. The Gram matrix can again be block-diagonalized according to isotypic components under the group action, and the SDP problem reduces accordingly.

Symmetry can thus be systematically exploited to simplify SOS decompositions and reduce the complexity of the associated SDP problems quite drastically. In practical applications, from real algebraic geometry to optimization, this often enables the treatment of larger and more structured problems than would otherwise be feasible. However, as mentioned, there are many non-negative polynomials that are not sums of squares (see \cite{blekherman2006}); in particular, this also holds in the symmetric case \cite{goel2016choi}. Therefore, we cannot, unfortunately, use this approach for all computational problems with symmetric polynomials, nor should we expect such a drastic reduction of complexity in every case. In the next section, we will see another way of reducing the geometric dimension of symmetric problems that leads to computational improvements.

\subsubsection{The (Half-)Degree Principle.}\label{sec:degree}

As discussed in previous sections, sums-of-squares techniques benefit significantly from symmetry. One key insight behind the reduced complexity in symmetric SOS decompositions is that, under symmetry, computations can often be transferred to the invariant ring, thus reducing the number of variables. This observation raises a natural question: can other tasks in real algebraic geometry similarly benefit from a reduction to the invariant ring? Unfortunately, the answer is not always yes as the following example.

\begin{example}
Consider the symmetric polynomial \( p_2 := \sum_{i=1}^n x_i^2 \). Clearly, \( p_2(x) \ge 0 \) for all \( x \in \R^n \). However, when expressed in the invariant ring (in terms of the elementary symmetric polynomials), we obtain
\[
  p_2 = e_1^2 - 2e_2.
\]
While this identity is algebraically correct, the polynomial \( e_1^2 - 2e_2 \) is not necessarily non-negative when viewed as a function on the real points of the invariant ring (it can take negative values).
\end{example}

This simple observation illustrates a fundamental limitation: many geometric properties, especially non-negativity, are not preserved under naive substitution into the invariant ring. Thus, a direct dimension reduction via invariants does not always apply.

Nevertheless, symmetry can still be exploited to reduce dimensionality through the so-called \emph{half-degree principle}, which simplifies the analysis of symmetric polynomial inequalities and the description of certain symmetric semi-algebraic sets. The core idea is that to test whether a symmetric inequality holds on all of \( \R^n \), it suffices to verify it on a significantly smaller set of inputs, namely, on those vectors that have only a few distinct coordinate values.

\begin{theorem}[Timofte's Half-Degree Principle {\cite{Timofte2003}}]
Let \( f \in \R[x_1, \ldots, x_n] \) be a symmetric polynomial of degree \( 2d > 2 \). Then \( f \ge 0 \) on \( \R^n \) if and only if \( f \ge 0 \) on the subset of \( \R^n \) in which at most \( d \) coordinates are distinct.
\end{theorem}

In other words, a symmetric polynomial \( f \) of degree \( 2d \) is non-negative on \( \R^n \) if and only if it is non-negative on the set of vectors in \( \R^n \) that have at most \( d \) distinct entries. This result reduces the task of verifying the non-negativity of \( f \) from checking all of \( \R^n \) to checking a potentially much smaller subset. However, the effectiveness of this method diminishes for very high-degree polynomials in relatively low dimensions, since if \( d \ge n \) there is essentially no reduction (every vector in \( \R^n \) automatically has at most \( n \) distinct coordinates).

A   simple proof of this principle and a slightly more general version suitable for optimization problems is presented in \cite{Riener2012}.

\begin{theorem}[Riener \cite{Riener2012}]
Let \( f_0, f_1, \dots, f_m \in \R[x_1,\dots,x_n] \) be symmetric polynomials and define the feasible set
\[
  K := \{\, x \in \R^n \mid f_1(x) \ge 0, \;\dots,\; f_m(x) \ge 0 \,\}.
\]
Let
\[
  r := \max\!\Big( 2, \; \big\lfloor \deg(f_0)/2 \big\rfloor, \; \deg(f_1), \; \dots, \; \deg(f_m) \Big).
\]
Then:
\[
  \inf_{x \in K} f_0(x) = \inf_{x \in\, K \cap A_r} f_0(x),
\]
where \( A_r \) denotes the set of points in \( \R^n \) with at most \( r \) distinct coordinates.
\end{theorem}

As an immediate consequence, we obtain a certification of real solutions for a single symmetric polynomial:

\begin{corollary}
Let \( f \in \R[x_1, \dots, x_n] \) be a symmetric polynomial, and let 
\[
  r := \max\{\, 2, \; \lfloor \deg(f)/2 \rfloor \}.
\]
If the real variety \( V_{\R}(f) \) is nonempty, then 
\[
  V_{\R}(f) \cap A_r \neq \emptyset.
\]
In other words, if the equation \( f = 0 \) has a real solution, then it has a solution in which at most \( r \) coordinates are distinct.
\end{corollary}

\begin{proof}
Suppose \( x \in \R^n \) is such that \( f(x) = 0 \). Let \( R > 0 \) be large enough so that \( x \) lies in the closed ball \( B_R(0) \) of radius \( R \) centered at the origin. Since \( f \) is continuous, it attains both a maximum and a minimum on \( B_R(0) \); denote these values by \( m_1 \) and \( m_2 \), respectively.

Applying the optimization theorem above (to both \( f \) and \( -f \) on the feasible set defined by \( B_R(0) \)), we find points \( x_1, x_2 \in A_r \) such that \( f(x_1) = m_1 \) and \( f(x_2) = m_2 \). By construction, \( m_1 \geq 0 \) and \( m_2 \leq 0 \) because \( x \) yields the value \( 0 \), which lies between \( m_1 \) and \( m_2 \).

Now, \( A_r \) is connected (intuitively, one can continuously interpolate between any two vectors that have at most \( r \) distinct coordinate values), and \( f \) is continuous. By the intermediate value theorem, there must exist some \( x_3 \in A_r \) on the line segment between \( x_1 \) and \( x_2 \) such that \( f(x_3) = 0 \). This \( x_3 \in A_r \) satisfies \( f(x_3) = 0 \), as required.
\end{proof}

\paragraph{Application to \(\S_n\)-closed sets.}

The original version of this principle applied only to sets defined by symmetric polynomials. However, it is possible to generalize it to more general algebraic and semialgebraic sets that are invariant under permutations, even if the defining polynomials themselves are not symmetric. For algebraic sets, this can be done via a simple observation. Let \( f_1, \dots, f_m \in \R[x_1, \dots, x_n] \) be polynomials (not necessarily symmetric) of degree at most \( d \), and suppose the real variety \( V_{\R}(f_1, \dots, f_m) \) is invariant under all permutations of the coordinates (i.e., it is an \(\S_n\)-closed set). Define a new polynomial
\[
  g := \sum_{i=1}^m \sum_{\sigma \in \S_n} \big(\sigma(f_i)\big)^2,
\]
where \(\sigma(f_i)\) denotes the polynomial obtained by applying the permutation \(\sigma\) to the indices of \(f_i\). By construction, \(g\) is a symmetric polynomial of degree at most \(2d\). Moreover, we have
\[
  V_{\R}(g) = V_{\R}(f_1, \dots, f_m).
\]
Indeed, if \(x\) is a common zero of \(f_1, \dots, f_m\), then obviously \(g(x) = 0\). Conversely, if \(g(x) = 0\), then each squared term in the sum defining \(g\) must vanish; in particular, taking \(\sigma\) to be the identity permutation shows \(f_i(x) = 0\) for all \(i\). Thus, \(V_{\R}(f_1, \dots, f_m)\) is nonempty if and only if \(V_{\R}(g)\) is nonempty. Since \(\deg(g) \le 2d\), by the corollary above, \(V_{\R}(g)\) (and hence also \(V_{\R}(f_1, \dots, f_m)\)) contains a point with at most \(\frac{2d}{2} = d\) distinct coordinates.

For general semi-algebraic sets, such a simple transfer is not directly possible. Indeed, although it is possible to represent every symmetric semi-algebraic set by symmetric polynomials, the degree will be more than just doubled, in general.  However, Riener and Safey El Din proved a more general result for systems of equations and inequalities that are equivariant under the symmetric group~\cite{Riener2018}. Here, equivariance means that we consider an \( n \)-tuple \( \g = (g_1, \dots, g_n) \) of polynomials with the property that 
\[
  g_i(\sigma(x)) = g_{\sigma^{-1}(i)}(x)
\]
for all \(\sigma \in \S_n\). For example, the gradient vector of a \(\mathfrak{S}_n\)-invariant polynomial, such as \(x_1x_2x_3\), forms an \(\mathfrak{S}_n\)-equivariant system; in this case, the gradient is \((x_2x_3, x_1x_3, x_1x_2)\).

\begin{theorem}[\cite{Riener2018}] \label{rienermohab}
Let \( \f = (f_1, \dots, f_k) \) be an \(\S_n\)-invariant tuple of polynomials (i.e., each \( f_i \) is symmetric), and let \( \g = (g_1, \dots, g_n) \) be an \(\S_n\)-equivariant tuple of polynomials in \(\R[x_1, \dots, x_n]\), all of degree at most \( d \). Assume furthermore that each \( g_j \) has degree at least 2. Then the semi-algebraic set
\[
  S(\f, \g) := \{\, x \in \R^n \mid f_i(x) \ge 0 \text{ for } i=1, \dots, k, \quad g_j(x) = 0 \text{ for } j=1, \dots, n \,\}
\]
is empty if and only if
\[
  S(\f, \g) \cap A_{2d-1} = \emptyset.
\]
\end{theorem}

In other words, if the system \(\{ f_i(x) \ge 0, \; g_j(x) = 0 \}\) has no real solution, then already no solution exists with at most \( 2d - 1 \) distinct coordinate values (and vice versa).

It was also shown by  Moustrou, Riener, and Verdure \cite{Moustrou} that a similar statement can be derived for a very large number of $S_n$-invariant ideals, also  over arbitrary fields.

\begin{theorem}
Let \( I \subset \mathbb{K}[x_1, \dots, x_n] \) be an ideal stable under the action of \( \S_n \), and let \( f \in I \) be a polynomial of degree \( d \) such that the degree-\( d \) part of \( f \) involves fewer than \( n - d \) variables (equivalently, at least \( n - d \) of the variables do not appear in any of the top-degree monomials of \( f \)). Then every common solution of \( I \) has at most \( d \) distinct coordinate values.
\end{theorem}

The full result of Moustrou–Riener–Verdure (see \cite[Theorem 4]{Moustrou}) provides a finer classification of the possible “partition types” (patterns of repeating coordinates) that can occur in common solutions, based on the monomials appearing in the highest-degree part of \( f \). Corresponding  results are also available for other reflection groups \cite{debus2,debus3}. 

A similar result holds in the setting of \emph{multi-symmetric} actions of the symmetric group $S_n$, where the group acts by permuting blocks of $k$-tuples of variables rather than individual coordinates. This was established by G\"orlach, Riener, and Weisser in \cite{multi}. Furthermore, Riener and Schabert 
\cite{slices} extended the scope of the degree principle to a broad class of symmetric polynomials without any restriction on degree. They showed that the principle applies to any family of symmetric polynomials that can be expressed as linear combinations of a fixed number of elementary symmetric polynomials, independent of the total degree. Finally, beyond the symmetric group, analogous reduction principles have been developed for finite reflection groups by Acevedo and Velasco \cite{av}, and  by Friedl, Riener, and Sanyal \cite{friedl}.

\subsubsection{Efficient Algorithms Using the Degree Principle.}\label{sec:degreealgo}

The degree principle enables the reduction of computational complexity in decision problems such as emptiness testing, feasibility, or optimization over real algebraic sets. Here, we provide a brief overview of problems for which the principle is leveraged to design efficient symbolic algorithms. In the next subsection, we focus specifically on the problem of deciding the emptiness of semi-algebraic sets.

\paragraph*{Testing Non-Negativity.}  Given a polynomial \(f \in \K[x_1, \dots, x_n]\), we want to check whether \( f(x) \ge 0 \) for all \( x \in \mathbb{R}^n \). Instead of checking this in full dimension \( n \), the degree principle tells us to test only over configurations with at most \(\max(2, \lfloor \frac{\deg(f)}{2}\rfloor)\) distinct coordinates. For example, when \(d = 4\), we only need to test over the points of the form
\[
x = (\underbrace{y_1, \ldots, y_1}_{k}, \underbrace{y_2, \ldots, y_2}_{n-k}) 
\]
Substituting this into \( f \), we get a polynomial in terms of \( y_1, y_2 \), and check whether the inequality holds over all real values. This symbolic computation can be handled via standard tools like resultants, Cylindrical Algebraic Decomposition, or SOS techniques.

\paragraph{Emptiness of a Symmetric Semi-Algebraic Set.}
Consider the set
\[
S = \{ x \in \mathbb{R}^n : f(x) = 0,\ g_1(x) \ge 0,\ \ldots,\ g_m(x) \ge 0 \},
\]
where \( f, g_1, \ldots, g_m \) are symmetric polynomials of degree at most \( d \). Using the degree principle, one constructs a family of candidate solutions:
\[
x = (\underbrace{y_1, \ldots, y_1}_{\lambda_1 \text{ times}}, \ldots, \underbrace{y_d, \ldots, y_d}_{\lambda_d\text{ times}}),
\]
where \( \lambda_1 + \cdots + \lambda_d = n \), and checks whether there exists a real solution to the constraints under this parametrization.

This reduces the problem to deciding the feasibility of a system of polynomial equalities and inequalities in the \( y_i \)'s and multiplicities \( \alpha_i \), which is significantly more tractable symbolically when \( d \ll n \). We will give more details about the problem of deciding the emptiness of semi-algebraic sets defined by symmetric polynomials in Section \ref{sub:empty} beyond the fixed degree case.

\paragraph{Computing the Generalized Euler–Poincaré Characteristic.}

The \emph{Euler–Poincaré characteristic} is an important invariant in real algebraic geometry. The Generalized Euler–Poincaré Characteristic is a significant generalization that - unlike the classical topological Euler characteristic - satisfies an \emph{additivity property}, making it suitable for use in algorithmic frameworks such as Euler integration and symbolic computation.

The development of efficient algorithms for computing this invariant in symmetric settings builds on the optimization version of the degree principle described above. In \cite{BasuRiener2013,basu2018equivariant}, Basu and Riener combined this principle with tools from \emph{Morse theory}, leading to new complexity bounds and algorithmic methods for analyzing symmetric semi-algebraic sets.

\begin{definition}
Let \( X \) be a semi-algebraic set, and let \( \F \) be a field. We denote  the \( i \)-th homology group of \( X \) with coefficients in \( \F \) by \( \HH_i(X, \F) \), and its dimension by 
\[
b_i(X, \F) := \dim_{\F} \HH_i(X, \F).
\]
\end{definition}

Morse theory tells us that Betti numbers are bounded by the number of critical points of a smooth function. In symmetric settings, this idea can be adapted using \emph{equivariant Morse functions}, which respect the symmetry of the space. The following theorem, proved in \cite{BasuRiener2013}, gives explicit bounds on equivariant Betti numbers for symmetric varieties.

\begin{theorem}[\cite{BasuRiener2013}]
Let \( V \subset \mathbb{R}^n \) be defined by a symmetric polynomial \( f \) of degree \( \le d \).Then:

\begin{itemize}
    \item The Betti numbers of the orbit space satisfy:
    \[
    b(V / \S_n, \F) \le \sum_{\substack{\boldsymbol{\ell}=(\ell_1,\ldots,\ell_\omega)\\ 1 \le \ell_i \le \min(k_i, 2d)}} p(\mathbf{k}, \boldsymbol{\ell}) \cdot d \cdot (2d - 1)^{|\boldsymbol{\ell}| + 1}.
    \]
    
    \item All higher Betti numbers vanish:
    \[
    b_i(V / \S_n, \F) = 0 \quad \text{for } i \ge \sum_{j=1}^\omega \min(k_j, 2d).
    \]
    
    \item If \( 2d \le k_i \) for all \( i \), then:
    \[
    b(V /{\S}_{n}, \F) \le (k_1 \cdots k_\omega)^{2d} \cdot (O(d))^{2\omega d + 1}.
    \]
    
    \item For \( \F = \mathbb{Q} \), the total equivariant Betti number satisfies:
    \[
    b^{\mathfrak{S}_{\mathbf{k}}}(V, \mathbb{Q}) \le \sum_{\boldsymbol{\ell}} p(\mathbf{k}, \boldsymbol{\ell}) \cdot d \cdot (2d - 1)^{|\boldsymbol{\ell}| + 1}.
    \]
\end{itemize}
\end{theorem}

The Morse theoretic proof for this bounds developed in \cite{BasuRiener2013}  forms the basis for an efficient algorithm of the Euler–Poincaré characteristic.
In general, the Euler-Poincaré charateristic is defined as follows:
\begin{definition}
The \emph{topological Euler–Poincaré characteristic} (EPC) of a semi-algebraic set \( S \subset \mathbb{R}^k \) is the alternating sum of its Betti numbers:
\[
\chi^{\mathrm{top}}(S, \F) := \sum_i (-1)^i b_i(S, \F).
\]
\end{definition}

However, for many algorithmic and integrative purposes, the \emph{generalized Euler–Poincaré characteristic} is more useful due to its additivity. It agrees with the topological one on compact sets but differs in general.

\begin{definition}[Generalized Euler–Poincaré Characteristic  \cite{Dries}]
The invariant \( \chi^{\gen}(S) \in \mathbb{Z} \) is uniquely characterized by:
\begin{enumerate}
    \item Invariance under semi-algebraic homeomorphisms,
    \item Multiplicativity: \( \chi^{\gen}(A \times B) = \chi^{\gen}(A) \cdot \chi^{\gen}(B) \),
    \item Additivity: \( \chi^{\gen}(A \cup B) = \chi^{\gen}(A) + \chi^{\gen}(B) - \chi^{\gen}(A \cap B) \),
    \item Normalization: \( \chi^{\gen}([0,1]) = 1 \).
\end{enumerate}
\end{definition}

There is a  general outline of singly-exponential algorithms for computing Euler characteristics (see \cite[Chapter 13]{BPR06}), which use deformations to smooth sets with non-degenerate critical points. Building on this general approach, Basu and Riener implemented  two critical adaptations to arrive at a polynomial time algorithm for the symmetric case:

\begin{itemize}
    \item \textbf{Equivariant construction:} Firstly, both the perturbation and the Morse function must respect the symmetry. This significantly restricts the deformation space and ensures polynomial complexity (for fixed degree).

    \item \textbf{Equivariant Morse theory:} Secondly, the analysis of topological changes at critical points relies on new equivariant Morse lemmas, which they provided. This then allowed for a control of the number of critical orbits is controlled using the \emph{half-degree principle} described above, rather than general Bézout-type bounds.
\end{itemize}
The combination of these adaptations gives rise to the following algorithmic result:

\begin{theorem}[Algorithm for generalized EPC \cite{Basu2018}]
\label{thm:algorithm-sa}
Let \( \mathcal{D} \subset \mathbb{R} \) be an ordered domain. Given:
\begin{itemize}
    \item a non-negative integer $k\in\N$;
    \item A set \( \mathcal{P} = \{P_1, \ldots, P_s\} \subset \mathcal{D}[\X] \) of symmetric polynomials of degree at most \( d \);
    \item A \( \mathcal{P} \)-semi-algebraic set \( S = \bigcup_{\sigma \in \Sigma} \mathcal{R}(\sigma, \mathbb{R}^n) \), with \( \Sigma \subset \{0,1,-1\}^\mathcal{P} \).
\end{itemize}
The algorithm computes \( \chi^{\gen}(S) \) and \( \chi^{\gen}_{{\S}_{n}}(S) \) with complexity bounded by:
\[
\card(\Sigma)^{O(1)} + s^{D} k^d d^{O(D^2)} + s^{D} d^{O(D)} (k \omega D)^{O(D)},
\]
where $D = \min\{n,d\}.$
\end{theorem}

\subsubsection{General Complexity Reduction via The Degree Principle.}

A key feature of the (half-)degree principle is that it allows  to reduce the original problem to a finite union of linear subspaces of dimension at most \( r = \max\{2, \lfloor d/2 \rfloor\} \). These subspaces correspond to points with a bounded number of distinct coordinates and their number is controlled by a polynomial in \( n \) (for fixed degree \( d \)). This means that for many algorithmic problems involving symmetric polynomials, such as testing non-negativity, deciding emptiness of semi-algebraic sets, or computing topological invariants, we obtain algorithms whose complexity is polynomial in the ambient dimension.

As we have seen in several examples above, this approach has proven to be surprisingly effective. It provides not only conceptual simplification, but also concrete complexity reductions in symbolic algorithms. In particular, it has been used to design polynomial-time algorithms for computing the generalized Euler–Poincaré characteristic, bounding Betti numbers, and deciding the feasibility of symmetric systems.

However, the dependence on the degree remains a crucial limitation. When the degree \( d \) increases with \( n \), the number and dimension of the subspaces involved can become too large, and the method loses its practical efficiency. This degree bottleneck is a fundamental caveat in direct use of the principle and highlights the need for alternative strategies or refinements when working with high-degree input. A result in this direction is presented in the next section. 

Furthermore, the underlying idea of reducing to points with few distinct coordinates is not limited to symmetric group actions. A generalization of the (half-)degree principle has been established for real algebraic sets invariant under the action of any finite reflection group by Friedl, Riener and Sanyal \cite{Riener2018}. This shows that the phenomenon is rooted more deeply in the structure of invariant theory and geometric symmetry and hopefully opens  the door to further applications beyond the symmetric case.

\subsection{More Efficient Algorithms Beyond the Fixed Degree Case} \label{sub:empty}
In real algebraic geometry, an important problem is to determine whether a given system of polynomial equations and inequalities has a real solution, that is, whether the real (semi-)algebraic set it defines is nonempty.  This problem is   intrinsically hard \cite{BlumCuckerShubSmale2012} and computationally challenging, for example, quantifier elimination over the reals is doubly exponential in the number of variables. 

The earliest algorithmic approaches trace back to Fourier~\cite{Fourier26}, who provided a method for solving linear inequalities, later rediscovered by Dines in 1919~\cite{Dines19}. These methods form an early bridge to elimination theory.
Tarski’s theorem~\cite{Tarski} asserts that the projection of a semi-algebraic set onto a coordinate subspace is also semi-algebraic. Its algorithmic counterpart, based on Sturm’s theorem for univariate real root counting, enables recursive strategies by eliminating variables one at a time. The first algorithm with elementary recursive complexity, known as Cylindrical Algebraic Decomposition (CAD), was introduced by Collins~\cite{Collins} and further developed in subsequent works~\cite{McCallum,McCallumeq,Hong:92d,Strz06,Strz14,Chen09,EnDa20,Chen20}.
However, these algorithms are typically run in doubly exponential time in the number of variables~\cite{Dav88,BrDa07}. Notably, many variants also address quantifier elimination, which is a more general and computationally harder problem than the real root decision problem.
More efficient algorithms with singly exponential time complexity (in the number of variables) and polynomial dependence on the degree were pioneered by Grigoriev and Vorobjov~\cite{GV88}, and Renegar~\cite{Ren}, with further advancements by Canny~\cite{Canny}, Heintz, Roy, and Solern\'o~\cite{HRS93}, and Basu, Pollack, and Roy~\cite{BPR96}. These approaches utilize the {critical point method}, which reduces the problem to computing finitely many complex critical points of a polynomial map that achieves extrema on each connected component of the semi-algebraic set.
Table \ref{tab:real-root-complexity} gives a summary of the complexities of the algorithms mentioned above. 

\begin{table}[h!]
\centering
\caption{Summary of Algorithms for the Real Root Decision Problem}
\renewcommand{\arraystretch}{1.3}
\begin{tabular}{|>{\centering\arraybackslash}p{3.5cm}|
                >{\centering\arraybackslash}p{4.5cm}|
                >{\centering\arraybackslash}p{3cm}|}
\hline
\textbf{Authors} & \textbf{Method} & \textbf{Complexity} \\
\hline
Fourier~\cite{Fourier26}, Dines~\cite{Dines19} & Fourier-Motzkin Elimination & Exponential in $n$ \\
\hline
Collins~\cite{Collins} & Cylindrical Algebraic Decomposition & $d^{2^{O(n)}}$ \\
\hline
Tarski~\cite{Tarski} & Projection Theorem (non-constructive) & Non-elementary \\
\hline
Grigoriev-Vorobjov~\cite{GV88} & Critical Point Method & $(sd)^{2^{O(n)}}$ \\
\hline
Renegar~\cite{Ren} & Critical Point Method & $s^{2^{O(n)}} d^{O(n)}$ \\
\hline
Canny~\cite{Canny} & Critical Point Method & $s^{O(n)} d^{O(n^2)}$ \\
\hline
Heintz, Roy, Solernó~\cite{HRS93} & Improved Critical Point Method & $s^{O(n)} d^{O(n)}$ \\
\hline
Basu, Pollack, Roy~\cite{BPR96} & Improved Critical Point Method & $s d^{O(n)}$ \\
\hline
\end{tabular}
\label{tab:real-root-complexity}
\end{table}

In the case of symmetric polynomials, the degree principle and its extensions offer a powerful simplification.
The degree principle, when the bound \(d\) on the degrees of the polynomials satisfies \(d \le n/2\), allows this question to be reduced to checking a finite union of lower-dimensional subspaces, specifically, those defined by tuples with bounded numbers of distinct entries, as shown in Theorem \ref{rienermohab} by Riener and Safey El Din~\cite{Riener2018}. Their work  extends the degree  principle to systems defined not only by symmetric but also equivariant families of polynomials. The main fininds show that such systems, of degree \(d\), either have no solution or possess one with at most \(2d-1\) distinct coordinates. This  geometric insight then enables the design of algorithms with complexity polynomial in the number of variables (and singly exponential in the degree). Moreover, the authors show that these techniques can be extended to parameterized systems, allowing efficient quantifier elimination in a broad class of problems with symmetric structure.

\begin{pb}[The Emptiness of Symmetric Semi-Algebraic Sets]   
Consider a real semi-algebraic set \(S\) defined by 
symmetric polynomials in \(\mathbb{R}[x_1, \dots, x_n] \). Determine whether the set \( S\) is nonempty? 
\end{pb}

However, when \(2d > n\), the principle becomes vacuous: every point in \(\R^n\) already has at most \(\lfloor d/2 \rfloor\) distinct entries. In this case, the method no longer restricts the search space, and the full \(\R^n\) must be considered, reverting to the original complexity of the problem. 

The main idea is to use the critical point method. However,
leveraging $\S_n$-invariance in critical point computations alone is not sufficient to solve real root decision problems more efficiently using the critical point method. Therefore, additional techniques are required.

Instead of applying the critical point method directly to the input system, we can apply it to each individual $\mathfrak{S}_n$-orbit, using a suitable polynomial map tailored to each orbit. Intuitively, working with $\mathfrak{S}_n$-orbits corresponds to separately examining real points with all distinct coordinates, points with some repeated coordinates, or points partitioned into groups of equal coordinates. In each case, an orbit can be characterized by points with a bounded number of pairwise distinct coordinates, a key observation that enables the construction of generic, orbit-invariant maps, encoded by partitions \(\lambda = (n_1^{\ell_1} \, \dots \, n_r^{\ell_r})\) of \(n\).

\subsubsection{Critical Point Method for Symmetric Polynomials.}
Let \(W \subset \C^\ell\) be an equidimensional algebraic set, and let \(\phi\) be a polynomial function defined on \(W\). A non-singular point \(\bm w \in W\) is called a \emph{critical point} of \(\phi\) on \(W\) if the gradient of \(\phi\) at \(\bm w\) is normal to the tangent space \(T_{\bm w} W\) of \(W\) at \(\bm w\).

If \(\bm g = (g_1, \dots, g_s)\) are generators of the ideal associated with \(W\), then \(T_{\bm w}W\) is the right kernel of the Jacobian matrix \(\jac(\bm g)\) evaluated at \(\bm w\). In the cases we consider, this matrix has rank \(s\) at all points of \(W\). The set of critical points of the restriction of \(\phi\) to \(W\) is then defined by the vanishing of \(\bm g\), along with the vanishing of all \((s+1)\)-minors of the Jacobian matrix \(\jac(\bm g, \phi)\).

Let $\lambda = (n_1^{\ell_1} \dots n_r^{\ell_r})$ be a partition of $n$, and let the variables be grouped accordingly as $\X = (\X_1, \ldots, \X_r)$ with each $\X_i$ of length $\ell_i$. The polynomials $g_1, \ldots, g_s$ are assumed to be invariant under the action of $\S_{\ell_1} \times \cdots \times \S_{\ell_r}$.

To construct suitable objective functions, we introduce the {power sums} $p_{i,j}$, which sum powers of the variables in each block $\X_i$. Using these, we define a family of symmetric polynomials of the form:
\begin{equation}
    \label{eq:calNN}
    \calN_{\fraka} =
\sum_{i=1}^r c_i\, p_{i,\ell_i+1} +
\sum_{i=1}^r \sum_{j=1}^{\macrodelta_i} \fraka_{i,j} p_{i,j},
\end{equation}
where the coefficients $c_i$ are chosen to ensure that $\calN_{\fraka}$ always has even degree and remains invariant under the group action, i.e., \(c_i=0\) if \(\ell_i\) is even and  \(c_i=1\) if \(\ell_i\) is odd. The parameters $\fraka_{i,j}$ are treated as indeterminates, so the family $\calN_{\fraka}$ depends linearly on these.

For \(\a = (\a_1, \dots, \a_r)\) in \(\C^{\macrodelta_1} \times \cdots \times \C^{\macrodelta_r}\), with each \(\a_i \in \C^{\macrodelta_i}\), we denote by \(\calN_\a\) the specialization of \(\calN_{\fraka}\) obtained by evaluating the indeterminates \(\fraka_i\) at \(\a_i\) for all \(i\).
For a partition $\lambda = (n_1^{\ell_1} \dots n_r^{\ell_r})$ of \(n\), we denote by $\calU_\lambda \subset \C^\ell$ the open set consisting of points $\bm{w} = (\bm{w}_1, \dots, \bm{w}_r) = (w_{i,j})_{1 \le i \le r, 1 \le j \le \ell_i}$ such that the coordinates of each $\bm{w}_i = (w_{i,j})_{1 \le j \le \ell_i}$ are pairwise distinct for $i = 1, \dots, r$. 

\begin{proposition}\label{prop:feasible}
  Let \( \g = (g_1, \dots, g_s) \) be \( \S_{\ell_1} \times \cdots \times \S_{\ell_r} \)-invariant
  polynomials in \( \Q[\X_1, \dots, \X_k] \). Suppose further that the Jacobian matrix of \( \g \) has rank \( s \) at any of its solutions. Then there exists a non-empty Zariski open set
  \( \calA_\lambda \subset \C^{\macrodelta_1} \times \cdots \times
  \C^{\macrodelta_r} \) such that for any \( \a \in \calA \), the restriction of
  \( \calN_\a \) to \( V(\g) \) has finitely many critical points in \( \calU_\lambda \).
\end{proposition}
The main idea for proving this proposition is to characterize the critical points of $\calN_{\fraka}$ restricted to $V(\bm{g})$. To this end, we introduce new variables $\vl_1, \ldots, \vl_s$ and define the associated Lagrange system:
\[
\calS_{\fraka} = \big(g_1, \ldots, g_s,\; [\vl_1 \cdots \vl_s~1] \cdot \jac(\bm{g}, \calN_{\fraka})\big).
\]
This system encodes both the constraints defining the variety and the first-order conditions for optimality. The projection of its solution set back to the $\X$-space recovers the set of critical points. More precisely, suppose that the Jacobian matrix of $\bm{g} = (g_1, \dots, g_s)$ has rank $s$ at its solutions. Then, for $\a \in \C^{\macrodelta_1} \times \cdots \times \C^{\macrodelta_r}$, the set $\pi_\X(V(\calS_{\a}))$ is the critical locus of the restriction of the map $\calN_{\a}$ to $V(\bm{g})$.

Furthermore, if $V(\bm{L}_\a)$ is finite, then $V(\calS_\a) \cap (\C^s \times \mathcal{U}_\lambda)$ is finite. A key step in the analysis uses the fact that the polynomial system can be re-expressed in power sum coordinates. The map $\bm{P}$ from the original variables to their power sums is invertible over an open set $\mathcal{U}_\lambda$ where the variables in each block are pairwise distinct. This ensures that the Jacobian of $\bm{P}$ (a multi-block Vandermonde matrix) is non-singular on $\mathcal{U}_\lambda$. Then, by showing that there exists a non-empty Zariski open set of parameters $\a$ for which the lifted system (expressed in the power sums) defines a radical ideal with finitely many solutions, we obtain the claim in the proposition above.

\paragraph{Proper maps.}
\label{sec:proper}

A real-valued function \(\psi : \mathbb{R}^n \to \mathbb{R}\) is called \emph{proper at} \(x \in \mathbb{R}\) if there exists \(\varepsilon > 0\) such that the preimage \(\psi^{-1}([x-\varepsilon, x+\varepsilon])\) is compact. Proper functions are important because a proper polynomial restricted to a real algebraic set \(W\) attains extrema on each connected component of \(W\). 

Using \cite[Theorem 2.1 and Corollary 2.2]{sakkalis2005note}, one can construct proper polynomials as follows. Let 
\[
H := H_k(x_1, \dots, x_n) + H_{k-1}(x_1, \dots, x_n) + \cdots + H_0(x_1, \dots, x_n) : \mathbb{R}^n \to \mathbb{R}
\]
be a real polynomial decomposed into its homogeneous components \(H_i\) of degree \(i\). If the leading form \(H_k\) is positive definite, then \(H\) is proper. 
In particular, the polynomial 
\[
p_{2m} + \sum_{i=0}^{2m-1} \gamma_i p_i,
\]
where \(p_i\) are the Newton sums in \(x_1, \dots, x_n\) and \(\gamma_i \in \mathbb{Q}\), is proper. This construction extends naturally to blocks of variables as follows. As a consequence, the map \(\phi_\a\) defined in \eqref{eq:calNN} is proper.

\begin{lemma}\label{lemma:proper_power}
Let \(\X_1, \dots, \X_r\) be blocks of variables of sizes \(\ell_1, \dots, \ell_r\), respectively. Define 
\[
p_{i, j} := x_{i,1}^j + \cdots + x_{i,\ell_i}^j.
\]
Then for any integers \(m_1, \dots, m_r \geq 1\) and coefficients \(\gamma_{i,j} \in \mathbb{Q}\), the polynomial
\[
\sum_{i=1}^r p_{2m_i, i} + \sum_{i=1}^r\sum_{j=0}^{2m_i -1} \gamma_{i,j} p_{i,j}
\]
is proper.
\end{lemma}

\subsubsection{Computing the Critical Points along Orbits.} 
Let \(\lambda = (n_1^{\ell_1}\, \dots \, n_r^{\ell_r})\) be a partition of \(n\), and denote \(\mathfrak{S}_\lambda = \mathfrak{S}_{\ell_1} \times \cdots \times \mathfrak{S}_{\ell_r}\). Let \(\bm g = (g_1, \dots, g_s)\) be a sequence of \(\mathfrak{S}_\lambda\)-invariant polynomials, and let \(\phi\) be a \(\mathfrak{S}_\lambda\)-invariant polynomial in \(\mathbb{K}[\bm {X}_1, \dots, \bm {X}_r]\), where each \(\bm {X}_i = (x_{i,1}, \dots, x_{i,\ell_i})\). Let \(\ell = \ell_1 + \cdots + \ell_r\), and assume \(s \le \ell\). Further assume that the Jacobian matrix of \(\bm {g}\) has rank \(s\) at every solution.

Let \(\Phi\) and \(\bm G = (G_1, \dots, G_s)\) be polynomials in \(\mathbb{K}[\bm {Y}_1, \dots, \bm {Y}_r]\), where each \(\bm {y}_i = (y_{i,1}, \dots, y_{i,\ell_i})\) is a set of \(\ell_i\) new variables, such that:
\[
\phi = \Phi(\bm {e}_1, \dots, \bm {e}_r) \quad \text{and} \quad \bm {g} = \bm {G}(\mathbf{e}_1, \dots, \bm {e}_r),
\]
where \(\bm {e}_i = (e_{i,1}, \dots, e_{i,\ell_i})\) is the tuple of elementary symmetric polynomials in \(\bm {X}_i\), with \(\deg(e_{i,j}) = j\).

\begin{lemma}\label{lemma:subroutine}
Let \(\bm {g}, \phi\), and \(\lambda\) be as above. Assume further that \(\Phi\) has finitely many critical points on \(V(\bm {G})\). Then there exists a randomized algorithm 
\[
\textsc{Critical\_Points}(\bm {g}, \phi, \lambda)
\]
which returns a zero-dimensional parametrization of the critical points of \(\Phi\) restricted to \(V(\bm {G})\).
The algorithm uses
\[
\softO\left(\delta^2 c_\lambda (e_\lambda + c_\lambda^5) n^4 \Gamma\right)
\]
operations in \(\mathbb{K}\), where:
\begin{align*}
c_\lambda &= \frac{\deg(g_1)\cdots \deg(g_s) \cdot e_{\ell-s}(\delta-1, \dots, \delta - \ell)}{\ell_1! \cdots \ell_r!}, \\
\Gamma &= n^2 \binom{n + \delta}{\delta} + n^4 \binom{n}{s+1}, \\
e_\lambda &= \frac{n (\deg(g_1)+1) \cdots (\deg(g_s)+1) \cdot e_{\ell-s}(\delta, \dots, \delta - \ell + 1)}{\ell_1! \cdots \ell_r!},
\end{align*}
and \(\delta = \max(\deg(\bm G), \deg(\Phi))\). The number of solutions returned is at most \(c_\lambda\).
\end{lemma}
The \textsc{Critical\_Points} procedure consists of two main steps: first computing the polynomials \(\bm G\) and \(\Phi\) and then finding a zero-dimensional representation of the set of critical points of \(\Phi\) on \(V(\bm {G})\). The first step is performed using the algorithm \textsc{Symmetric\_Coordinates} from \cite[Lemma 9]{faugere2020computing}, which requires
\(
\softO\left(\binom{\ell + \delta}{\delta}^2\right)
\)
operations in \(\mathbb{K}\). 

Since the Jacobian matrix of \(\bm {g}\) has full rank \(s\) at every point of \(V(\bm {g})\), Lemma~\ref{lemma:transferreg_ele} implies that the Jacobian matrix of \(\bm {G}\) also has rank \(s\) at every point of its zero set. Consequently, the set \(W(\Phi, \bm {G})\) is defined by the vanishing of \(\bm {G}\) together with all the \((s+1)\)-minors of the Jacobian matrix of \((\bm {G}, \Phi)\).  Note that each elementary symmetric polynomial \(e_{i,j}\) has degree \(j\). Assigning weight \(j\) to the variable \(y_{i,j}\) endows the polynomial ring \(\mathbb{K}[\bm {Y}_1, \dots, \bm {Y}_r]\) with a weighted grading.

To compute a zero-dimensional parametrization of \(W(\Phi, \bm G)\), we apply the symbolic homotopy method for weighted polynomial systems described in \cite[Theorem 5.3]{labahn2021homotopy} (see also \cite[Section 5.2]{faugere2020computing}). This randomized algorithm requires
\(
\softO\left(\delta^2 c_\lambda (e_\lambda + c_\lambda^5) n^4 \Gamma \right)
\)
operations in \(\mathbb{K}\). Moreover, the output contains at most \(c_\lambda\) points, as guaranteed by \cite[Theorem 5.3]{labahn2021homotopy}. 

\medskip
Finally, we refer the reader to~\cite{faugere2020computing} and~\cite{Vu2022} for further details on how symmetry can be exploited to design faster algorithms for computing the critical points of an invariant map restricted to an invariant algebraic variety. These works illustrate how the underlying group actions and invariant structures can be used to reduce the size of the systems to solve, lower the dimension of the search space, and improve the efficiency of symbolic computation methods.
\subsubsection{The Emptiness of Real Algebraic Sets.}

Let \(\f = (f_1, \dots, f_s)\) be symmetric polynomials in \(\Q[x_1, \dots, x_n]\) with \(s \le n\). Assume further that the Jacobian matrix of \(\f\) with respect to \(x_1, \dots, x_n\) has rank \(s\) at any point in \(V(\f)\). We present here an algorithm to test whether the algebraic set \(V(\f)\) has a real point, i.e., whether the real algebraic set
\[
V_\R(\f) := \left\{ \a \in \R^n \; | \; f_1(\a) = \cdots = f_s(\a) \right\}
\]
is empty.
Note that if the Jacobian matrix of \(\f\) has full rank \(s\) at every point in its solution set, then by the Jacobian criterion, the algebraic set \(V(\f)\) is equidimensional of dimension \(n - s\). Therefore, it is natural to assume that \(s \le n\); otherwise, the algebraic set \(V(\f)\) is empty.

Let \(\lambda = (n_1^{\ell_1} \dots n_r^{\ell_r})\) be a partition of \(n\), and let \(\f^{[\lambda]} = \big(f_1^{[\lambda]}, \dots, f_s^{[\lambda]}\big)\) be a sequence of polynomials in \(\K[\X_1, \dots, \X_r]\), where each \(\X_k = (x_{k,1}, \dots, x_{k, \ell_k})\) is a set of \(\ell_k\) variables, obtained from \(\f\) as described in Section~\ref{sec_hom}. By Lemma~\ref{lemma:extend_slice}, the Jacobian matrix of \(\f^{[\lambda]}\) with respect to \(\X_1, \dots, \X_r\) also has rank \(s\) at any of its zeros.

Let \(\calN_\fraka\) be the map defined in~\eqref{eq:calNN}, and let \(\calA_\lambda \subset \Kbar^{\macrodelta_1} \times \cdots \times \Kbar^{\macrodelta_r}\) be the non-empty Zariski open set defined in Proposition~\ref{prop:feasible}. Assume that \(\a\) is chosen in \(\calA_\lambda\) (this constitutes one of the probabilistic steps of our algorithm) at Step~\ref{step:n}. Since the Jacobian matrix of \(\f^{[\lambda]}\) has rank \(s\) at every vanishing point, Proposition~\ref{prop:feasible} implies that the critical locus of the restriction of \(\calN_\a\) to \(V(\f^{[\lambda]})\) has dimension at most zero. Moreover, the map \(\calN_\a\) is invariant under the action of the group \(\S_\lambda = \S_{\ell_1} \times \cdots \times \S_{\ell_r}\).

Let \(\Phi_{\a}\) and \({\bm F^{[\lambda]}}\) be polynomials in \(\K[\bm Y_1, \dots, \bm Y_r]\), where \(\bm Y_k = (y_{k,1}, \dots, y_{k,\ell_k})\), such that
\[
 \calN_\a = \Phi_\a(\bm e_1, \dots, \bm e_r) \quad
 \text{and} \quad  \f^{[\lambda]} = \bm F^{[\lambda]}(\bm e_1, \dots, \bm e_r).
\]
Here, \(\bm e_i = (e_{i,1}, \dots, e_{i,\ell_i})\) denotes the vector of elementary symmetric polynomials in the variables \(\X_i\). 

In the next step, we compute a zero-dimensional parametrization \(\scrR_\lambda\) of the critical set \(W_\lambda := W(\Phi_\a, \bm F^{[\lambda]})\), which is the set of critical points of \(\Phi_\a\) restricted to \(V(\bm F^{[\lambda]})\), using the \textsc{Critical\_Points} algorithm from Lemma~\ref{lemma:subroutine}. This parametrization \(\scrR_\lambda\) consists of a sequence of polynomials \[(q, v_{1,1}, \dots, v_{1,\ell_1}, \dots, v_{r,1}, \dots, v_{r,\ell_r}) \in \K[T]^{\ell+1},\] along with a linear form \(\gamma\).

In the final step, we apply the procedure \textsc{Decide}(\(\scrR_\lambda\)) to determine whether the preimage of \(W_\lambda\) under the map \(E_\lambda\) (the compression via the elementary symmetric polynomials, defined in~\eqref{eq:comp}) contains real points.

\begin{algorithm}
\caption{\textsc{Real\_Emptiness}}
\label{al:crit}
\begin{algorithmic}[1]
\Require symmetric polynomials $\f = (f_1, \dots, f_s)$ in $\K[x_1, \dots, x_n]$, with $s < n$
\Ensure \false \ if $V(\f) \cap \R^n$ is non-empty; \true  \ otherwise
\Statex \hspace{-0.5cm}\textbf{Assumption:} the Jacobian matrix of \(\f\) has rank \(s\) at any point in \(V(\f)\)
\ForAll{partitions $\lambda = (n_1^{\ell_1} \dots n_r^{\ell_r})$ of $n$ of length at least $s$}
  \State compute $\f^{[\lambda]} = (f_1^{[\lambda]}, \dots, f_s^{[\lambda]})$
  \State construct map $\calN_\fraka$ as in \eqref{eq:calNN}
  \State choose $\a \in \calA$ as in Proposition~\ref{prop:feasible} \label{step:n}
  \State compute $\calN_\a$
  \State $\scrR_\lambda \gets \textsc{Critical\_Points}(\calN_\a, \f^{[\lambda]})$ \label{step:real_part}
  \If{$\textsc{Decide}(\scrR_\lambda) = \false$}
    \State \Return \false
  \EndIf
\EndFor
\State \Return \true
\end{algorithmic}
\end{algorithm}

\begin{theorem}[\cite{Labahn2023}]
There exists an algorithm called \textsc{Real\_Emptiness} which takes as input a sequence of symmetric polynomials \(\f = (f_1, \dots, f_s)\) and returns {\true} if \(V(\f) \cap \R^n = \emptyset\), and {\false} otherwise.

Moreover, the complexity of this algorithm is bounded by
\[
\softO\left(d^{6s+2}n^{11}{{n+d}\choose n}^6\left({{n+d}\choose n}
       +  {{n}\choose {s+1}} \right) \right)
\]
arithmetic operations over \(\K\), where \(d = \max\{\deg(f_i) \mid 1 \le i \le s\}\).
\end{theorem}

\subsection{Efficiently Computing Topological Invariants}\label{sec:topology}

The strategy of \emph{reducing the dimension} (as seen in the degree principle) can be applied to compute topological invariants of semi-algebraic sets defined by symmetric polynomials. In particular, using results of V.~I.~Arnol'd, one can analyze the homology of the \emph{orbit space} of a symmetric set rather than the original set in full dimension. This approach exploits symmetry to simplify the problem and often yields more efficient computations of invariants such as Betti numbers and connectivity.
\subsubsection{Computing in the Orbit Space.}

As mentioned before, it is often more insightful (and efficient) in symmetric problems to consider points \emph{up to symmetry}. Instead of examining each individual point of a symmetric set \( S \subseteq \R^n \), we group together those related by any permutation of coordinates. Each equivalence class under this relation is called an \emph{orbit}, and the set of all orbits forms the \emph{orbit space}, denoted \( S/\mathfrak{S}_n \). Concretely, one can identify \( S/\mathfrak{S}_n \) with the subset of \( S \) lying in a designated \emph{Weyl chamber} (for example, the region \( \mathcal{W}_c := \{ x_1 \le x_2 \le \cdots \le x_n \} \)), since each orbit has a unique representative with sorted coordinates in \( \mathcal{W}_c \).

Using the result of Arnold (Theorem~\ref{thm:arnold}), Basu and Riener showed that for a closed and bounded symmetric semi-algebraic set \( S \subset \R^n \) defined by symmetric polynomials of degree at most \( d \le n \), the orbit space \( S/\mathfrak{S}_n \) (i.e., \( S \) restricted to the Weyl chamber) is \emph{homotopy equivalent} to a certain semi-algebraic subset of \( \R^d \). In other words, one can coordinatize the orbit space by \( d \) parameters (e.g., the first \( d \) elementary symmetric polynomials or power-sum functions). This immediately implies that the higher cohomology groups of \( S/\mathfrak{S}_n \) vanish above a certain dimension. For any field of coefficients \( \F \), one has
\[
H^i(S/\mathfrak{S}_n, \F) = 0 \quad \text{for all } i \ge \min(n,d).
\]

The proof avoids Morse-theoretic technicalities by relying on results about \emph{Vandermonde mappings} due to Arnol'd~\cite{arnold1986hyperbolic}, Givental~\cite{givental1987moments}, and Kostov~\cite{kostov1989geometric}. These tools allow one to bound the \emph{equivariant} Betti numbers of symmetric semi-algebraic sets by relating \( S/\mathfrak{S}_n \) to the intersection \( S_{n,d} \) of \( S \) with certain \( d \)-dimensional faces of the Weyl chamber. The number of such faces is  
\[
\binom{n - \lceil d/2 \rceil - 1}{\lfloor d/2 \rfloor - 1} = O\left(d(n)^{\lfloor d/2 \rfloor - 1} \right),
\]
Moreover, since the intersection of each face with \( S \) lies in a \( d \)-dimensional subspace, the Betti numbers of each piece are bounded polynomially in \( s \) and \( n \). Mayer–Vietoris inequalities are then used to bound the Betti numbers of the union. A naive application leads to combinatorial explosion, so infinitesimal thickenings and shrinkings of the faces are employed to reduce the number of flags that need to be considered, yielding a bound of \( O\big(d(n)^{\lfloor d/2 \rfloor - 1}\big) \) on these contributions.

Combining these facts with classical bounds on the Betti numbers of semi-algebraic sets leads to the following.

\begin{theorem}[Basu--Riener~{\cite{basu2018equivariant}}]
Let \( S \subset \R^n \) be a \( \mathcal{P} \)-closed semi-algebraic set, where 
\[
\mathcal{P} \subset \R[X_1,\ldots,X_n]^{\mathfrak{S}_n}_{\leq d}, 
\quad |\mathcal{P}| = s.
\] 
Define
\[
F(d,n) = 
\begin{cases}
(2^d - 1) \displaystyle\prod_{i=1}^{\lfloor d/2 \rfloor - 1} (n - \lceil d/2 \rceil - i), & \text{if } d \leq n, \\[6pt]
\leq (2^n - 1)(n - 1)!, & \text{if } d > n.
\end{cases}
\]
Let \( d' := \min(n,d) \). Then:
\[
b(S/\mathfrak{S}_n, \F) \;\leq\; \big(O(s\, d\, d')\big)^{d'} \cdot F(d,n).
\]

In particular, if \( |\mathcal{P}| = 1 \) and \( S = \ZZ(\mathcal{P}, \R^n) \) with \( d < n \), then
\begin{equation}
\label{eqn:thm:bound:1}
b(S/\mathfrak{S}_n, \F) \;\leq\; d^{O(d)} \cdot n^{\lfloor d/2 \rfloor - 1}.
\end{equation}
\end{theorem}

The parity behavior in the exponent of \( n \) may appear surprising at first, but it is explained by Proposition~5.6 in \cite{brocker1998symmetric}, which relates the structure of \( S/\mathfrak{S}_n \) to the number of generators of certain 2-subgroups—remaining constant between \( \mathfrak{S}_{2n} \) and \( \mathfrak{S}_{2n+1} \).
Furthermore, these methods extend to symmetric \emph{definable} sets in arbitrary o-minimal expansions of real closed fields, demonstrating the generality of the approach.

\medskip

There is a deep connection between polynomial bounds and algorithmic complexity. The bounds above lead to efficient algorithms for computing equivariant Betti numbers.

\begin{theorem}
\label{thm:algorithm}
For every fixed $d \geq 0$, there exists an algorithm that takes as input a 
$\mathcal{P}$-closed formula $\Phi$, with $\mathcal{P} \subset \R[X_1,\ldots,X_n]^{\mathfrak{S}_n}_{\leq d}$, and computes 
$b_i(S/\mathfrak{S}_n,\F)$ for $0 \leq i < d$, where $S = \RR(\Phi,\R^n)$. The complexity is bounded by
\[
(|\mathcal{P}| \cdot n \cdot d)^{2^{O(d)}},
\]
i.e., polynomial in $|\mathcal{P}|$ and $n$ for fixed $d$.
\end{theorem}

\subsubsection{Connectivity up to Symmetry.}

The principle of computing up to symmetry naturally extends to topological questions beyond homology, such as connectivity. This was investigated by Riener, Schabert, and Vu~\cite{Issac}.

\begin{definition}
Two points \( \u, \v \in S \) are said to be \emph{orbit-connected} if their orbits belong to the same connected component of the orbit space \( S/\mathfrak{S}_n \). Equivalently, their sorted representatives in \( \mathcal{W}_c \) lie in the same connected component of \( S \cap \mathcal{W}_c \).
\end{definition}
This simplifies connectivity testing: rather than working in the full space \( S \subset \R^n \), one works in the reduced, sorted chamber \( S \cap \mathcal{W}_c \).

As with Betti numbers, dimension reduction plays a crucial role. Since symmetric polynomials of degree \( \le d \) depend only on the first \( d \) power sums, one can restrict attention to a lower-dimensional subset called the \emph{\( d \)-dimensional orbit boundary}: the subset of \( S \) consisting of points with at most \( d \) distinct coordinate values. Every connected component of \( S \cap \mathcal{W}_c \) intersects this orbit boundary.

Given two points in \( S \), we:
\begin{enumerate}
    \item Sort them to their representatives in \( \mathcal{W}_c \).
    \item Construct simplified representatives on the orbit boundary, preserving the first \( d \) power sums.
    \item Check connectivity of these representatives in the orbit boundary using standard graph techniques.
\end{enumerate}

\begin{theorem}[Riener--Schabert--Vu~{\cite{Issac}}]
Let \( S \subset \mathbb{R}^n \) be a semi-algebraic set defined by \( s \) symmetric polynomials of degree at most \( d \le n \). Let \( \u, \v \in S \) be points with sorted coordinates (i.e., \( \u, \v \in \mathcal{W}_c \)). Then there exists an algorithm that decides whether \( \u \) and \( \u \) are orbit-connected.

Moreover, for fixed \( d \) and \( s \), the algorithm runs in time
\[
O(n^{d^2}),
\]
i.e., polynomial in \( n \) with an exponent depending only on \( d \).
\end{theorem}

This illustrates the strength of symmetry-based reduction: high-dimensional connectivity problems become tractable through orbit space analysis, dimension reduction, and the geometry of symmetric semi-algebraic sets.

\subsubsection{Mirror Spaces.}

To compute global topological information from equivariant data, it is necessary to understand how the entire space is built up from the orbit space information. The framework of \emph{mirror spaces} provides a powerful tool to extend the computation of homological invariants of symmetric semi-algebraic sets from the orbit space to the full space.

In the previous section, we presented an effective method for determining connectivity within the canonical Weyl chamber. However, establishing orbit connectivity between two points does not necessarily imply that the points are connected in the ambient space. For instance, distinct points in the same orbit are always orbit-connected but need not lie in the same connected component of \( S \). To extend our efficient algorithm to the more general case where the points lie in different Weyl chambers, we introduce the notion of a \emph{mirror space}. This construction allows us to relate the topology of the orbit space, i.e., the topology of its intersection with the canonical Weyl chamber, to the topology of the full semi-algebraic set \( S \).

The visual intuition behind mirror spaces is similar to that of a kaleidoscope: the walls of the Weyl chamber act as mirrors, reflecting the contents of the chamber to generate the entire symmetric set. Depending on how a symmetric set \( S \) intersects the Weyl chamber, the topology of \( S \) can often be described entirely in terms of its orbit space. More generally, the notion of a mirror space is defined in the context of Coxeter group actions.

\begin{definition}
A \emph{Coxeter pair} \((\Ww, \Ss)\) consists of a group \(\Ww\) together with a set of generators \(\Ss = \{ s_i \mid i \in I \}\), where each \(s_i\) is of order \(2\), and a symmetric family of integers \((m_{ij})_{i,j \in I}\) such that \((s_i s_j)^{m_{ij}} = e\) for all \(i, j \in I\).
\end{definition}

\begin{example}
Consider the symmetric group \(\S_n\), which can be presented as a Coxeter group with generating set
\[
\Ss = \{ s_i = (i, i+1) \mid 1 \leq i \leq n-1 \},
\]
where each \(s_i\) is the transposition swapping the elements \(i\) and \(i+1\). The Coxeter relations are:
\begin{itemize}
    \item \(s_i^2 = e\) for all \(i\),
    \item \((s_i s_j)^2 = e\) if \(|i-j| > 1\),
    \item \((s_i s_{i+1})^3 = e\) for all \(1 \leq i \leq n-2\).
\end{itemize}
Thus, the corresponding Coxeter matrix \((m_{ij})\) has entries
\[
m_{ii} = 1, \quad m_{ij} = 2 \text{ if } |i-j| > 1, \quad m_{i,i+1} = 3,
\]
and the pair \((S_n, \Ss)\) is a Coxeter pair corresponding to the Coxeter system of type \(A_{n-1}\).
\end{example}

\begin{definition}[Mirrored Space]
\label{def:mirrored-space}
Given a Coxeter pair \((\Ww, \Ss)\) (i.e., \(\Ww\) is a Coxeter group and \(\Ss\) is a set of reflections generating \(\Ww\)), a space \(Z\) equipped with a family of closed subspaces \((Z_s)_{s \in \Ss}\) is called a \emph{mirror structure} on \(Z\) \cite[Chapter 5.1]{Davis-book}. The pair \((Z, (Z_s)_{s \in \Ss})\) is then called a \emph{mirrored space} over \(\Ss\).
\end{definition}
\begin{example}
Consider the Coxeter pair \((\Ww, \Ss)\) where \(\Ww\) is the symmetric group \(\S_n\) generated by the set \(\Ss = \{ s_i = (i, i+1) \mid 1 \leq i \leq n-1 \}\). 
Let \(Z = \R^n\), and for each reflection \(s_i \in \Ss\), define the mirror
\[
Z_{s_i} := \{ x \in \R^n \mid x_i = x_{i+1} \},
\]
which is the fixed-point set of the reflection \(s_i\).

Then the family \((Z_s)_{s \in \Ss}\) thus forms a mirror structure on \(Z\), and the pair \((Z, (Z_s)_{s \in \Ss})\) is a mirrored space over \(\Ss\).
Intuitively, each \(Z_{s_i}\) is a hyperplane reflecting points by swapping the coordinates \(x_i\) and \(x_{i+1}\), and together these mirrors generate the full symmetry group \(\S_n\).
\end{example}

Given a mirrored space \(Z\) with mirror family \((Z_s)_{s \in \Ss}\) over \(\Ss\), the following so-called \emph{Basic Construction} (see \cite[Chapter 5]{Davis-book}) produces a space \(\mathcal{U}(\Ww, Z)\) equipped with a natural \(\Ww\)-action. This construction is defined as follows.

\begin{definition}[The Basic Construction \cite{Koszul,Tits,Vinberg,Davis}]
\label{def:U}
We define
\begin{equation}
\label{eqn:U}
\mathcal{U}(\Ww,Z) := \Ww \times Z / \sim,
\end{equation}
where the topology on \(\Ww \times Z\) is the product topology, with \(\Ww\) given the discrete topology, and the equivalence relation \(\sim\) is defined by  
\[
(w_1, \u) \sim (w_2, \v) \iff \u = \v \text{ and } w_1^{-1} w_2 \in \Ww^{\Ss(\u)},
\]
with
\[
\Ss(\u) := \{ s \in \Ss \mid \u \in Z_s \},
\]
and \(\Ww^{\Ss(\u)}\) the subgroup of \(\Ww\) generated by \(\Ss(\u)\).

The group \(\Ww\) acts on \(\mathcal{U}(\Ww,Z)\) by
\(
w_1 \cdot [(w_2, \z)] := [(w_1 w_2, \z)],
\)
where \([(w, \z)]\) denotes the equivalence class of \((w, \z) \in \Ww \times Z\) under the relation \(\sim\).
\end{definition}

\begin{example}
Consider the Coxeter pair \((\Ww, \Ss)\) where \(\Ww = \S_3\), the symmetric group on three elements, generated by the adjacent transpositions
\[
\Ss = \{ s_1 = (1\,2), \quad s_2 = (2\,3) \}.
\]
Let \(Z = \mathbb{R}^3\), and define the mirror structure by
\[
Z_{s_1} = \{ \u = (u_1, u_2, u_3) \in \mathbb{R}^3 \mid u_1 = u_2 \}, \quad 
Z_{s_2} = \{ \u  \in \mathbb{R}^3 \mid u_2 = u_3 \}.
\]
The Basic Construction \(\mathcal{U}(\Ww, Z)\) is formed by taking pairs \((w, \u)\) with \(w \in \S_3\) and \(\u \in \mathbb{R}^3\), and identifying
\[
(w_1, \u) \sim (w_2, \u) \quad \text{if and only if} \quad w_1^{-1} w_2 \in \Ww^{\Ss(\u)},
\]
where \(\Ss(\u) = \{ s \in \Ss \mid \u \in Z_s \}\).
Intuitively, \(\mathcal{U}(\Ww,Z)\) assembles copies of \(Z\) indexed by elements of the group \(\S_3\), glued along the mirrors \(Z_{s_1}\) and \(Z_{s_2}\) according to the subgroup generated by reflections fixing the point \(\u\).
This construction provides a space with a natural \(\S_3\)-action reflecting the symmetry of the system.
\end{example}

The action of \(\Ww\) on \(\mathcal{U}(\Ww,Z)\) naturally induces an action on the homology groups \(\HH_*(\mathcal{U}(\Ww,Z))\), endowing them with the structure of a finite-dimensional \(\Ww\)-module. The decomposition of this module into irreducible or other \(\Ww\)-modules is studied in \cite{Davis-book} in the case where \(Z\) is a finite CW-complex. The following theorem is a fundamental result from \cite{Davis}.

\begin{theorem}[\cite{Davis-book}, Theorem 15.4.3]
\label{thm:Davis}
Let $(\Ww, \Ss) = (\mathfrak{S}_k, \Coxeter(k))$, and let $Z$ be a semi-algebraic mirrored space over $\Ss$ with closed and bounded mirrors $Z_s$ for each $s \in \Ss$. Then there is an isomorphism of $\mathfrak{S}_k$-modules:
\[
\HH_*(\mathcal{U}(\Ww, Z)) \cong_{\mathfrak{S}_k} \bigoplus_{T \subseteq \Ss} \HH_*(Z, Z^T) \otimes \Psi^{(k)}_T,
\]
where for each $T \subseteq \Ss$, $Z^T := \bigcup_{s \in T} Z_s$, and $\Psi^{(k)}_T$ is a Solomon module.
\end{theorem}

This result allows for an effective reconstruction of the Betti numbers of the full symmetric set from those of $Z$ and its relative intersections with mirror sets $Z^T$. In particular, a paper by Basu and Riener  used the decomposition induced by the mirror structure to compute the first $\ell$ Betti numbers of such sets.

\begin{theorem}
\label{thm:alg}
Let $\D$ be an ordered domain contained in a real closed field $\R$, and let $\ell, d \geq 0$.
There exists an algorithm which takes as input a finite set 
$\mathcal{P} \subset \D[X_1,\ldots, X_k]^{\mathfrak{S}_k}_{\leq d}$, and a $\mathcal{P}$-formula $\Phi$,
and computes
the tuple of integers
\[
(b_0(\RR(\Phi)), \ldots, b_\ell(\RR(\Phi))).
\] 

The complexity of the algorithm, measured by the number of arithmetic operations in 
$\D$, is bounded by
$(s k d)^{2^{O(d+\ell)}}$.

If $\D = \Z$, and the bit-sizes of the coefficients of the input are bounded by $\tau$, then the bit-complexity
of the algorithm is bounded by 
\[
(\tau s k d)^{2^{O(d+\ell)}}.
\]
\end{theorem}

\paragraph{Sketch of the Algorithm.}
The algorithm exploits the structure provided by the mirror space decomposition. First, it restricts the input symmetric set \( S = \RR(\Phi) \) to the canonical Weyl chamber, producing a semi-algebraic set \( Z = S \cap \mathcal{W} \), where \( \mathcal{W} \) denotes the Weyl chamber. The homology of \( Z \) is computed using known algorithms for equivariant Betti number computation of general semi-algebraic sets, as outlined above.

To recover the Betti numbers of \( S \), the algorithm computes relative homology groups \( \HH_*(Z, Z^T) \) for subsets \( T \subseteq \Ss \) corresponding to the walls of the chamber. The computational complexity is governed by the number and structure of these subsets. Due to representation-theoretic constraints, specifically, the structure of the Solomon modules \( \Psi_T^{(k)} \), only subsets \( T \) with \( |T| \leq O(d+\ell) \) contribute to the first \( \ell \) Betti numbers. These subsets are of relatively low dimension, and their number can be bounded polynomially.

This approach yields a complexity bound that is singly exponential in \( d + \ell \), and polynomial in the remaining parameters. Thus, symmetry enables a dramatic reduction in complexity compared to the general semi-algebraic case, where even computing \( b_0 \) can be intractable.

Furthermore, building on the mirror space structure, Riener, Schabert, and Vu extended their results on equivariant connectivity to develop an efficient algorithm for deciding general connectivity in symmetric semi-algebraic sets \cite{deciding}.
\subsubsection{Algorithm for Testing the Connectivity of any  Points.}
Let \( S \subset \mathbb{R}^n \) be a closed and bounded semi-algebraic set invariant under the action of the symmetric group \(\mathfrak{S}_n\). Suppose we are given two points \(\u, \v \in S\). In general, deciding whether these two points lie in the same connected component of \( S \) is a challenging problem. However, for symmetric sets, the additional structure imposed by group invariance can be exploited algorithmically through the mirror space framework.

The key idea is to reduce the connectivity problem in \( S \) to a problem within the intersection \( S_c := S \cap \mathcal{W}_c \), where \( \mathcal{W}_c \) denotes the canonical Weyl chamber. By analyzing how \( S \) intersects the walls of this chamber, one can determine connectivity in the full space by leveraging the symmetries of the group action.

This reduction relies on a homotopy-theoretic characterization of the connectedness of the mirrored space \(\mathcal{U}(\mathfrak{S}_n, S_c)\), which is semi-algebraically homeomorphic to \( S \). A central result states that \(\mathcal{U}(\mathfrak{S}_n, S_c)\) is connected if and only if the relative homology groups \( H_0(S_c, S_c^T) \) vanish for all non-empty subsets \( T \) of the set of Coxeter generators. Here, \( S_c^T \) denotes the union of the intersections of \( S_c \) with the chamber walls indexed by \( T \).

Moreover, if a point \(\v \in S\) lies in the \(\mathfrak{S}_n\)-orbit of a point \(\v' \in \mathcal{W}_c\) (i.e., \(\v = \sigma(\v')\) for some permutation \(\sigma\)), then \(\u\) and \(\v\) lie in the same connected component of \( S \) if and only if:

\begin{enumerate}
    \item \(\u\) and \(\v'\) lie in the same connected component of \( S_c \); and
    \item the permutation \(\sigma\) belongs to the subgroup \( G(S_c) \) generated by transpositions \((i,i+1)\) such that \( S_c \cap \mathcal{W}_c^{(i,i+1)} \) is connected.
\end{enumerate}
This leads to an efficient algorithm with the following structure:
\begin{itemize}
    \item Project the points \(\u\) and \(\v\) to \(\mathcal{W}_c\), obtaining \(\u\) and \(\v'\).
    \item Use roadmap techniques to decide whether \(\u\) and \(\v'\) lie in the same connected component of \( S_c \).
    \item Decompose the permutation \(\sigma\) mapping \(\v'\) to \(\v\) into adjacent transpositions.
    \item For each transposition \((i, i+1)\) in this decomposition, test whether \(\u\) is connected to the corresponding wall \(\mathcal{W}_c^{(i, i+1)}\) within \( S_c \).
\end{itemize}

Each of these wall tests reduces to lower-dimensional connectivity problems by exploiting the symmetric structure. The complexity of the algorithm is controlled by the size of a combinatorial object known as the \emph{alternate odd compositions}, and the overall procedure enjoys favorable complexity bounds. This algorithm exploits the symmetry of \( S \) and is asymptotically more efficient than general-purpose methods for connectivity testing in semi-algebraic sets.

\begin{theorem}[Efficient Connectivity Test \cite{deciding}]
Let \( S \subset \mathbb{R}^n \) be a closed and bounded semi-algebraic set defined by \( s \) symmetric polynomials of degree at most \( d \). Let \(\u, \v \in S\) be two points. There exists an algorithm that decides whether \(\u\) and \(\v\) belong to the same connected component of \( S \).

The complexity of this algorithm, measured by the number of arithmetic operations over \(\mathbb{Q}\), is bounded by
\[
O\left( s^d \cdot n^{d^2} \cdot d^{O(d^2)} \right).
\]
\end{theorem}
\begin{remark}
While the results in this section focus on symmetric groups, the approach may generalize to semi-algebraic sets invariant under other Coxeter group actions, provided that an appropriate mirrored space construction is available.
\end{remark}
In summary, symmetry dramatically reduces the complexity of deciding connectivity in semi-algebraic sets. Instead of relying solely on general-purpose roadmap constructions, we leverage the mirror space decomposition to localize computations to the fundamental domain and selected strata. This approach not only offers conceptual clarity but also yields significant practical gains in computational efficiency.
\subsubsection{Final Discussions.}

Computationaly analysing   properties of real algebraic varieties and semi-algebraic sets remains a fundamental algorithmic challenge, particularly when the sets in question exhibit high complexity or dimension. In this chapter, we have demonstrated how symmetry, especially under the action of the symmetric group, provides a powerful algebraic lens that leads to both conceptual simplifications and concrete computational gains. By exploiting the structure of symmetric polynomials and orbits, we are able to design algorithms that reduce intrinsic complexity and circumvent traditional bottlenecks in symbolic computation.

In particular, we have illustrated how recent works on    tasks such as emptiness testing, sampling connected components, and computing the Euler–Poincaré characteristic have have provided more efficient  approachess by  taken symmetry into account. These works not only gave more practical algorithms but also opened avenues for theoretical refinement. The degree principle and related reduction techniques exemplify how structural properties translate directly into algorithmic advantages.

Looking forward, several interesting  research directions are still open. One is the generalization of the methods presented here to semi-algebraic sets invariant under other finite or compact group actions, extending beyond symmetric groups.

Finally, there remains a rich landscape of open problems in understanding the asymptotic behavior of topological invariants for families of symmetric varieties, particularly in connection to representation theory and homological stability. We anticipate that the ideas surveyed in this chapter will serve as a foundation for further advances in the algorithmic understanding of symmetry in real algebraic geometry.


\begin{thebibliography}{4}
\bibitem{av}
J.G. Acevedo, M. Velasco: 
Test sets for nonnegativity of polynomials invariant under a finite reflection group, Journal of  Pure and Applied  Algebra, 220 (2016)

\bibitem{Blekherman}
Acevedo, J., Blekherman, G., Debus, S., Riener, C.:
The wonderful geometry of the Vandermonde map.
\textit{Foundations of Computational Mathematics}, (to appear).



\bibitem{Alonso1996} Alonso, M.-E., Becker, E., Roy, M.-F., Wörmann, T.: Zeros, multiplicities, and idempotents for zero-dimensional systems. In: Algorithms in Algebraic Geometry and Applications, pp. 1–15. Springer (1996)

\bibitem{arnold1986hyperbolic}
Arnold, V. I.:
Hyperbolic polynomials and Vandermonde mappings.
\textit{Funktsional'nyi Analiz i ego Prilozheniya}, 20(2):52--53 (1986)



\bibitem{BPR96}Basu, S., Pollack, R., Roy, M.-F.: On the combinatorial and algebraic complexity of quantifier elimination. Journal of the ACM, 43(6), 1002–1045 (1996)

\bibitem{basu2000computing}
Basu, S., Pollack, R., Roy, M.-F.:
Computing roadmaps of semi-algebraic sets on a variety.
Journal of the American Mathematical Society, 13(1), 55–82 (2000)

\bibitem{BPR06} Basu, S., Pollack, R., Roy, M.-F.:  Algorithms in Real Algebraic Geometry. Algorithms and Computation in Mathematics, vol. 10. Springer (2006)

\bibitem{basu2014divide}
Basu, S., Roy, M.-F.:
Divide and conquer roadmap for algebraic sets.
\textit{Discrete \& Computational Geometry}, \textbf{52}:278--343 (2014)


\bibitem{basu2014baby}
Basu, S., Roy, M.-F., Safey El Din, M., Schost, É.:
A baby step–giant step roadmap algorithm for general algebraic sets.
\textit{Foundations of Computational Mathematics}, \textbf{14}:1117--1172 (2014)
\bibitem{survey}
Basu, S.: 
\newblock Algorithms in real algebraic geometry: a survey.
\newblock {\em arXiv preprint arXiv:1409.1534} (2014)

\bibitem{BasuRiener2013}
Basu, S., Riener, C.:
\newblock Bounding the equivariant Betti numbers of symmetric semi-algebraic sets.
\newblock {\em Advances in Mathematics}, 305:803--855 (2017)

\bibitem{Basu2018}
Basu, S., Riener, C.:
\newblock Efficient algorithms for computing the Euler--Poincar{\'e} characteristic of symmetric semi-algebraic sets.
\newblock In {\em Ordered Algebraic Structures and Related Topics: International Conference on Ordered Algebraic Structures and Related Topics, October 12--16, 2015, Centre International de Rencontres Math{\'e}matiques (CIRM), Luminy, France}, volume 697, pages 53--81. American Mathematical Society (2017)

\bibitem{basu2018equivariant}
Basu, S.,  Riener, C.:
On the equivariant Betti numbers of symmetric definable sets: vanishing, bounds and algorithms.
\textit{Selecta Mathematica}, 24:3241--3281 (2018)


\bibitem{basu2022vandermonde}
Basu, S., Riener, C.:
Vandermonde varieties, mirrored spaces, and the cohomology of symmetric semi-algebraic sets.
Foundations of Computational Mathematics, 22(5), 1395–1462 (2022)


\bibitem{BlaserJindal18} Bläser, M., Jindal, G.: On the complexity of symmetric polynomials. In 10th Innovations in Theoretical Computer Science Conference (ITCS 2019) (Leibniz International Proceedings in Informatics (LIPIcs), Vol. 124), A. Blum (Ed.). Schloss Dagstuhl–Leibniz-Zentrum fuer Informatik, Dagstuhl, Germany, 47:1–47:14 (2019)

\bibitem{blekherman2006}
Blekherman, G.: 
\newblock There are significantly more nonnegative polynomials than sums of squares.
\newblock {\em Israel Journal of Mathematics}, 153:355--380 (2006)

\bibitem{blekherman2012nonnegative}
Blekherman, G. and Riener, C.
\newblock Symmetric nonnegative forms and sums of squares.
\newblock \emph{Discrete and Computational Geometry}, 57(4):854--886 (2017)

\bibitem{blekherman2021symmetric}
Blekherman, G. and Riener, C.:
\newblock Symmetric non-negative forms and sums of squares.
\newblock {\em Discrete \& Computational Geometry}, 65:764--799 (2021)

\bibitem{BlumCuckerShubSmale2012}
Blum, L., Cucker, F., Shub, M., Smale, S.: Complexity and Real Computation. Springer Science \& Business Media (2012)

\bibitem{brocker1998symmetric}
Br{\"o}cker, L.:
\newblock On symmetric semialgebraic sets and orbit spaces.
\newblock {\em Banach Center Publications}, 44(1):37--50 (1998)

\bibitem{BrDa07}
Brown, C. W., Davenport, J. H.: The complexity of quantifier elimination and cylindrical algebraic decomposition. In: Proceedings of the 2007 International Symposium on Symbolic and Algebraic Computation, pp. 54–60 (2007)

\bibitem{Canny}
Canny, J.: The Complexity of Robot Motion Planning. MIT Press, Cambridge, MA (1987)

\bibitem{canny1993computing}
Canny, J.:
Computing roadmaps of general semi-algebraic sets.
\emph{The Computer Journal}, \textbf{36}(5), 504--514 (1993)

\bibitem{canny1992finding}
Canny, J., Grigor’ev, Yu, D., Vorobjov, N. N.:
Finding connected components of a semialgebraic set in subexponential time.
\emph{Applicable Algebra in Engineering, Communication and Computing}, \textbf{2}(4), 217--238 (1992)

\bibitem{capco2023positive}
Capco, M., Safey El Din, M., Schicho, J.:
Robots, computer algebra and eight connected components.
In: \textit{Proceedings of the 45th International Symposium on Symbolic and Algebraic Computation (ISSAC 2020)}, pp. 62–69. ACM (2020)


\bibitem{capco2020robots}
Capco, M., Safey El Din, M., Schicho, J.:
Positive dimensional parametric polynomial systems, connectivity queries and applications in robotics.
\textit{Journal of Symbolic Computation}, \textbf{115}, 320–345 (2023)


\bibitem{Cox1997} Cox, D., Little, J., O'shea, D.,  Sweedler, M.: Ideals, varieties, and algorithms, vol. 3. Springer (1997)

\bibitem{Collins} Collins, G. E.: Quantifier elimination for real closed fields by cylindrical algebraic
decomposition. Lecture notes in computer science, 33:515–532 (1975)

\bibitem{chaugule2023schur} Chaugule, P., Kumar, M., Limaye,  N., Mohapatra,  C. K., She, A.,  Srinivasan, S.: Schur polynomials do not have small formulas if the determinant does not. Computational Complexity 32, 1 (2023)

\bibitem{Chen20}
Chen, C., Davenport, J. H., May, J. P., Moreno Maza, M., Xia, B., Xiao, R.: Triangular Decomposition of Semi-Algebraic Systems. In: Proceedings of the 2010 International Symposium on Symbolic and Algebraic Computation, pp. 187--194. Springer (2010)

\bibitem{Chen09}
Chen, C., Moreno Maza, M., Xia, B., Yang, L.: Computing Cylindrical Algebraic Decomposition via Triangular Decomposition. In: Proceedings of the 2009 International Symposium on Symbolic and Algebraic Computation, pp. 95--102. Springer (2009)

\bibitem{choi1995sums}
Choi,  M.-D., Lam, T.~Y., Reznick, B.:
\newblock Sums of squares of real polynomials.
\newblock In {\em Proceedings of Symposia in Pure Mathematics}, volume 58, pages 103--126. American Mathematical Society (1995)



\bibitem{Dav88} Davenport, J. H., Heintz, J.: Real quantifier elimination is doubly exponential. Journal of Symbolic Computation, 5(1-2), 29–35 (1988)

\bibitem{Davis}
Davis, M. W.:
Groups generated by reflections and aspherical manifolds not covered by Euclidean space.
Ann. of Math. (2), 117(2):293--324 (1983)

\bibitem{Davis-book}
Davis, M. W.:
\textit{The Geometry and Topology of Coxeter Groups}.
London Mathematical Society Lecture Note Series, vol. 32,
Princeton University Press, Princeton (2008)

\bibitem{debus2020reflection}
Debus, S., Riener, C.: 
\newblock Reflection groups and cones of sums of squares.
\newblock {\em Journal of Symbolic Computation}, 119:112--144 (2023)

\bibitem{debus2}
Debus, S, Moustrou, P, Riener, C. Verdure, H. \newblock The poset of Specht ideals for hyperoctahedral groups. Algebraic Combinatorics, Volume 6  no. 6, pp. 1593-1619 (2023)

\bibitem{debus3}
Debus, S.,  Gottwald, K. K. \newblock Posets for Specht ideals of essential real reflection groups. arXiv preprint arXiv:2506.15335 (2025)

\bibitem{derksen2015} Derksen, H., Gregor, K.: Computational Invariant Theory. Springer (2015)

\bibitem{Dines19} Dines, L. L.: Systems of linear inequalities. Annals of Mathematics, pages 191–199 (1919)

\bibitem{Dries}
van~den Dries, L.: \emph{Tame topology and o-minimal structures}, London
  Mathematical Society Lecture Note Series, vol. 248, Cambridge University
  Press, Cambridge (1998)
  
\bibitem{eisenbud2013commutative}
Eisenbud, D.:
Commutative Algebra: with a View Toward Algebraic Geometry.
Graduate Texts in Mathematics, vol. 150. Springer, New York (2013)
\bibitem{ElGiSc20}
Elliott, J., Giesbrecht, M., Schost, É.:
On the bit complexity of finding points in connected components of a smooth real hypersurface.
In: Proceedings of the International Symposium on Symbolic and Algebraic Computation (ISSAC 2020), pp. 170--177. ACM, New York (2020)


\bibitem{EnDa20}
England, M., Bradford, R., Davenport, J. H.: Cylindrical Algebraic Decomposition with Equational Constraints. J. Symb. Comput. 100, 38--71 (2020)

\bibitem{faugere2020computing}
Faugère, J.-C., Labahn, G., Safey El Din, M., Schost, É., Vu, T. X.:
Computing critical points for invariant algebraic systems.
Journal of Symbolic Computation, 116, 365--399 (2023)


\bibitem{friedl}
Friedl, T., Riener, C. Sanyal, R.:
\newblock Reflection groups, reflection arrangements, and invariant real varieties.{Proceedings of the American Mathematical Society}, 146(3) (2018)


\bibitem{Fourier26} Fourier, J. B. J.: Solution d’une question particuliere du calcul des inégalités.
Nouveau Bulletin des Sciences par la Société philomatique de Paris, 99:100 (1826)

\bibitem{gatermann2004symmetry}
Gatermann, K., Parrilo, P.~A.:
\newblock Symmetry groups, semidefinite programs, and sums of squares.
\newblock {\em Journal of Pure and Applied Algebra}, 192(1--3):95--128 (2004)

\bibitem{gathen2003multivariate} von zur Gathen, J., Gutierrez, J., Rubio, R.: Multivariate polynomial decomposition. Applicable Algebra in Engineering, Communication
and Computing 14, 1, 11–3 (2003)


\bibitem{GathenGerhard2003} von zur Gathen, J., Gerhard, J.: Modern Computer Algebra. 2nd edn. Cambridge University Press (2003)


\bibitem{gaudry2006evaluation} Gaudry, P.,  Schost,  É., Thiéry,  N. M: Evaluation properties of symmetric polynomials. International Journal of Algebra and Compu- tation 16, 03, 505–523  (2006)

\bibitem{Gianni1989} Gianni, P., Mora, T.: Algebraic solution of systems of polynomial equations using Gröbner bases. In: AAECC, vol. 356, pp. 247–257. Springer (1989)

\bibitem{goel2016choi}
Goel, C., Kuhlmann, S., Reznick, B.: 
\newblock On the Choi--Lam analogue of Hilbert's 1888 theorem for symmetric forms.
\newblock {\em Linear Algebra and its Applications}, 496:114--120 (2016)

\bibitem{multi} Görlach, P., Riener, C.,  Weißer, T.: \newblock Deciding positivity of multisymmetric polynomials. Journal of Symbolic Computation, 74, 603-616 (2016)

\bibitem{givental1987moments}
Givental, A. B.:
Moments of random variables and the equivariant Morse lemma.
\textit{Russian Mathematical Surveys}, 42(2):275--276 (1987)



\bibitem{giusti1995polynomial} Giusti, M., Heintz, J., Morais, J.E., Pardo, L.M.: When polynomial equation systems can be “solved” fast?. In: AAECC-11, pp. 205–231. Springer (1995)

\bibitem{Giustifast1995} Giusti, M., Heintz, J., Morais, J.-E., Pardo, L.-M.: When polynomial equation systems can be solved fast? In: AAECC-11, vol. 948, pp. 205–231. Springer (1995)

\bibitem{giusti1997lower} Giusti, M., Heintz, J., Hägele, K., Morais, J.E., Pardo, L.M., Montana, J.L.: Lower bounds for Diophantine approximations. J. Pure Appl. Algebra 117, 277–317 (1997)

\bibitem{giusti1998straight} Giusti, M., Heintz, J., Morais, J.E., Morgenstern, J., Pardo, L.M.: Straight-line programs in geometric elimination theory. J. Pure Appl. Algebra 124(1–3), 101–146 (1998)

\bibitem{GiustiGrob2001} Giusti, M., Lecerf, G., Salvy, B.: A Gröbner-free alternative for polynomial system solving. J. Complexity 17(1), 154–211 (2001)

 \bibitem{GV88}
Grigoriev, D., Vorobjov, N.: Solving systems of polynomial inequalities in subexponential time. Journal of Symbolic Computation, 5, 37–64 (1988)

 \bibitem{grigor1992counting}
Grigor’ev, D. Yu., Vorobjov, N.:
Counting connected components of a semialgebraic set in subexponential time.
\emph{Computational Complexity}, \textbf{2}, 133--186 (1992)

 \bibitem{gournay1993construction}
Gournay, L., Risler, J.-J.:
Construction of roadmaps in semi-algebraic sets.
Applicable Algebra in Engineering, Communication and Computing, 4(4):239–252 (1993)


\bibitem{HRS93}
Heintz, J., Roy, M.-F., Solernò, P.: On the theoretical and practical complexity of the existential theory of the reals. The Computer Journal, 36(5), 427–431 (1993)


\bibitem{Hartshorne2013} Hartshorne, R.: Algebraic Geometry, volume 52. Springer Science \& Business Media (2013)

\bibitem{heintz1981absolute} Heintz, J., Sieveking, M.: Absolute primality of polynomials is decidable in random polynomial time in the number of variables. In: ICALP, pp. 16–28. Springer (1981)

\bibitem{heintz1994single}
Heintz, J., Roy, M.-F., Solernó, P.:
Single exponential path finding in semi-algebraic sets, part II: The general case.
In: Algebraic Geometry and its Applications: Collections of Papers from Shreeram S. Abhyankar’s 60th Birthday Conference, 
pp. 449–465. Springer (1994)

\bibitem{Hong:92d}
Hong, H.: Heuristic search strategies for cylindrical algebraic decomposition. In: Proceedings of Artificial Intelligence and Symbolic Mathematical Computing, Lecture Notes in Computer Science, vol. 737, pp. 152–165. Springer (1992)

\bibitem{iraji2014nuroa}
Iraji, R., Chitsaz, H.:
Nuroa: A numerical roadmap algorithm.
In: \textit{Proceedings of the 53rd IEEE Conference on Decision and Control}, pp. 5359–5366. IEEE (2014)

\bibitem{jungnickel2005graphs}
Jungnickel, D.:
\textit{Graphs, Networks and Algorithms}, Volume 3.
Springer (2005)

\bibitem{kaltofen1988greatest} Kaltofen, E.: Greatest common divisors of polynomials given by straight-line programs. J. ACM 35(1), 231–264 (1988)

\bibitem{kaltofen1989factorization} Kaltofen, E.: Factorization of polynomials given by straight-line programs. Adv. Comput. Res. 5, 375–412 (1989)


\bibitem{kostov1989geometric}
Kostov, V. P.:
On the geometric properties of Vandermonde’s mapping and on the problem of moments.
\textit{Proceedings of the Royal Society of Edinburgh Section A: Mathematics}, 112(3--4):203--211 (1989)

\bibitem{Koszul}
Koszul, J.-L.:
\textit{Lectures on Groups of Transformations}.
Notes by R. R. Simha and R. Sridharan,
Tata Institute of Fundamental Research Lectures on Mathematics, No. 32,
Tata Institute of Fundamental Research, Bombay (1965)

\bibitem{Kronecker1882} Kronecker, L.: Grundzüge einer arithmetischen Theorie der algebraischen Grössen. J. Reine Angew. Math. 92, 1–122 (1882)

\bibitem{labahn2021homotopy}
Labahn, G., Safey El Din, M., Schost, É., Vu, T. X.:
Homotopy techniques for solving sparse column support determinantal polynomial systems.
Journal of Complexity, 66, 101557 (2021)

\bibitem{Labahn2023}
Labahn, G., Riener, C., Safey El Din, M., Schost, {\'E}., Vu, T.~X.:
\newblock Faster real root decision algorithm for symmetric polynomials.
\newblock In \emph{Proceedings of the 2023 International Symposium on Symbolic and Algebraic Computation}, pages 452--460 (2023)


\bibitem{lasserre2001global}
Lasserre, J.~B.:
\newblock Global optimization with polynomials and the problem of moments.
\newblock {\em SIAM Journal on Optimization}, 11(3):796--817 (2001)

\bibitem{Macaulay1916} Macaulay, F.S.: The Algebraic Theory of Modular Systems. Cambridge University Press (1916)

\bibitem{Macdonald1995}
Macdonald, I. G.:
\textit{Symmetric Functions and Hall Polynomials}, 2nd Edition.
Oxford University Press (1995)

\bibitem{McCallum}
McCallum, S.: An improved projection operator for Cylindrical Algebraic Decomposition. PhD thesis, University of Wisconsin-Madison (1984)

\bibitem{McCallumeq}
McCallum, S.: On projection in CAD-based quantifier elimination with equational constraint. In: Proceedings of the 1999 International Symposium on Symbolic and Algebraic Computation (ISSAC), pp. 145–149. ACM (1999)

\bibitem{meguerditchian1992theorem}
Meguerditchian, I.:
A theorem on the escape from the space of hyperbolic polynomials.
\textit{Mathematische Zeitschrift}, 211:449--460 (1992)
\bibitem{Moustrou}
Moustrou, P., Riener, C., Verdure, H.:
\newblock Symmetric ideals, Specht polynomials and solutions to symmetric systems of equations.
\newblock {\em Journal of Symbolic Computation}, 107:106--121 (2021)

\bibitem{parrilo2000thesis}
Parrilo, P.~A.:
\newblock Structured semidefinite programs and semialgebraic geometry methods in robustness and optimization.
\newblock PhD thesis, California Institute of Technology (2000)

\bibitem{prebet2024computing}
Pr\'ebet, R., Safey El Din, M., Schost, \'E.:
Computing roadmaps in unbounded smooth real algebraic sets I: connectivity results.
\textit{Journal of Symbolic Computation}, \textbf{120}:102234 (2024)

\bibitem{PrebetSafeySchost2024b}
Pr\'ebet, R., Safey El Din, M., Schost, \'E.:
Computing roadmaps in unbounded smooth real algebraic sets II: algorithm and complexity.
\textit{arXiv preprint} arXiv:2402.03111 (2024)

\bibitem{prestel2001positive}
Prestel. A, Delzell, C.~N.: 
\newblock {\em Positive Polynomials: From Hilbert’s 17th Problem to Real Algebra}.
\newblock Springer (2001)

\bibitem{Ren}
Renegar, J.: On the computational complexity and geometry of the first order theory of the reals. Journal of Symbolic Computation, 13(3), 255–352 (1992)


\bibitem{Riener2012} Riener, C.: On the degree and half-degree principle for symmetric polynomials. Journal of Pure and Applied Algebra, 216(4), 850-856 (2012)

\bibitem{riener2013exploiting}
Riener, C., Theobald, T., Andr{\'e}n, L.~J.,  Lasserre, J.~B.: 
\newblock Exploiting symmetries in SDP-relaxations for polynomial optimization.
\newblock {\em Mathematics of Operations Research}, 38(1):122--141 (2013)

\bibitem{Rouillier1999} Rouillier, F.: Solving zero-dimensional systems through the Rational Univariate Representation. Appl. Algebra Eng. Commun. Comput. 9(5), 433–461 (1999)

\bibitem{rie1}
Riener, C. Symmetries in semidefinite and polynomial optimization: relaxations, combinatorics, and the degree principle. Dissertation Goethe University Frankfurt (2011)

\bibitem{Riener2018} Riener, C., Safey El Din, M.: Real root finding for equivariant semi-algebraic systems. In: Proceedings of the 2018 ACM International Symposium on Symbolic and Algebraic Computation, pp. 335–342 (2018)

\bibitem{Issac}
Riener, C.,  Schabert, R.,  Vu, T.~X.: 
\newblock Connectivity in symmetric semi-algebraic sets.
\newblock In {\em Proceedings of the 2024 International Symposium on Symbolic and Algebraic Computation}, pages 162--169 (2024)


\bibitem{deciding}
Riener, C., Schabert, R.,  and Vu, T. X.:
\newblock Deciding connectivity in symmetric semi-algebraic sets.
\newblock {\em arXiv preprint arXiv:2503.12275} (2025)


\bibitem{slices}
Riener, C. Schabert, R.:
\newblock
Linear slices of hyperbolic polynomials and positivity of symmetric polynomial functions,
Journal of Pure and Applied Algebra,
228 (5), 107552 (2025)

\bibitem{rouillier1999solving}
Rouillier, F.:
Solving zero-dimensional systems through the rational univariate representation.
\textit{Applicable Algebra in Engineering, Communication and Computing}, 9(5):433--461 (1999)

\bibitem{safey2011baby}
Safey El Din, M., Schost, É.:
A baby steps/giant steps probabilistic algorithm for computing roadmaps in smooth bounded real hypersurface.
\textit{Discrete \& Computational Geometry}, \textbf{45}(1):181--220 (2011)


\bibitem{din2017nearly}
Safey El Din, M. and Schost, É.:
A nearly optimal algorithm for deciding connectivity queries in smooth and bounded real algebraic sets.
\textit{Journal of the ACM (JACM)}, \textbf{63}(6):1--37  (2017)


\bibitem{sakkalis2005note}
Sakkalis, T.:
A note on proper polynomial maps.
\emph{Communications in Algebra}, \textbf{33}(9), 3359--3365 (2005)

\bibitem{Schwartz1980} Schwartz, J.T.: Fast probabilistic algorithms for verification of polynomial identities. J. ACM 27(4), 701–717 (1980)

\bibitem{schwartz1983piano}
Schwartz, J.T., Sharir, M.:
On the “piano movers” problem. II. General techniques for computing topological properties of real algebraic manifolds.
Advances in Applied Mathematics, 4(3):298–351 (1983)
\bibitem{stengle1974nullstellensatz}
Stengle, G.:
\newblock A Nullstellensatz and a Positivstellensatz in semialgebraic geometry.
\newblock {\em Mathematische Annalen}, 207:87--97 (1974)

\bibitem{sweedler1993using} Sweedler, M.: Using Gröebner bases to determine the algebraic and transcendental nature of field extensions: return of the killer tag variables. Applied Algebra, Algebraic Algorithms and Error-Correcting Codes: 10th International Symposium, AAECC-10 San Juan de Puerto Rico, Puerto Rico, May 10–14, 1993 Proceedings 10. Springer, Berlin, Heidelberg, 66–75 (1993)



\bibitem{Tarski} Tarski, A.: A Decision Method for Elementary Algebra and Geometry. The Rand Corporation, Santa Monica, Calif. (1948)

\bibitem{Timofte2003} Timofte, V.: On the positivity of symmetric polynomial functions.: Part I: General results. Journal of Mathematical Analysis and Applications, 284(1), 174-190  (2003) 

\bibitem{shafarevich1994basic}
Shafarevich, I. R., Reid, M.: Basic Algebraic Geometry, Volume 2: Schemes and Complex Manifolds. Springer, Berlin (1994)

\bibitem{Solomon1968}
Solomon, L.:
A decomposition of the group algebra of a finite Coxeter group.
\textit{Journal of Algebra}, 9:220--239 (1968)


\bibitem{Strz06}
Strzeboński, A. W.: Cylindrical algebraic decomposition using validated numerics. Journal of Symbolic Computation, 41(9), 1021–1038 (2006)

\bibitem{Strz14}
Strzeboński, A. W.: Cylindrical algebraic decomposition using local projections. In: Proceedings of the 39th International Symposium on Symbolic and Algebraic Computation, pp. 389–396 (2014)

\bibitem{Tits}
Tits, J.:
On buildings and their applications.
In: Proceedings of the International Congress of Mathematicians (Vancouver, B.C., 1974), Vol. 1,
pp. 209--220 (1975)

\bibitem{vandenberghe1996semidefinite}
Vandenberghe, L., Boyd, S.: 
\newblock Semidefinite programming.
\newblock {\em SIAM Review}, 38(1):49--95 (1996)

\bibitem{Vinberg}
Vinberg, E. B.:
Discrete linear groups that are generated by reflections.
Izv. Akad. Nauk SSSR Ser. Mat., 35:1072--1112 (1971)


\bibitem{Vu2022}
Vu, T.~X.:
\newblock Computing critical points for algebraic systems defined by hyperoctahedral invariant polynomials.
\newblock In \emph{Proceedings of the 2022 International Symposium on Symbolic and Algebraic Computation}, pages 167--175 (2022)


\bibitem{vu2025} Vu, T. X.: Computing Polynomial Representation in Subrings of Multivariate Polynomial Rings. arXiv preprint arXiv:2504.21708 (2025)
\end{thebibliography}
\end{document}